\documentclass[10pt]{article}

\usepackage[margin=1in]{geometry} 
\usepackage{pxfonts} 

\usepackage[english]{babel} 
\usepackage{amsmath, amsfonts, amsthm, amssymb} 
\usepackage[scr=boondoxo,cal=euler]{mathalfa} 
\usepackage{tikz-cd} 
\usepackage{indentfirst} 
\usepackage{fancyhdr} 
\usepackage{graphicx} 
\usepackage{enumitem} 
\usepackage{microtype} 
\usepackage{pdfpages} 
\usepackage{centernot} 
\usepackage{ragged2e} 
\usepackage{pinlabel}
\usepackage{caption}
\usepackage[hidelinks]{hyperref} 

\allowdisplaybreaks 
\usetikzlibrary{decorations.pathmorphing}

\theoremstyle{plain}
\newtheorem{thm}{Theorem}[section]
\newtheorem*{thm*}{Theorem}
\newtheorem{lem}[thm]{Lemma}
\newtheorem*{lem*}{Lemma}

\newtheorem*{cor*}{Corollary}
\newtheorem{prop}[thm]{Proposition}
\newtheorem*{prop*}{Proposition}

\newtheorem*{conj*}{Conjecture}

\newtheorem*{ques*}{Question}

\theoremstyle{definition}
\newtheorem{df}[thm]{Definition}
\newtheorem*{df*}{Definition}
\newtheorem*{dfs*}{Definitions}

\newtheorem*{exercise*}{Exercise}

\theoremstyle{remark}
\newtheorem{rem}[thm]{Remark}
\newtheorem*{rem*}{Remark}

\newtheorem*{example}{Example}
\newtheorem*{examples}{Examples}

\usepackage{etoolbox}
\patchcmd{\thmhead}{(#3)}{#3}{}{}
\makeatletter
\g@addto@macro\bfseries{\boldmath}
\makeatother

\newcommand{\cl}[1]{\mathcal{#1}}
\newcommand{\fk}[1]{\mathfrak{#1}}

\newcommand{\sr}[1]{\mathscr{#1}}
\newcommand{\Z}{\mathbf{Z}} 
\newcommand{\Q}{\mathbf{Q}} 
\newcommand{\R}{\mathbf{R}} 
\newcommand{\C}{\mathbf{C}} 
\newcommand{\A}{\mathbf{A}} 
\newcommand{\G}{\mathbf{G}} 
\newcommand{\F}{\mathbf{F}}

\newcommand{\x}{\times}

\newcommand{\emp}{\emptyset}

\renewcommand{\sl}{\fk{sl}}

\newcommand{\qbinom}[2]{\genfrac{[}{]}{0pt}{}{#1}{#2}}
\newcommand{\ol}[1]{\overline{#1}}

\newcommand{\rKh}{\smash{\overline{\Kh}}}
\newcommand{\rKR}{\smash{\overline{\KR}}}

\newcommand{\B}{\mathrm{B}}
\newcommand{\E}{\mathrm{E}}

\DeclareMathOperator{\Isharp}{I^\sharp}
\DeclareMathOperator{\Inat}{I^\natural}

\DeclareMathOperator{\rank}{rank}
\DeclareMathOperator{\rk}{rk}

\DeclareMathOperator{\sign}{sgn}

\DeclareMathOperator{\Hom}{Hom}
\DeclareMathOperator{\Tor}{Tor}

\DeclareMathOperator{\Id}{Id}

\DeclareMathOperator{\SU}{SU}
\DeclareMathOperator{\U}{U}

\DeclareMathOperator{\BU}{BU}

\DeclareMathOperator{\Kh}{Kh}

\DeclareMathOperator{\KR}{KR}
\DeclareMathOperator{\KRC}{KRC}

\DeclareMathOperator{\Sym}{Sym}

\usepackage{mathtools}
\usepackage{ytableau}
\renewcommand{\coloneq}{\coloneqq}

\newcommand{\Thetak}{
	\quad\:\:\qquad\begin{gathered}
			\vspace{-3pt}
			\centering			
			\labellist
			\pinlabel {\small${b-k}$} at 26 81
			\pinlabel {\small$k$} at 59 44
			\pinlabel {\small${a+b-k}$} at -22 46
			\pinlabel {\small$a$} at 33 55
			\pinlabel {\small${a - k}$} at 35 18
			\pinlabel {\small$b$} at 70 5
			\endlabellist
			\includegraphics[width=.13\textwidth]{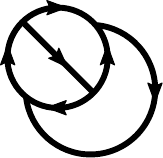}
		\end{gathered}\:
}
\newcommand{\HopfLink}{
	\:\begin{gathered}
		\vspace{-3pt}
		\centering
		\labellist
		\pinlabel {\small$a$} at 8 19
		\pinlabel {\small$b$} at 52 20
		\endlabellist
		\includegraphics[width=.11\textwidth]{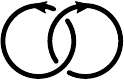}
	\end{gathered}\:
}

\newcommand{\Xdoti}{
	\:\begin{gathered}
		\labellist
		\pinlabel {\large$\bullet$} at 51 49
		\pinlabel {\small$e_i$} at 61 41
		\endlabellist
		\includegraphics[width=.08\textwidth]{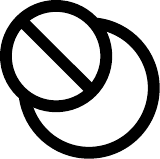}
	\end{gathered}\:
}
\newcommand{\Ydoti}{
	\:\begin{gathered}
		\labellist
		\pinlabel {\large$\bullet$} at 27 25
		\pinlabel {\small$e_i$} at 19 40
		\endlabellist
		\includegraphics[width=.08\textwidth]{smallerTheta}
	\end{gathered}\:
}
\newcommand{\Zdoti}{
	\:\begin{gathered}
		\labellist
		\pinlabel {\large$\bullet$} at 27 72
		\pinlabel {\small$e_i$} at 33 60
		\endlabellist
		\includegraphics[width=.08\textwidth]{smallerTheta}
	\end{gathered}\:
}

\newcommand{\Xprimedoti}{
	\:\begin{gathered}
		\labellist
		\pinlabel {\large$\bullet$} at 49 23
		\pinlabel {\small$e_i$} at 61 24
		\endlabellist
		\includegraphics[width=.09\textwidth]{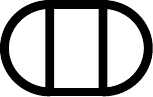}
	\end{gathered}\:
}
\newcommand{\Yprimedoti}{
	\:\begin{gathered}
		\labellist
		\pinlabel {\large$\bullet$} at 24 23
		\pinlabel {\small$e_i$} at 36 24
		\endlabellist
		\includegraphics[width=.09\textwidth]{smallThetaFewer}
	\end{gathered}\:
}
\newcommand{\Zprimedoti}{
	\:\begin{gathered}
		\labellist
		\pinlabel {\large$\bullet$} at 71 23
		\pinlabel {\small$e_i$} at 60 24
		\endlabellist
		\includegraphics[width=.09\textwidth]{smallThetaFewer}
	\end{gathered}\:
}

\newcommand{\Wprimeprimedoti}{
	\:\begin{gathered}
		\labellist
		\pinlabel {\large$\bullet$} at 24 23
		\pinlabel {\small$e_i$} at 35 24
		\endlabellist
		\includegraphics[width=.12\textwidth]{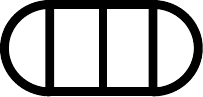}
	\end{gathered}\:
}

\newcommand{\Yprimeprimedoti}{
	\:\begin{gathered}
		\labellist
		\pinlabel {\large$\bullet$} at 49.5 23
		\pinlabel {\small$e_i$} at 61 25
		\endlabellist
		\includegraphics[width=.12\textwidth]{smallTheta}
	\end{gathered}\:
}

\newcommand{\Xprimeprimedoti}{
	\:\begin{gathered}
		\labellist
		\pinlabel {\large$\bullet$} at 73.5 23
		\pinlabel {\small$e_i$} at 62 25
		\endlabellist
		\includegraphics[width=.12\textwidth]{smallTheta}
	\end{gathered}\:
}

\newcommand{\Zprimeprimedoti}{
	\:\begin{gathered}
		\labellist
		\pinlabel {\large$\bullet$} at 95 23
		\pinlabel {\small$e_i$} at 84 25
		\endlabellist
		\includegraphics[width=.12\textwidth]{smallTheta}
	\end{gathered}\:
}

\newcommand{\Nprimeprimedoti}{
	\:\begin{gathered}
		\labellist
		\pinlabel {\large$\bullet$} at 2 23
		\pinlabel {\small$e_i$} at -8 25
		\endlabellist
		\includegraphics[width=.12\textwidth]{smallTheta}
	\end{gathered}\:
}

\title{Colored sl(N) homology and SU(N) representations}
\author{Joshua Wang}
\date{}

\begin{document}
\maketitle

\begin{abstract}
	We provide the first complete computations of colored $\sl(N)$ homology for a nontrivial knot. In doing so, we show that the colored $\sl(N)$ homology of the trefoil labeled by an exterior power of the defining representation is isomorphic to the cohomology of a closed manifold naturally associated to the trefoil. This manifold is the set of homomorphisms from the fundamental group of the complement of the trefoil to $\SU(N)$ that send meridians to a particular conjugacy class depending on the label. We also provide complete computations and analogous isomorphisms for the first nontrivial link, the Hopf link.
\end{abstract}

\tableofcontents

\section{Introduction}\label{sec:introduction}

\subsection{Statement of the main result}

Colored $\sl(N)$ homology \cite{MR3234803,MR3545951,MR4164001}, also known as colored Khovanov--Rozansky homology, is an invariant of oriented \textit{labeled} links, where the labeling is a choice of integer $a$ satisfying $0 \leq a \leq N$ for each component of the link. The label $a$ may be interpreted as the $a$th exterior power of the defining representation of $\sl(N)$. We generally let $L$ or $K$ denote an oriented labeled link and do not include the labels explicitly in our notation. Like Khovanov homology, the colored $\sl(N)$ homology of an oriented labeled link $L$, denoted $\KR_N(L)$, takes the form of a bigraded abelian group and can now be defined combinatorially \cite{MR4164001}. Khovanov homology is the special case that $N = 2$ and all labels are $1$, and (uncolored) Khovanov--Rozansky homology \cite{MR2391017} is the case that $N$ is arbitrary and all labels are $1$. 

We provide the first complete computations of colored $\sl(N)$ homology for a nontrivial knot and a nontrivial link. In contrast to Khovanov homology and (uncolored) Khovanov--Rozansky homology, colored $\sl(N)$ homology is very difficult to compute. The only prior computations in the mathematical literature for links with all labels greater than $1$ are the unlinks, though there are a number of relevant calculations for closely related invariants of the Hopf link \cite{MR2863366,MR3470705,https://doi.org/10.48550/arxiv.2107.09590} and other computations in the physics literature \cite{MR2670927,MR2985329,MR3863060}. 

We now describe a space naturally associated to an oriented labeled link that is directly analogous to the moduli space studied by Lobb and Zentner \cite{MR3190356} and Grant \cite{MR3125899}, and generalizes the space of meridian-traceless $\SU(2)$ representations, see for example \cite{MR2860345,https://doi.org/10.48550/arxiv.0806.2902}.

\begin{df}\label{df:representationSpace}
	Let $L$ be an oriented labeled link. If $\mu$ is a meridian of a component $C$ of $L$, let $a(\mu)$ be the label of $C$. Let $\sr R_N(L)$ be the space of homomorphisms\[
		\sr R_N(L) \coloneq \left\{\:\rho\colon \pi_1(\R^3\setminus L) \to \SU(N) \:\big|\: \rho(\mu) \text{ is conjugate to } \Phi_{a(\mu)} \text{ for each meridian }\mu \:\right\}
	\]where \[
		\Phi_a \coloneq e^{\,a\pi i/N} \begin{pmatrix}
			-\Id_{a} & 0\\0 & \Id_{N - a}
		\end{pmatrix} \in \SU(N).
	\]
\end{df}

In computing the colored $\sl(N)$ homology of the trefoil and Hopf link, we verify the following isomorphism.

\begin{thm}\label{thm:plainisomorphismforHopfLinkAndTrefoil}
	Fix integers $0 \leq a,b \leq N$, and let $L$ be either the right-handed trefoil labeled $a$ or the positive Hopf link with components labeled $a,b$. Then there is an isomorphism of abelian groups \[
		\KR_N(L) \cong H^*(\sr R_N(L)).
	\]
\end{thm}

Theorem~\ref{thm:plainisomorphismforHopfLinkAndTrefoil} is the most basic form of the isomorphism. We provide refinements that capture additional structure of $\KR_N(L)$ including the bigrading information, the module structure induced by a basepoint, and the reduced and equivariant versions of colored $\sl(N)$ homology all over $\Z$. See Tables~\ref{table:trefoil24} -- \ref{table:trefoil36} for explicit bigraded computations for the trefoil with small values of $a,N$.

\subsection{Context and motivation}

Kronheimer and Mrowka's landmark result that Khovanov homology detects the unknot \cite{MR2805599} uses a spectral sequence they construct from Khovanov homology to a link invariant $\Isharp$ that they define using singular $\SU(2)$ instantons. The construction of $\Isharp$ and the spectral sequence was motivated by an unexpected coincidence between the Khovanov homology of the trefoil $K$ and the cohomology of the space $\sr R_2(K)$ of homomorphisms $\pi_1(\R^3\setminus K) \to \SU(2)$ sending meridians to traceless matrices. 

Singular $\SU(N)$ instanton homology for oriented labeled links was defined in \cite{MR2860345} and was initially expected to coincide with the cohomology of the $\SU(N)$ representation space $\sr R_N(L)$, defined above, for certain simple knots and links. An analogous conjectural spectral sequence from colored $\sl(N)$ homology to singular $\SU(N)$ instanton homology would suggest a potential relationship between $\KR_N(L)$ and $H^*(\sr R_N(L))$. Lobb and Zentner \cite{MR3190356} and Grant \cite{MR3125899}, motivated by the conjectural spectral sequence, considered the relationship between these $\SU(N)$ representation spaces and colored $\sl(N)$ invariants in the context of planar webs. 

In contrast to the initial expectation, however, the $\SU(N)$ instanton homology that is currently defined is now known \textit{not} to coincide with the cohomology of $\sr R_N(L)$, even for the trefoil labeled $1$ when $N \ge 3$ \cite{privateCom}. A spectral sequence from $\KR_N(L)$ to an $\SU(N)$ instanton homology would likely require modifying the definition of $\SU(N)$ instanton homology, but the author is unaware of particularly promising candidates. Furthermore, a spectral sequence from \textit{uncolored} $\sl(N)$ homology to $\SU(N)$ instanton homology could likely be proved using a well-established technique related to skein exact triangles \cite{MR2141852,MR2764887,MR2805599,MR3394316,https://doi.org/10.48550/arxiv.1811.07848}. \textit{Colored} $\sl(N)$ homology does not satisfy a skein exact triangle, so a spectral sequence from colored $\sl(N)$ homology would require something new.

Despite these complications, one can still attempt to compare $\KR_N(L)$ and $H^*(\sr R_N(L))$. In fact, given the lack of explicit computations of colored $\sl(N)$ homology for links, checking whether $\KR_N(L)$ and $H^*(\sr R_N(L))$ agree for simple knots and links is not merely a test of the viability of a spectral sequence to singular $\SU(N)$ instanton homology, which is already known to be unviable in its current form. Instead, it provides a precise, concrete, and concise prediction for a complicated and mysterious mathematical object $\KR_N(L)$. This is unlike the case of planar webs considered in \cite{MR3190356,MR3125899}, where the colored $\sl(N)$ invariant is straightforward to compute using algorithmic reduction rules. 

We also note that even if the conjectural spectral sequence is established, Theorem~\ref{thm:plainisomorphismforHopfLinkAndTrefoil} would not follow as a consequence. A concrete illustration may be found in the uncolored $N = 2$ case for the reduced invariants where the spectral sequence has been established. There are three relevant invariants: reduced Khovanov homology $\rKh(L)$, the cohomology $H^*(\ol{\sr R_2}(L))$ of the reduced space of meridian-traceless $\SU(2)$ representations, and reduced singular $\SU(2)$ instanton homology $\Inat(L)$. Just among torus knots, there are examples where the ranks of any two of the three invariants agree but disagree with the third. See Table~\ref{table:ranksOf3Invariants}. 

\begin{table}[!ht]
	\centering
	\begin{tabular}{ |c|c|c|c| }
	\hline
	\phantom{\Big|}Torus knot $K$ & $\rk \rKh(K)$ & $\rk \Inat(K)$ & $\rk H^*(\ol{\sr R_2}(K))$\\
	\hline
	$T(2,3)$ & $3$ & $3$ & $3$\\
	\hline
	$T(3,4)$ & $5$ & $5$ & $7$\\
	\hline
	$T(4,5)$ & $9$ & $7$ & $9$\\
	\hline
	$T(4,7)$ & $17$ & $11$ & $15$\\
	\hline
	$T(5,7)$ & $29$ & $17$ & $17$\\
	\hline
	\end{tabular}
	\captionsetup{width=.8\linewidth}
	\caption{The ranks of three related invariants for certain torus knots. The ranks of any two of the three invariants can agree while disagreeing with the third. The ranks of all three can all agree or all disagree as well. The second and fourth columns are from \cite{MR3158776} and the third column is from \cite{MR4407491}.}
	\label{table:ranksOf3Invariants}
\end{table}

\begin{table}[!ht]
	\centering
	\def\arraystretch{1.1}
	\begin{tabular}{ |c|c|c|c|c|c|c|c| }
	\hline
	$30$ & $\Z$ & \hspace{2em} & \hspace{2em} & \hspace{2em} & \hspace{2em} & &\\
	\hline
	$28$ & & $\:\,\Z_2$ & & & & &\\
	\hline
	$26$ & & $\:\,\Z_4$ & & & & &\\
	\hline
	$24$ & & $\!\!\Z \hspace{-.1em}\oplus\hspace{-.1em}\Z_2\!\!\!$ & & & & &\\
	\hline
	$22$ & & $\Z$ & $\:\,\Z_2$ & $\Z$ & & &\\
	\hline
	$20$ & & & $\:\,\Z_4$ & $\:\,\Z^2$ & & &\\
	\hline
	$18$ & & & $\:\,\Z_2$ & $\:\,\Z^2$ & & &\\
	\hline
	$16$ & & & $\Z$ & $\Z$ & $\:\,\Z_4$ & &\\
	\hline
	$14$ & & & & & $\!\!\Z \hspace{-.1em}\oplus\hspace{-.1em} \Z_4\!\!\!$ & &\\
	\hline
	$12$ & & & & & $\:\,\Z^2$ & & $\Z$\\
	\hline
	$10$ & & & & & $\:\,\Z^2$ & & $\Z$\\
	\hline
	$8$  & & & & & $\Z$ & & $\!\:\,\Z^2\!$\\
	\hline
	$6$  & & & & & & & $\Z$\\
	\hline
	$4$  & & & & & & & $\Z$\\
	\hline
	& $-6$ & $-5$ & $-4$ & $-3$ & $-2$ & $-1$ & $0$\\
	\hline
	\end{tabular}
	\captionsetup{width=.8\linewidth}
	\caption{Colored $\sl(4)$ homology of the right-handed trefoil labeled $2$. The $q$-grading is vertical and the $h$-grading is horizontal. Its graded Euler characteristic is $q^{30} - q^{24} - 2q^{22} - 2q^{20} - 2q^{18} + q^{14} + 3q^{12} + 3q^{10} + 3q^8 + q^6 + q^4$ which is the A3 quantum invariant of weight $0,1,0$ on the Knot Atlas \cite{KnotAtlas}.}
	\label{table:trefoil24}
\end{table}

\begin{table}[!ht]
	\centering
	\def\arraystretch{1.1}
	\begin{tabular}{ |c|c|c|c|c|c|c|c| }
	\hline
	$40$ & $\Z$ & \hspace{2.2em} & \hspace{2.2em} & \hspace{2.2em} & \hspace{2.2em} & &\\
	\hline
	$38$ & $\Z$ & & & & & &\\
	\hline
	$36$ & $\Z$ & & & & & &\\
	\hline
	$34$ & & $\Z_{10}\!\!\!\!$ & & & & &\\
	\hline
	$32$ & & $\Z \hspace{-.1em}\oplus\hspace{-.1em} \Z_5\!\!$ & & & & &\\
	\hline
	$30$ & & $\:\,\Z^2$ & & $\Z$ & & &\\
	\hline
	$28$ & & $\:\,\Z^2$ & & $\:\,\Z^2$ & & &\\
	\hline
	$26$ & & $\Z$ & $\Z_5\!\!$ & $\:\,\Z^3$ & & &\\
	\hline
	$24$ & & & $\Z_{10}\!\!\!\!$ & $\:\,\Z^3$ & & &\\
	\hline
	$22$ & & & $\Z$ & $\:\,\Z^2$ & $\Z_5\!\!$ & &\\
	\hline
	$20$ & & & $\Z$ & $\Z$ & $\!\Z\hspace{-.1em}\oplus\hspace{-.1em}\Z_5\!\!\!\!$ & &\\
	\hline
	$18$ & & & $\Z$ & & $\!\!\!\Z^2\!\oplus\hspace{-.1em}\Z_5\!\!\!\!$ & & $\Z$\\
	\hline
	$16$ & & & & & $\:\,\Z^3$ & & $\Z$\\
	\hline
	$14$ & & & & & $\:\,\Z^3$ & & $\!\:\,\Z^2\!$\\
	\hline
	$12$ & & & & & $\:\,\Z^2$ & & $\!\:\,\Z^2\!$\\
	\hline
	$10$ & & & & & $\Z$ & & $\!\:\,\Z^2\!$\\
	\hline
	$8$  & & & & & & & $\Z$\\
	\hline
	$6$  & & & & & & & $\Z$\\
	\hline
	& $-6$ & $-5$ & $-4$ & $-3$ & $-2$ & $-1$ & $0$\\
	\hline
	\end{tabular}
	\captionsetup{width=.8\linewidth}
	\caption{Colored $\sl(5)$ homology of the right-handed trefoil labeled $2$. The $q$-grading is vertical and the $h$-grading is horizontal. Its graded Euler characteristic is $q^{40} + q^{38} + q^{36} - q^{32} - 3q^{30} - 4q^{28} - 4q^{26} - 3q^{24} - q^{22} + q^{20} + 4q^{18} + 4q^{16} + 5q^{14} + 4q^{12} + 3q^{10} + q^8 + q^6$ which is the A4 quantum invariant of weight $0,1,0,0$ on the Knot Atlas \cite{KnotAtlas}.}
	\label{table:trefoil25}
\end{table}

\begin{table}[!ht]
	\centering
	\def\arraystretch{1.1}
	\begin{tabular}{ |c|c|c|c|c|c|c|c| }
	\hline
	$50$ & $\Z$ & $\hspace{4em}$ & $\hspace{2.5em}$ & $\hspace{2.5em}$ & $\hspace{2.5em}$ & &\\
	\hline
	$48$ & $\Z$ & & & & & &\\
	\hline
	$46$ & $\!\:\,\Z^2\!$ & & & & & &\\
	\hline
	$44$ & $\Z$ & $\Z_2\!\!$ & & & & &\\
	\hline
	$42$ & $\Z$ & $\Z_6\!\!$ & & & & &\\
	\hline
	$40$ & & $\!\!\Z\hspace{-.1em} \oplus\hspace{-.1em} \Z_2\hspace{-.1em} \oplus\hspace{-.1em} \Z_6\!\!\!$ & & & & &\\
	\hline
	$38$ & & $\Z^2\!\oplus\hspace{-.1em}\Z_6\!$ & & $\Z$ & & &\\
	\hline
	$36$ & & $\Z^3\!\oplus\hspace{-.1em}\Z_2\!$ & & $\:\,\Z^2$ & & &\\
	\hline
	$34$ & & $\:\,\Z^3$ & $\Z_2\!\!$ & $\:\,\Z^3$ & & &\\
	\hline
	$32$ & & $\:\,\Z^2$ & $\Z_6\!\!$ & $\:\,\Z^4$ & & &\\
	\hline
	$30$ & & $\Z$ & $\Z_2\!\oplus\hspace{-.1em} \Z_6\!$ & $\:\,\Z^4$ & & &\\
	\hline
	$28$ & & & $\Z\hspace{-.1em}\oplus\hspace{-.1em}\Z_6\!\!\!$ & $\:\,\Z^3$ & $\Z_6\!\!$ & &\\
	\hline
	$26$ & & & $\Z\hspace{-.1em}\oplus\hspace{-.1em}\Z_2\!\!\!$ & $\:\,\Z^2$ & $\!\Z\hspace{-.1em}\oplus\hspace{-.1em} \Z_6\!\!\!\!$ & &\\
	\hline
	$24$ & & & $\:\,\Z^2$ & $\Z$ & $\!\!\Z^2\!\oplus\hspace{-.1em} \Z_6\!\!\!$ & & $\Z$\\
	\hline
	$22$ & & & $\Z$ & & $\!\!\Z^3\!\oplus\hspace{-.1em} \Z_6\!\!\!$ & & $\Z$\\
	\hline
	$20$ & & & $\Z$ & & $\:\,\Z^4$ & & $\!\:\,\Z^2\!$\\
	\hline
	$18$ & & & & & $\:\,\Z^4$ & & $\!\:\,\Z^2\!$\\
	\hline
	$16$ & & & & & $\:\,\Z^3$ & & $\!\:\,\Z^3\!$\\
	\hline
	$14$ & & & & & $\:\,\Z^2$ & & $\!\:\,\Z^2\!$\\
	\hline
	$12$ & & & & & $\Z$ & & $\!\:\,\Z^2\!$\\
	\hline
	$10$ & & & & & & & $\Z$\\
	\hline
	$8$  & & & & & & & $\Z$\\
	\hline
	& $-6$ & $-5$ & $-4$ & $-3$ & $-2$ & $-1$ & $0$\\
	\hline
	\end{tabular}
	\captionsetup{width=.8\linewidth}
	\caption{Colored $\sl(6)$ homology of the right-handed trefoil labeled $2$. The $q$-grading is vertical and the $h$-grading is horizontal.}
	\label{table:trefoil26}
\end{table}

\begin{table}[!ht]
	\centering
	\def\arraystretch{1.1}
	\begin{tabular}{ |c|c|c|c|c|c|c|c|c|c|c| }
	\hline
	$63$ & $\Z$ & \hspace{2.5em} & \hspace{3.5em} & \hspace{3.5em} & \hspace{5.5em} & \hspace{5.5em} & \hspace{3.5em} & \hspace{3.5em} & \hspace{2.5em} &\\
	\hline
	$61$ & & $\:\Z_2$ & & & & & & & &\\
	\hline
	$59$ & & $\:\Z_4$ & & & & & & & &\\
	\hline
	$57$ & & $\:(\Z_2)^2\!\!$ & & & & & & & &\\
	\hline
	$55$ & & $\!\Z\hspace{-.1em}\oplus\hspace{-.1em}\Z_2\!\!$ & $\Z_2$ & & & & & & &\\
	\hline
	$53$ & & $\!\Z\hspace{-.1em}\oplus\hspace{-.1em}\Z_2\!\!$ & $\Z_2$ & $\Z$ & & & & & &\\
	\hline
	$51$ & & $\!\Z\hspace{-.1em}\oplus\hspace{-.1em}\Z_2\!\!$ & $\:\,(\Z_2)^2$ & $\:\,\Z^2$ & & & & & &\\
	\hline
	$49$ & & & $\!\!(\Z_2)^2\!\oplus\hspace{-.15em}\Z_4\!\!$ & $\:\,\Z^3$ & & & & & &\\
	\hline
	$47$ & & & $\!\!(\Z_2)^2\!\oplus\hspace{-.15em}\Z_4\!\!$ & $\:\,\Z^3$ & $\:\,\Z_2$ & & & & &\\
	\hline
	$45$ & & & $\Z\hspace{-.1em}\oplus\hspace{-.1em}(\Z_2)^2\!\!$ & $\Z^2\!\oplus\hspace{-.1em}\Z_2$ & $\Z_2\!\oplus\hspace{-.1em}\Z_6$ & & & & &\\
	\hline
	$43$ & & & $\Z\hspace{-.1em}\oplus\hspace{-.1em}\Z_2\!$ & $\:\Z\hspace{-.1em}\oplus\hspace{-.1em}\Z_2\!$ & $\!\Z\hspace{-.1em}\oplus\hspace{-.1em}\Z_2\!\oplus\hspace{-.1em}(\Z_6)^2\!\!$ & & & & &\\
	\hline
	$41$ & & & $\Z\hspace{-.1em}\oplus\hspace{-.1em}\Z_2$ & $\:\,\Z_2$ & $\!\!\Z^3\!\oplus\hspace{-.1em}\Z_2\!\oplus\hspace{-.1em}(\Z_6)^2\!\!\!$ & & $\Z$ & & &\\
	\hline
	$39$ & & & & $\:\,(\Z_2)^2\!\!$ & $\Z^4\!\oplus\hspace{-.1em}\Z_2\!\oplus\hspace{-.1em}\Z_6$ & & $\:\,\Z^2$ & & &\\
	\hline
	$37$ & & & & $\:\,\Z_4$ & $\Z^6\!\oplus\hspace{-.1em}\Z_2$ & $\:\,\Z_2$ & $\:\,\Z^4$ & & &\\
	\hline
	$35$ & & & & $\:\,\Z_2$ & $\:\,\Z^4$ & $\Z_2\!\oplus\hspace{-.1em}\Z_6$ & $\:\,\Z^5$ & & &\\
	\hline
	$33$ & & & & $\Z$ & $\:\,\Z^3$ & $\Z_2\!\oplus\hspace{-.1em}(\Z_6)^2\!\!$ & $\:\,\Z^6$ & & &\\
	\hline
	$31$ & & & & & $\Z$ & $\!\!\Z\hspace{-.1em}\oplus\hspace{-.1em}\Z_2\!\oplus\hspace{-.1em}(\Z_6)^2\!\!\!$ & $\:\,\Z^5$ & $\Z_6$ & &\\
	\hline
	$29$ & & & & & & $\Z^2\!\oplus\hspace{-.1em}\Z_2\!\oplus\hspace{-.1em}\Z_6\!\!$ & $\:\,\Z^4$ & $\Z \hspace{-.1em}\oplus\hspace{-.1em}\Z_6$ & &\\
	\hline
	$27$ & & & & & & $\Z^3\!\oplus\hspace{-.1em}\Z_2$ & $\:\,\Z^2$ & $\!\!\!\Z^2\!\oplus\hspace{-.1em}(\Z_6)^2\!\!\!\!$ & & $\Z$\\
	\hline
	$25$ & & & & & & $\:\,\Z^3$ & $\Z$ & $\Z^4\!\oplus\hspace{-.1em}\Z_6$ & & $\Z$\\
	\hline
	$23$ & & & & & & $\:\,\Z^2$ & & $\Z^5\!\oplus\hspace{-.1em}\Z_6$ & & $\!\:\,\Z^2\!$\\
	\hline
	$21$ & & & & & & $\Z$ & & $\:\,\Z^6$ & & $\!\:\,\Z^3\!$\\
	\hline
	$19$ & & & & & & & & $\:\,\Z^5$ & & $\!\:\,\Z^3\!$\\
	\hline
	$17$ & & & & & & & & $\:\,\Z^4$ & & $\!\:\,\Z^3\!$\\
	\hline
	$15$ & & & & & & & & $\:\,\Z^2$ & & $\!\:\,\Z^3\!$\\
	\hline
	$13$ & & & & & & & & $\Z$ & & $\!\:\,\Z^2\!$\\
	\hline
	$11$ & & & & & & & & & & $\Z$\\
	\hline
	$9$  & & & & & & & & & & $\Z$\\
	\hline
	& $-9$ & $-8$ & $-7$ & $-6$ & $-5$ & $-4$ & $-3$ & $-2$ & $-1$ & $0$\\
	\hline
	\end{tabular}
	\captionsetup{width=.8\linewidth}
	\caption{Colored $\sl(6)$ homology of the right-handed trefoil labeled $3$. The $q$-grading is vertical and the $h$-grading is horizontal.}
	\label{table:trefoil36}
\end{table}

\subsection{Variations of the main result}

Colored $\sl(N)$ homology carries much more structure than simply that of an abelian group. Much of this structure is reflected in the cohomology of the $\SU(N)$ representation space $\sr R_N(L)$. In particular, we discuss bigradings, the module structure induced by a basepoint, and the reduced and equivariant versions of colored $\sl(N)$ homology. Although we consider each of these structures individually, the results discussed can be combined in the natural way.

\subsubsection*{Module structure and reduced homology}

If $D$ is a diagram of an oriented labeled link $L$, there is an associated colored $\sl(N)$ chain complex $\KRC_N(D)$ whose homology is $\KR_N(L)$. A basepoint $p$ on an arc labeled $a$ determines an action of the cohomology ring of $\G(a,N)$ on $\KRC_N(D)$ through chain maps, where $\G(a,N)$ is the complex Grassmannian of $a$-dimensional vector subspaces of $\C^N$. The induced action of $H^*(\G(a,N))$ on $\KR_N(L)$ is the \textit{module structure induced by the basepoint $p$}. The image of the action of the fundamental class in $H^*(\G(a,N))$ on $\KRC_N(D)$ is a subcomplex whose homology, denoted $\rKR_N(L,p)$, is the \textit{reduced colored $\sl(N)$ homology of $L$ with respect to $p$}. 

Let $C_a \subset \SU(N)$ denote the conjugacy class of $\Phi_a \in \SU(N)$, which is given in Definition~\ref{df:representationSpace}. It turns out that associating to a matrix in $C_a$ its $(-e^{\,a\pi i/N})$-eigenspace gives an identification between $C_a$ and $\G(a,N)$. 
The basepoint $p$ on an arc of $D$ determines a meridian $\mu_p \in \pi_1(\R^3\setminus L)$. It is the generator of the Wirtinger presentation associated to the arc containing $p$. This choice of basepoint $p$ thereby determines a map $\sr R_N(L) \to C_a$ by sending $\rho \in \sr R_N(L)$ to $\rho(\mu_p) \in C_a$. This map to $C_a = \G(a,N)$ is the projection map of a fiber bundle \[
	\ol{\sr R_N}(L,p) \to \sr R_N(L) \to \G(a,N)
\]where $\ol{\sr R_N}(L,p)$ is the space of representations $\rho \in \sr R_N(L)$ for which $\rho(\mu_p) = \Phi_a$. The map on cohomology induced by $\sr R_N(L) \to \G(a,N)$ makes $H^*(\sr R_N(L))$ into a module over $H^*(\G(a,N))$. 

\begin{prop}\label{prop:moduleandReducedIsomorphism}
	Fix integers $0 \leq a,b \leq N$, and let $L$ be either the right-handed trefoil labeled $a$ or the positive Hopf link with components labeled $a,b$. If $p$ is a basepoint on the component labeled $a$ of $L$, then $\KR_N(L) \cong H^*(\sr R_N(L))$ as modules over $H^*(\G(a,N))$ and $\rKR_N(L,p) \cong H^*(\ol{\sr R_N}(L,p))$ as abelian groups.
\end{prop}

\subsubsection*{Equivariant homology}

We briefly review equivariant cohomology for the unitary group $G = \U(N)$. Recall that there is a principal $G$-bundle $\E G \to \B G$ where $\E G$ is contractible and $\B G$ is the classifying space of $G$. The cohomology of $\B G$ can be canonically identified with the ring of symmetric polynomials in $N$ variables \[
	H^*(\B G) = \Z[X_1,\ldots,X_N]^{\fk{S}_N} = \Z[e_1,\ldots,e_N] \eqqcolon \Sym(N)
\]where $e_i$ denotes the $i$th elementary symmetric polynomial. The identification is given by sending the Chern class $c_i \in H^*(\B G)$ to $e_i \in \Sym(N)$. If $X$ is a space equipped with a continuous left action of $G$, then the \textit{$G$-equivariant cohomology of $X$} is a graded-commutative ring $H^*_G(X)$ that is naturally a module over $H^*(\B G)$. It is defined to be the ordinary cohomology of the balanced product $\E G \x_G X$ \[
	H^*_G(X) \coloneq H^*(\E G \x_G X)
\]and its module structure over $H^*(\B G)$ arises from the natural map $\E G \x_G X \to \B G$. 

The $\SU(N)$ representation space $\sr R_N(L)$ of an oriented labeled link $L$ naturally admits a continuous left action of $G = \U(N)$, given by conjugation. Explicitly, a matrix $U \in \U(N)$ sends $\rho \in \sr R_N(L)$ to the representation $U \rho U^{-1} \in \sr R_N(L)$ defined by $\gamma \mapsto U\rho(\gamma)U^{-1}$ for $\gamma \in \pi_1(\R^3\setminus L)$. We may therefore consider the $\U(N)$-equivariant cohomology $H^*_{\U(N)}(\sr R_N(L))$ of the $\SU(N)$ representation space with respect to this action. It takes the form of a module over $H^*(\BU(N)) = \Sym(N)$. 

The \textit{$\U(N)$-equivariant colored $\sl(N)$ homology} \cite{MR4164001,MR3877770} of an oriented labeled link $L$, which we denote $\KR_{\U(N)}(L)$, is a bigraded module over $H^*(\BU(N)) = \Sym(N)$. It is the homology of a chain complex $\KRC_{\U(N)}(D)$ over $\Sym(N)$ associated to a diagram $D$. It generalizes $\U(2)$-equivariant Khovanov homology, which is the theory associated to the Frobenius system $\cl F_5$ of \cite{MR2232858}. The complex underlying $\U(2)$-equivariant Khovanov homology determines all other versions of Khovanov homology including Lee homology \cite{MR2173845} and Bar-Natan homology \cite{MR2174270}. In a similar way, $\KRC_{\U(N)}(D)$ determines the different variations of colored $\sl(N)$ homology analogous to Lee homology and Bar-Natan homology. 

\begin{prop}\label{prop:equivariantCohomologyisomorphism}
	Fix integers $0 \leq a,b \leq N$, and let $L$ be either the right-handed trefoil labeled $a$ or the positive Hopf link with components labeled $a,b$. Then $\KR^{\phantom{}}_{\U(N)}(L) \cong H^*_{\U(N)}(\sr R_N(L))$ as modules over $H^*(\BU(N)) = \Sym(N)$. 
\end{prop}

Since $\Sym(N)$ is not a principal ideal domain, the universal coefficient theorem does not apply to the chain complex $\KRC_{\U(N)}$. Our proof of Proposition~\ref{prop:equivariantCohomologyisomorphism} gives an explicit description of the $\U(N)$-equivariant complex up to homotopy from which the analogues of Bar-Natan homology and Lee homology for colored $\sl(N)$ homology can be determined. 

\subsubsection*{Bigradings}

Colored $\sl(N)$ homology is equipped with a homological grading called the $h$-grading and another grading called the $q$-grading. Its graded Euler characteristic with respect to the $h$-grading is the Reshetikhin--Turaev $\sl(N)$ polynomial. To relate the bigrading information of colored $\sl(N)$ homology to $\sr R_N(L)$, we first describe the structure of $\sr R_N(L)$ as a left $\U(N)$-space in more detail. If $L$ is the trefoil or the Hopf link, then $\sr R_N(L)$ turns out to be a disjoint union of finitely many orbits of the $\U(N)$ action given by conjugation. This observation extends to any labeled $2$-bridge knot or link. Hence each connected component of $\sr R_N(L)$ is a homogeneous space $G/K$, which is the space of left cosets of a subgroup $K$ of $G = \U(N)$. The subgroups $K$ that arise for us are always isomorphic to a product of unitary groups. There is a fiber bundle \[
	G/K \to \B K \to \B G
\]obtained by viewing $\B K$ as the quotient of $\E G$ by the action of $K$ induced by the natural action of $G$. The Eilenberg--Moore spectral sequence associated to this bundle has $E_2$-page \[
	\Tor_{H^*(\B G)}(\Z,H^*(\B K))
\]and converges to $H^*(G/K)$. This $\Tor$ group is bigraded; it is the version of $\Tor$ for graded modules over a graded ring, see for example \cite[Chapter 7]{MR1793722}. We call the grading on $\Tor_{H^*(\B G)}(\Z,H^*(\B K))$ arising from the internal cohomological gradings on $H^*(\B K)$ and $H^*(\B G)$ the $q$-grading, and we call the homological grading of $\Tor$ the $h$-grading. Under general conditions on $K$ that are satisfied if $K$ is a product of unitary groups, Gugenheim and May \cite{MR0394720} prove that this spectral sequence degenerates and \[
	H^*(G/K) \cong \Tor_{H^*(\B G)}(\Z,H^*(\B K))
\]as graded abelian groups. The bigrading on $\Tor_{H^*(\B G)}(\Z,H^*(\B K))$ is collapsed to a single grading by declaring that the summand in bidegree $(i,j)$ lies in total degree $i + j$. 
If $V$ is a bigraded abelian group whose direct summand in $(h,q)$-bidegree $(i,j)$ is $V_{i,j}$, then define $h^kq^lV$ by $(h^kq^lV)_{k+i,l+j} = V_{i,j}$.

\begin{prop}\label{prop:bigradedTrefoil}
	Let $K$ be the right-handed trefoil labeled $0 \leq a \leq N$. Then \[
		\sr R_N(K) = \bigsqcup_{l = \max(2a - N,0)}^a \U(N)/K_l
	\]where $K_l \subseteq \U(N)$ is the subgroup $\U(l) \x \Delta\!\U(a - l) \x \U(N - 2a + l)$ of block diagonal matrices of the form \[
		\begin{pmatrix}
			U\\ & V\\ & & V\\ & & & W
		\end{pmatrix} \in \U(N) \quad\qquad U \in \U(l),\quad V \in \U(a - l),\quad W \in \U(N - 2a + l)
	\]and there is an isomorphism of bigraded abelian groups \[
		\KR_N(K) \cong \bigoplus_{l=\max(2a - N,0)}^a h^{-2(a-l)}q^{a(N-a)+2(a-l)(a-l+1)} \Tor_{H^*(\BU(N))}(\Z,H^*(\B K_l)).
	\]
\end{prop}

\begin{prop}\label{prop:bigradedHopfLink}
	Let $L$ be the positive Hopf link with components labeled $0 \leq a,b \leq N$. Then \[
		\sr R_N(L) = \bigsqcup_{k=\max(a + b - N,0)}^{\min(a,b)} \F(k,a-k,b-k;N)
	\]where $\F(k,a-k,b-k;N)$ denotes the partial flag manifold consisting of triples of pairwise orthogonal vector subspaces of $\C^N$ of dimensions $k,a-k,b-k$. There is an isomorphism of bigraded abelian groups
	\[
		\KR_N(L) \cong \bigoplus_{k=\max(a + b - N,0)}^{\min(a,b)} h^{2k} q^{(a+b)(a+b-N-2k)+2k^2} H^*(\F(k,a-k,b-k;N))
	\]
	where $H^*(\F(k,a-k,b-k;N))$ is supported in $h$-grading zero and its cohomological grading is its $q$-grading.
\end{prop}

The partial flag manifold $\F(k,a-k,b-k;N)$ is the homogeneous space $\U(N)/H_k$ where $H_k$ is the subgroup $\U(k) \x \U(a - k) \x \U(b - k) \x \U(N - a - b + k)$ consisting of block diagonal matrices. Because $H_k$ and $\U(N)$ have the same rank, a theorem of Borel \cite{MR51508} implies that $\Tor_{H^*(\BU(N))}(\Z,H^*(\B H_k))$ is concentrated in $h$-grading zero, so Propositions~\ref{prop:bigradedTrefoil} and \ref{prop:bigradedHopfLink} are directly analogous.

\subsection{Structure of the paper}

In section~\ref{sec:generalities}, we provide general background on colored $\sl(N)$ homology, all of which is either standard or a straightforward generalization of an existing result. We draw heavily from \cite{MR3545951,MR3590355,MR4164001,MR3877770,https://doi.org/10.48550/arxiv.2107.08117} for this material. 

The bulk of our work is done in section~\ref{sec:complexesOfHopfAndTref}, where we simplify the colored $\sl(N)$ chain complexes associated to the standard diagrams of the Hopf link and the trefoil which have $2$ and $3$ crossings, respectively. The more crossings there are in a diagram, the more complicated the associated complex is. We use a technical trick involving the homological perturbation lemma to reduce away the complexity of one of the crossings. We thereby obtain complexes for the Hopf link and the trefoil that are only as complicated as complexes associated to $1$ and $2$ crossings, respectively. At this stage, the Hopf link complex has no differential but the trefoil complex still admits further simplification. The remaining simplification that we perform comes from adapting the work of Hogancamp, Rose, and Wedrich \cite{https://doi.org/10.48550/arxiv.2107.08117}, who provide a simplification of the complex associated to a full twist on two strands in the context of colored HOMFLYPT homology. 

In section~\ref{sec:SUNRepSpaces}, we compute the representation spaces associated to the Hopf link and the trefoil. We show that they are disjoint unions of homogeneous spaces $G/K$ and explicitly describe the subgroups $K \subseteq G = \U(N)$ that arise. The same technique gives an explicit description of the representation spaces of the $(2,n)$ torus knots and links. Descriptions for all $2$-bridge knots and links are thereby obtained since the representation spaces of a $2$-bridge knot or link are the same as the representation spaces of the $(2,n)$ torus knot or link with the same determinant. We end the section by proving a slight refinement of Gugenheim--May's result concerning module structures needed for Proposition~\ref{prop:moduleandReducedIsomorphism}. 

Finally in section~\ref{sec:proofsofMainResults}, we identify the homologies of the simplified complexes obtained from section~\ref{sec:complexesOfHopfAndTref} with the cohomologies of the representation spaces computed in section~\ref{sec:SUNRepSpaces}. The simplified complex of the Hopf link has no differential, and the identification is an application of Robert--Wagner's work concerning partial flag manifolds \cite[Section 4.2]{MR4164001}. The simplified complex associated to the trefoil splits as a direct sum of complexes, one for each component of $\sr R_N(K)$. The direct summand corresponding to the component $G/K$ turns out to be isomorphic to the tensor product of $\Z$ with a free resolution of $H^*(\B K)$ as a $H^*(\B G)$-module. Hence, the homology of this complex is $\Tor_{H^*(\B G)}(\Z,H^*(\B K))$ which is isomorphic to $H^*(G/K)$ by Gugenheim--May's result. 

\theoremstyle{definition}
\newtheorem*{ack}{Acknowledgments}
\begin{ack}
	I thank Louis-Hadrien Robert, David Rose, Matt Stoffregen, Joshua Sussan, Emmanuel Wagner, and Michael Willis for many helpful discussions and correspondences. 
	I also thank my advisor Peter Kronheimer for his continued guidance, support, and encouragement. This material is based upon work supported by the NSF GRFP through grant DGE-1745303.
\end{ack}

\section{Generalities on colored \texorpdfstring{$\sl(N)$}{sl(N)} homology}\label{sec:generalities}

In section~\ref{subsec:symmetricPolys}, we review relevant aspects of the basic theory of symmetric polynomials. Our exposition is partly drawn from \cite{MR3877770,https://doi.org/10.48550/arxiv.2107.08117}, and a standard reference for this material is \cite{MR1464693}. In sections~\ref{subsec:websAndFoams} and \ref{subsec:foamEvaluationAndMOYCalc}, we review the definitions of $\sl(N)$ webs and foams and Robert and Wagner's combinatorial closed foam evaluation \cite{MR4164001}. A number of local relations on foams that can be derived from Robert--Wagner's evaluation are given. In section~\ref{subsec:RickardComplexes}, we review Rickard complexes, which are the complexes associated to crossings between labeled strands, and the corresponding results concerning invariance under Reidemeister moves. We also define dot maps and review their basic properties in this section. Finally, in section~\ref{subsec:differentVersionsOfLinkInvt}, we provide the definitions of the different versions of colored $\sl(N)$ homology. 

\subsection{Symmetric polynomials}\label{subsec:symmetricPolys}

A \textit{partition} $\lambda$ is a sequence $\lambda = (\lambda_1,\ldots,\lambda_l)$ of positive integers for which $\lambda_1 \ge \cdots \ge \lambda_l > 0$. The numbers $\lambda_1,\ldots,\lambda_l$ are called the \textit{parts} of $\lambda$, and we say that $\lambda$ is a \textit{partition of} $|\lambda| = \lambda_1 + \cdots + \lambda_l$. By convention, we set $\lambda_j = 0$ if $j > l$. A partition is often visually represented by its \textit{Young diagram}; for example, the Young diagram of the partition $(4,3,1)$ is \[
	\ydiagram{4,3,1}
\]Let $P(l,k)$ denote the set of partitions with at most $l$ parts where each part is at most $k$. Let $\mathrm{box}(l,k) \in P(l,k)$ be the partition with $l$ parts all equal to $k$. If $\lambda \in P(l,k)$, then its \textit{complement} $\lambda^c$ in $P(l,k)$ is the partition whose Young diagram is the complement of $\lambda$ within $\mathrm{box}(l,k)$ rotated by $180^\circ$. Its \textit{transpose} $\bar{\lambda} \in P(k,l)$ is obtained by reflecting its Young diagram along its main diagonal. Finally, we let $\hat{\lambda} \in P(k,l)$ denote the transpose of the complement of $\lambda$. For example \[
	\mathrm{box}(2,4) = \begin{gathered}
		\ydiagram{4,4}
	\end{gathered} \qquad \lambda = \begin{gathered}
		\ydiagram{2,1}
	\end{gathered} \qquad \lambda^c = \begin{gathered}
		\ydiagram{3,2}
	\end{gathered} \qquad \hat{\lambda} = \begin{gathered}
		\ydiagram{2,2,1}
	\end{gathered}
\]

Let $\mathbf{A} = \{X_1,\ldots,X_a\}$ be a finite set, thought of as a collection of indeterminates of a polynomial ring. We refer to such a finite set as an \textit{alphabet}. Let $\Sym(\mathbf{A}) = \Z[X_1,\ldots,X_a]^{\fk{S}_a}$ denote the ring of symmetric polynomials in the variables $X_1,\ldots,X_a$. Given a partition $\lambda$ with at most $a$ parts, the \textit{Schur polynomial} $s_\lambda(\mathbf{A})$ associated to $\lambda$ is the symmetric polynomial given by the formula \[
	s_\lambda(\mathbf{A}) = \frac{\det(X_i^{\lambda_j + a - j})}{\prod_{1 \leq i < j \leq a} (X_i - X_j)}
\]where $\det(X_i^{\lambda_j+a-j})$ is the determinant of the $a \x a$ matrix whose $(i,j)$-entry is $X_i^{\lambda_j + a -j}$. The Schur polynomial $s_\lambda(\mathbf{A})$ can also be described by \textit{semi-standard Young tableaux}, which are ways of filling the Young diagram of $\lambda$ with the numbers $1,\ldots,a$, allowing repetition, so that rows are weakly increasing while columns are strictly increasing. The semi-standard Young tableaux for $\lambda = (2,2,1)$ and $a = 3$ are \[
	\begin{ytableau}
		1 & 1\\
		2 & 2\\
		3
	\end{ytableau}\qquad\begin{ytableau}
		1 & 1\\
		2 & 3\\
		3
	\end{ytableau}\qquad\begin{ytableau}
		1 & 2\\
		2 & 3\\
		3
	\end{ytableau}
\]Associated to each tableau is a monomial $X_1^{n_1}\cdots X_a^{n_a}$ where $n_i$ is the number of times $i$ appears in the tableau. The Schur polynomial is the sum of these monomials over all tableaux; for example, \[
	s_{(2,2,1)}(X_1,X_2,X_3) = X_1^2X_2^2X_3^{\phantom{2}} + X_1^2X_2^{\phantom{2}}X_3^2 + X_1^{\phantom{2}}X_2^2X_3^2.
\]Recall that for $i \ge 0$, the \textit{elementary symmetric polynomials} $e_i(\mathbf{A})$ and the \textit{complete homogeneous symmetric polynomials} $h_i(\mathbf{A})$ are polynomials in $\Sym(\mathbf{A})$ defined by \begin{align*}
	e_i(\mathbf{A}) &\coloneq \text{sum of square-free monomials of degree $i$}\\
	h_i(\mathbf{A}) &\coloneq \text{sum of all monomials of degree $i$.}
\end{align*}The Schur polynomial of a partition whose Young diagram consists of a single column is an elementary symmetric polynomial. Similarly, the Schur polynomial of a partition whose Young diagram consists of a single row is a complete homogeneous symmetric polynomial. 

A basis for the free abelian group of homogeneous symmetric polynomials in $\mathbf{A}$ of degree $d$ is given by the Schur polynomials $s_\lambda(\mathbf{A})$ where $\lambda$ ranges over partitions of $d$ with at most $a$ parts. If $\lambda$ and $\mu$ are partitions with at most $a$ parts, there are unique numbers $c_{\lambda\mu}^\nu \in \Z$ for which \[
	s_\lambda(\mathbf{A})\cdot s_\mu(\mathbf{A}) = \sum_{\nu} c_{\lambda\mu}^\nu s_\nu(\mathbf{A})
\]where the sum is over all partitions $\nu$ with at most $a$ parts for which $|\nu| = |\lambda| + |\mu|$. These coefficients are the \textit{Littlewood--Richardson coefficients} and turn out to be independent of $a$.

We recall \textit{Pieri's formula}, which is essentially a formula for the Littlewood--Richardson coefficients in the special case that the Young diagram of $\lambda$ is either a single row or a single column. If $\mu$ is a partition with at most $a$ parts, then \[
	h_k(\mathbf{A})\cdot s_\mu(\mathbf{A}) = \sum_{\nu} s_\nu(\mathbf{A})
\]where the sum is over all partitions $\nu$ with at most $a$ parts for which \begin{itemize}[noitemsep]
	\item $|\nu| = k + |\mu|$ 
	\item the Young diagram of $\mu$ fits inside of the Young diagram of $\nu$
	\item the \textit{skew} Young diagram $\nu\setminus\mu$, obtained from the Young diagram of $\nu$ by deleting the boxes of the Young diagram of $\mu$, has at most one box in each column. 
\end{itemize}Similarly, $e_k(\mathbf{A})\cdot s_\mu(\mathbf{A}) = \sum_\nu s_\nu(\mathbf{A})$ where the sum is over all partitions $\nu$ with at most $a$ parts whose Young diagrams are obtained by adding $k$ boxes to the Young diagram of $\mu$ so that at most one box is added to each row. 

Now consider two disjoint alphabets $\A,\mathbf{B}$ of sizes $a,b$, respectively. Let $\Sym(\A|\mathbf{B})$ denote the ring of polynomials in the variables $\A \cup \mathbf{B}$ that are symmetric in $\A$ and symmetric in $\mathbf{B}$. Note that there is an isomorphism $\Sym(\A|\mathbf{B}) \cong \Sym(\A) \otimes_\Z \Sym(\mathbf{B})$. Since $\Sym(\A \cup \mathbf{B}) \subseteq \Sym(\A|\mathbf{B})$, any symmetric polynomial in $\A\cup\mathbf{B}$ can be expressed as a $\Z$-linear combination of the products $s_\lambda(\A)\cdot s_{\mu}(\mathbf{B})$ where $\lambda$ and $\mu$ range over partitions with at most $a$ and $b$ parts, respectively. It turns out that if $\nu$ is a partition with at most $a + b$ parts, then \[
	s_\nu(\A \cup \mathbf{B}) = \sum_{\lambda,\mu} c_{\lambda\mu}^\nu s_{\lambda}(\A) \cdot s_{\mu}(\mathbf{B})
\]where $c_{\lambda\mu}^\nu$ are again the Littlewood--Richardson coefficients and the sum is over partitions $\lambda,\mu$ with at most $a,b$ parts, respectively. 

Lastly, we note that the elementary and complete homogeneous symmetric polynomials have generating functions \[
	E(\A,t) = \prod_{X \in \A} (1 + Xt) \eqqcolon \sum_{j=0}^\infty e_j(\A) t^j \qquad\qquad H(\A,t) = \prod_{X \in\A} \frac{1}{1 - Xt} \eqqcolon \sum_{j=0}^\infty h_j(\A) t^j.
\]Simple identities between generating functions lead to useful relations among the symmetric polynomials. For example, if $\A$ and $\mathbf{B}$ are disjoint, then the equality $E(\A\cup \mathbf{B},t)\cdot H(\A,-t) = E(\mathbf{B},t)$ implies the following lemma. 
\begin{lem}\label{lem:generatingFunctionIdentity}
	If $\A$ and $\mathbf{B}$ are disjoint alphabets, then $\sum_{i+j=k}(-1)^j e_i(\A \cup \mathbf{B}) h_j(\mathbf{B}) = e_k(\mathbf{A})$.
\end{lem}

\subsection{Webs and foams}\label{subsec:websAndFoams}

\begin{df}
	A \textit{closed dotted $\sl(N)$ foam $F$ in $\R^3$} consists of: \begin{itemize}[noitemsep]
		\item A compact space $F^2 \subset \R^3$ with compact subsets $F^0 \subset F^1 \subset F^2$ such that \begin{itemize}[noitemsep]
			\item $F^0$ is a finite set of points called \textit{singular points}, 
			\item $F^1\setminus F^0$ is a smoothly embedded $1$-manifold with finitely many components, which are called \textit{bindings},
			\item $F^2\setminus F^1$ is a smoothly embedded $2$-manifold with finitely many components, which are called \textit{facets}.
		\end{itemize}
		\item An orientation of each facet, and an orientation of each binding. 
		\item A label $\ell(f) \in \{1,\ldots,N\}$ for each facet $f$. 
		\item A finite collection of weighted points called \textit{dots} that lie on the (interiors of) facets. The weight $w(d) \in \Z$ of a dot $d$ lying on a facet $f$ must satisfy $1 \leq w(d) \leq \ell(f)$. 
	\end{itemize}We require that \begin{itemize}[noitemsep]
		\item Points on bindings and singular points have neighborhoods in $\R^3$ that intersect $F^2$ in the following local models: \[
			\labellist
			\pinlabel {\Large$\bullet$} at 61.2 87
			\pinlabel {\Large$\bullet$} at 304 55
			\endlabellist
			\begin{gathered}
				\includegraphics[width=.15\textwidth]{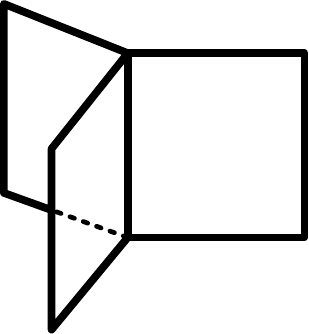}
			\end{gathered} \qquad\qquad \begin{gathered}
				\includegraphics[width=.14\textwidth]{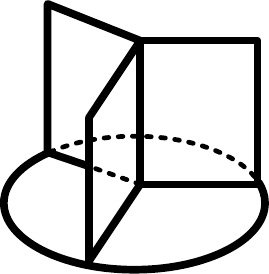}
			\end{gathered}
		\]
		\item The orientation of each binding agrees with the boundary orientations of exactly two of its three adjacent facets. The sum of the labels of those two facets equals the label of the third. 
	\end{itemize}
\end{df}
\begin{rem}\label{rem:decorations}
	Let $F$ be a foam with a facet $f$ labeled $a$. One way to interpret a dot of weight $i$ on $f$ is view it as the $i$th elementary symmetric polynomial $e_i(\mathbf{A})$ where $\mathbf{A}$ is an alphabet of size $a$ associated to $f$. A collection of dots of possibly different weights on $f$ is then thought of as a monomial in the elementary symmetric polynomials. The facet \textit{decorated} by an arbitrary symmetric polynomial $p(\mathbf{A}) \in \Sym(\mathbf{A})$ is to be interpreted as the unique formal linear combination of dotted foams corresponding to the unique expression of $p(\mathbf{A})$ as a polynomial in $e_1(\mathbf{A}),\ldots,e_a(\mathbf{A})$.
\end{rem}

\begin{df}
	A \textit{closed $\sl(N)$ web $W$ in $\R^2$} is an oriented trivalent graph embedded in the plane where multiple edges and circles without vertices are permitted. Each vertex must have both incoming and outgoing edges, and each edge $e$ is given a label $\ell(e) \in \{1,\ldots,N\}$ such that at every vertex, the sum of the labels of the incoming edges equals the sum of the labels of the outgoing edges. 
\end{df}

\begin{rem}
	We use the convention that labels of facets of foams must be positive. A foam having a facet labeled zero is understood to mean the same foam with that facet deleted. Similarly, dots of weight zero and edges of webs labeled zero may be erased. This convention agrees with \cite{MR3877770} but disagrees with \cite{MR4164001} but is ultimately inconsequential. 
\end{rem}

A closed $\sl(N)$ foam $F$ in $\R^3$ is said to transversely intersect a properly embedded surface $P \subset \R^3$ if $P$ is disjoint from all of the singular points and dots of $F$ and transversely intersects each binding and facet. If $P$ is the plane $0 \x \R^2$, then the transverse intersection $W = F \cap P$ is a closed $\sl(N)$ web. The labels on the edges of $W$ are inherited from the labels on the facets of $F$, and orientations of the edges of $W$ are induced from those of the facets of $F$ using the standard orientation of $P$. Similarly, a closed $\sl(N)$ web $W$ in $\R^2$ is said to transversely intersect a given embedded circle $C$ if $C$ is disjoint from the vertices of $W$ and transversely intersects the edges.

\begin{df}
	Let $S \subset \R^2$ be a closed subsurface, which we usually take to be disc.
	An \textit{$\sl(N)$ web $V$} in $S$ is the intersection of a closed web $W \subset \R^2$ with $S$ such that $W$ and $\partial S$ intersect transversely. Its \textit{boundary} $\partial V$ is the intersection $V \cap \partial S$ viewed as a collection of oriented labeled points on $\partial S$, where the label of a point is just the label of the corresponding edge of $V$ and where the orientation of the point indicates whether the edge is incoming or outgoing. A collection $\beta$ of oriented labeled points on $\partial S$ is called a \textit{web-boundary} if for each component $C$ of $\partial S$, the sum of the labels of the positively-oriented points on $C$ equals the sum of the labels of the negatively-oriented points on $C$. 

	Suppose $V_0$ and $V_1$ are webs in $S$ having the same boundary $\beta = \partial V_0 = \partial V_1$. Let $\ol{V}_0$ denote the web obtained from $V_0$ by reversing the orientations of its edges. An \textit{$\sl(N)$ foam $G$} from $V_0$ to $V_1$ is the intersection of a closed $\sl(N)$ foam $F \subset \R^3$ with $[0,1] \x S \subset \R \x \R^2$, subject to the requirement that $F$ and $\partial ([0,1] \x S)$ intersect transversely and this intersection consists of \[
		F \cap (0 \x S) = 0 \x \ol{V}_0 \qquad F \cap ([0,1] \x \partial S) = [0,1] \x \beta \qquad F \cap (1 \x S) = 1 \x V_1. 
	\]
	Facets, bindings, singular points, labels, and dots are all defined in the same way as in the closed case. We identify foams that are isotopic rel boundary. 
\end{df}

Suppose $V_0,V_1,V_2$ are webs in a surface $S$ all having the same web-boundary $\beta$. If $G$ is an $\sl(N)$ foam from $V_0$ to $V_1$ and $H$ is an $\sl(N)$ foam from $V_1$ to $V_2$, then they may be glued together to give an $\sl(N)$ foam $G \cup_{V_1} H$ from $V_0$ to $V_1$. This version of composition is called \textit{vertical composition}, and is obtained by stacking in the ordinary way for cobordisms. 
There is also a notion of \textit{horizontal composition}. Suppose $S$ and $S'$ are subsurfaces of $\R^2$ that intersect along boundary components and whose interiors are disjoint. Then $S \cup S'$ is another subsurface of $\R^2$. If $V \subset S$ and $V' \subset S'$ are webs such that their web-boundaries agree in an orientation-reversing way, then they glue to a web $V \cup V' \subset S \cup S'$. Foams in $[0,1] \x S$ and $[0,1] \x S'$ with compatible boundaries can also be glued. Bar-Natan's canopolies formalism encodes this structure \cite{MR2174270}. See \cite[Section 2.2]{MR3877770} for more details of this structure in the context of webs and foams. 

Finally, we define the degree of an $\sl(N)$ foam from $V_0$ to $V_1$ based on the definitions given in \cite{MR3545951,MR4164001,MR3877770}. 

\begin{df}\label{df:degreeOfFoam}
	Suppose $V_0$ and $V_1$ are $\sl(N)$ webs with the same boundary $\beta$, and let $G$ be an $\sl(N)$ foam from $V_0$ to $V_1$. We define the \textit{degree} of $G$ to be \[
		\deg(G) = - \sum_{\text{facets }f} d(f) + \sum_{\substack{\text{interval}\\\text{bindings }i}} d(i) - \sum_{\substack{\text{singular}\\\text{points }p}} d(p) + \sum_{q \in \beta} \frac{\ell(q)(N - \ell(q))}{2} + \sum_{\text{dots }d} 2w(d)
	\]where \begin{itemize}[noitemsep]
		\item if $f$ is a facet with $\ell(f) = a$, then $d(f) = a(N - a)\chi(f)$ where $\chi$ is the Euler characteristic,
		\item if $i$ is a binding diffeomorphic to an interval (as opposed to a circle) whose adjacent facets are labeled $a,b$, and $a + b$, then $d(i) = ab + (a + b)(N - a - b)$,
		\item if $p$ is a singular point whose adjacent facets are labeled $a,b,c,a + b, b + c, a + b + c$, then $d(p) = ab + bc + ac + (a + b + c)(N - a - b - c)$, 
		\item if $q$ is a point in the web-boundary $\beta = \partial V_0 = \partial V_1$, then $\ell(q)$ is the label of $q$,
		\item if $d$ is a dot, then $w(d)$ is its weight.
	\end{itemize}
\end{df}

\begin{rem}\label{rem:bendingTrick}
	The degree of $\sl(N)$ foams is additive under vertical and horizontal composition. The \textit{identity $\sl(N)$ foam} of an $\sl(N)$ web $V$ is $[0,1]\x V \subset [0,1] \x S$, viewed as an $\sl(N)$ foam from $V$ to $V$. The degree of the identity foam is zero.

	Let $G \subset [0,1] \x D$ be a foam between webs $V_0$ and $V_1$ in a disc $D$. By an isotopy of $G$ keeping $\partial G$ within the boundary of the $3$-ball $[0,1] \x D$, we may slide all of $\partial G$ into $1 \x D$. The result is a foam $G'$ from $\emp$ to the closed web $\ol{V}_0 \cup V_1$. If $\beta = \partial V_0 = \partial V_1$, then \[
		\deg(G') = \deg(G) - \sum_{q\in\beta} \frac{\ell(q)(N - \ell(q))}2.
	\]This procedure is called the \textit{bending trick} in \cite{MR3877770}.
\end{rem}

\subsection{Foam evaluation and MOY calculus}\label{subsec:foamEvaluationAndMOYCalc}

Robert and Wagner \cite{MR4164001} define an explicit combinatorial evaluation \[
	\langle F \rangle \in \Z[X_1,\ldots,X_N]^{\fk{S}_N} = \Sym(N)
\]for closed $\sl(N)$ foams $F \subset \R^3$. The evaluation takes the form of a homogeneous symmetric polynomial in $N$ variables. The variables $X_1,\ldots,X_N$ are defined to be of degree $2$, and with this definition, the degree of $\langle F \rangle$ is the degree of $F$ viewed as a foam from $\emp$ to $\emp$ given in Definition~\ref{df:degreeOfFoam}. 

Let $V_0$ and $V_1$ be $\sl(N)$ webs having the same boundary. Consider the free $\Sym(N)$-module $\smash{\bigoplus_G \Sym(N)} \cdot G$ with basis the set of all foams $G$ from $V_0$ to $V_1$. Any such foam $G$ can be viewed as a foam from $\emp$ to $\ol{V}_0 \cup V_1$, so if $H$ is a foam from $\ol{V}_0 \cup V_1$ to $\emp$, then $G \cup_{\ol{V}_0 \cup V_1} H$ is a closed $\sl(N)$ foam. Define \[
	\Hom^*(V_0,V_1) \coloneq \left(\bigoplus_G \Sym(N) \cdot G\right) \bigg/\!\sim
\]where a finite sum $\sum_i a_i G_i$ is declared to be equivalent to zero if $\sum_i a_i \:\langle G_i \smash{\cup_{\ol{V}_0 \cup V_1}} H\rangle = 0$ for every foam $H$ from $V_0 \cup \ol{V}_1$ to $\emp$. This definition is a version of the universal construction \cite{MR1362791}. Also see \cite[Section 2]{MR3877770}. Our notation $\Hom^*(V_0,V_1)$ indicates that foams $G$ of all degrees are considered; in particular, there is $\Z$-grading \[
	\Hom^*(V_0,V_1) = \bigoplus_{k \in \Z} \Hom^k(V_0,V_1)
\]where a foam $G$ of degree $k$ lies in $\Hom^k(V_0,V_1)$ and the variables $X_1,\ldots,X_N$ are homogeneous of degree $2$. 

\begin{df}\label{df:categoryCb}
	Let $\beta$ be a web-boundary, and let $\sr C(\beta)$ be the following category. For each web $V$ with $\partial V = \beta$ there is a family of objects $q^i V \in \sr C(\beta)$ for $i \in \Z$, called \textit{$q$-shifts} of $V$. An arbitrary object of $\sr C(\beta)$ is a finite formal direct sum of $q$-shifts of webs with boundary $\beta$. In particular, there is a zero object given by the empty direct sum. The morphism space between $q$-shifts of webs is given by \[
		\Hom(q^iV_0,q^jV_1) = \Hom^{i-j}(V_0,V_1)
	\]while a morphism between finite formal direct sums is given by a matrix of morphisms between $q$-shifts of webs. Composition in $\sr C(\beta)$ is given by vertical composition of foams. We let $\sr C(\emp)$ denote the case of closed webs in the plane where $\beta = \emp$. 

	If $W$ is a web and $P = \sum_i a_i q^i$ is a Laurent polynomial with $a_i \in \Z_{\ge 0}$, we let $P W$ denote the formal direct sum $PW\coloneq \bigoplus_i \smash{\left(q^i W \right)^{a_i}}$ viewed as an element in $\sr C(\partial W)$. As another matter of notation, if $W$ is a web with an edge labeled by a number greater than $N$, then we let $q^iW$ denote the zero object in $\sr C(\partial W)$. 
\end{df}

Horizontal composition gives rise to bifunctors on these categories. We refer to \cite[Section 2.2]{MR3877770} for the details of canopolies of webs and foams. 
Since $\sr C(\beta)$ is an additive category, we may consider chain complexes in $\sr C(\beta)$. A chain complex in $\sr C(\beta)$ is a sequence $\{A_i\}_{i\in\Z}$ of objects in $\sr C(\beta)$ with maps $d\colon A_i \to A_{i+1}$ satisfying $d\circ d = 0$. 
If $A$ is a chain complex, we let $h^iA$ denote the complex $(h^iA)_j = A_{j - i}$ with the same differential. 
An object $X \in \sr C(\beta)$ may be viewed as a chain complex supported in homological grading $0$ with trivial differential. 
All chain complexes we consider are bounded, which is to say that $A_i = 0$ for all but finitely many $i \in \Z$. Chain complexes may also be horizontally composed by combining the ordinary tensor product of complex with horizontal composition of foams. Since the differential on a tensor product involves negating the differential on one of the tensor factors, an ordering of the tensor factors is required to define horizontal composition. The resulting complex up to chain isomorphism is independent of this choice of ordering. See \cite{MR3877770} for more details. 

If $W$ is a closed $\sl(N)$ web, let $\langle W\rangle \in \Z[q,q^{-1}]$ be its MOY polynomial \cite{MR1659228}. This polynomial has nonnegative coefficients, is invariant under $q \mapsto q^{-1}$, and can be computed recursively from a list of relations. See for example \cite[Definition 3.2]{MR4164001} for the list. We also note that if $\ol{W}$ is obtained by reversing the orientations of all edges, then $\langle \ol{W} \rangle = \langle W \rangle$. We record in Figure~\ref{fig:MOYcalculus} the MOY calculus relations that we will use in this paper. Recall that the \textit{quantum binomial coefficient} $\qbinom{n}{k} \in \Z[q,q^{-1}]$ is $[n]!/([k]![n-k]!)$ if $0 \leq k \leq n$ and is zero otherwise where $[k]! \coloneq [k][k-1]\cdots[2][1]$ and $[m]\coloneq (q^m - q^{-m})/(q - q^{-1})$.

\begin{figure}[!ht]
	\vspace{5pt}
	\begin{equation}\label{eq:forkRelation}
		\begin{gathered}
			\centering
			\labellist
			\pinlabel {\small$a$} at 0 78
			\pinlabel {\small$b$} at 22 79
			\pinlabel {\small$c$} at 44 78
			\pinlabel {\small$a+b+c$} at 51 5
			\endlabellist
			\includegraphics[width=.072\textwidth]{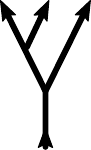}
		\end{gathered}\:\:\:\: = \:\:\:\begin{gathered}
			\centering
			\labellist
			\pinlabel {\small$a$} at 0 78
			\pinlabel {\small$b$} at 22 79
			\pinlabel {\small$c$} at 44 78
			\pinlabel {\small$a+b+c$} at 51 5
			\endlabellist
			\includegraphics[width=.072\textwidth]{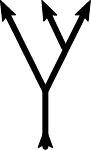}
		\end{gathered} 
	\end{equation}\vspace{5pt}\begin{equation}\label{eq:bigonRelation}
		\begin{gathered}
			\centering
		\labellist
		\pinlabel {\small$a$} at -4 35
		\pinlabel {\small$b$} at 29 36
		\pinlabel {\small${a+b}$} at 31 5
		\endlabellist
		\includegraphics[width=.037\textwidth]{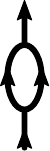}
		\end{gathered}\quad = \qbinom{a+b}{a}\:\: \begin{gathered}
			\centering
			\labellist
			\pinlabel {\small${a+b}$} at 31 5
			\endlabellist
			\includegraphics[width=.037\textwidth]{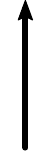}
		\end{gathered} \hspace{90pt}
		\begin{gathered}
			\centering
			\labellist
			\pinlabel {\small${a+b}$} at -15 35
			\pinlabel {\small$b$} at 29 35
			\pinlabel {\small$a$} at 21 4
			\endlabellist
			\includegraphics[width=.037\textwidth]{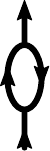}
		\end{gathered} \:\:\:\: = \qbinom{N - a}{b} \:\: \begin{gathered}
			\centering
			\labellist
			\pinlabel {\small${a}$} at 21 4
			\endlabellist
			\includegraphics[width=.037\textwidth]{verticalStrand}
		\end{gathered} 
	\end{equation}\vspace{5pt}\begin{equation}\label{eq:complicatedRelation}
		\begin{gathered}
			\centering
			\labellist
			\pinlabel {\small$a+b-k$} at -12.5 35.5
			\pinlabel {\small$k$} at 25 36
			\pinlabel {\small$a$} at 25 19
			\pinlabel {\small${a-k}$} at 11 32.5
			\pinlabel {\small$b$} at -3 19
			\pinlabel {\small${a-l}$} at 11 16
			\pinlabel {\small$l$} at 25 2.5
			\pinlabel {\small$a+b-l$} at -12.5 3
			\endlabellist
			\includegraphics[width=.07\textwidth]{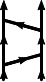}
		\end{gathered} \quad = \sum_{j=\max(a+b-N,0)}^{\min(k,l)} \qbinom{a-k-l+b}{a-k-l+j}\qquad\qquad \begin{gathered}
			\centering
			\labellist
			\pinlabel {\small$a+b-k$} at -12.5 35.5
			\pinlabel {\small$k$} at 25 36
			\pinlabel {\small$j$} at 25 19
			\pinlabel {\small${k-j}$} at 11 32.5
			\pinlabel {\small$a+b-j$} at -12.5 19
			\pinlabel {\small${l-j}$} at 11 16
			\pinlabel {\small$l$} at 25 2.5
			\pinlabel {\small$a+b-l$} at -12.5 3
			\endlabellist
			\includegraphics[width=.07\textwidth]{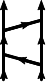}
		\end{gathered}
	\end{equation}
	\vspace{-10pt}
	\captionsetup{width=.8\linewidth}
	\caption{Some relations in MOY calculus. The same relations hold if all pictures are mirrored or if all orientations are reversed. 
	}
	\vspace{-5pt}
	\label{fig:MOYcalculus}
\end{figure}

\begin{thm}[{\cite[Theorem 3.30]{MR4164001}}]\label{thm:RWcategorificationOfMOYCalculus}
	Let $W$ be a closed $\sl(N)$ web. Then $\Hom^*(\emp,W)$ is a finitely generated free $\Sym(N)$-module, and its graded rank coincides with the MOY polynomial of $W$.
	In particular, there are isomorphisms between the two sides of each MOY calculus relation in Figure~\ref{fig:MOYcalculus}, viewed as objects in $\sr C(\beta)$ where $\beta$ depends on the relation in question. 
\end{thm}

\begin{rem}
	In fact, Robert and Wagner show that there is a homogeneous basis $b_1,\ldots,b_m$ for $\Hom^*(\emp,W)$ as a $\Sym(N)$-module where each $b_i$ is a $\Z$-linear combination of foams (rather than simply a $\Sym(N)$-linear combination of foams). The number of the basis elements of degree $k$ is precisely the coefficient of $q^k$ in $\langle W\rangle$. 
\end{rem}

We record some local relations for foams. These are relations on morphisms in $\sr C(\beta)$ where the web-boundary $\beta$ depends on the given relation. Most of these relations are drawn from \cite[Proposition 3.38]{MR4164001}, though some are drawn from \cite{MR3545951} and require checking that they hold equivariantly, which is done in \cite[Proposition 2.20]{MR3877770}. 

\begin{prop}\label{prop:foamRelations}
	\leavevmode\begin{enumerate}
		\item\label{item:sphereRelation} Sphere relations. If $\lambda$ is a partition in $P(a,N-a)$, then\[
			\begin{gathered}
				\centering
				\labellist
				\pinlabel {$s_\lambda$} at 37 55
				\pinlabel {\small$a$} at 6 6
				\endlabellist
				\includegraphics[width=.11\textwidth]{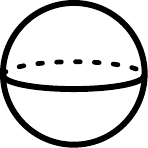}
			\end{gathered}\:\: = \begin{cases}
				(-1)^{a(a+1)/2} & \lambda = \mathrm{box}(a,N-a) \\
				0 & \text{else}
			\end{cases}
		\]where $s_\lambda$ is interpreted as linear combinations of collections of dots as in Remark~\ref{rem:decorations}. 
		\item\label{item:thicknilHecke} Thick nilHecke relations. \[
			\begin{gathered}
				\vspace{-3pt}
				\labellist
				\pinlabel {\small$a$} at 23 51
				\pinlabel {\small$b$} at 23 73
				\pinlabel {\small$a+b$} at -9 57
				\endlabellist
				\includegraphics[width=.13\textwidth]{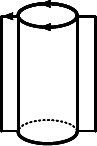}
			\end{gathered}\:\: = \sum_{\lambda \in P(a,b)} (-1)^{|\hat{\lambda}|} \:\: \begin{gathered}
				\vspace{-3pt}
				\labellist
				\pinlabel {$s_\lambda$} at 23 6
				\pinlabel {$s_{\hat{\lambda}}$} at 23 63
				\endlabellist
				\includegraphics[width=.13\textwidth]{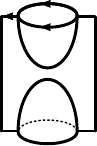}
			\end{gathered}
		\]
		\item Blister relations. If $\lambda \in P(a,b)$ and $\mu \in P(b,a)$, then \[
			(-1)^{|\mu|} \:\:\begin{gathered}
				\vspace{-3pt}
				\labellist
				\pinlabel {\small$a$} at 30 21
				\pinlabel {$s_\lambda$} at 21 20
				\pinlabel {\small$b$} at 16 40
				\pinlabel {$s_\mu$} at 25 41
				\pinlabel {\small$a+b$} at -9 57
				\endlabellist
				\includegraphics[width=.13\textwidth]{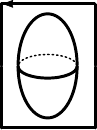}
			\end{gathered}\:\: = \delta_{\mu,\hat{\lambda}} \:\: \begin{gathered}
					\centering
					\includegraphics[width=.13\textwidth]{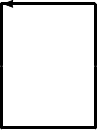}
				\end{gathered} 
		\]where $\delta_{\mu,\hat{\lambda}}$ is $1$ if $\mu = \hat{\lambda}$ and is $0$ otherwise. 
		\item\label{item:matveevPiergalini} Matveev--Piergalini relations. \[
			\begin{gathered}
				\centering
				\labellist
				\pinlabel {\small$a$} at 30 54
				\pinlabel {\small$b$} at -3 62
				\pinlabel {\small$c$} at 4 70
				\pinlabel {\small$a+b+c$} at 60 62
				\endlabellist
				\includegraphics[width=.13\textwidth]{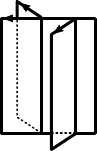}
			\end{gathered}\qquad = \qquad\begin{gathered}
				\centering
				\includegraphics[width=.13\textwidth]{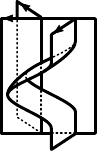}
			\end{gathered}
		\]
		\item\label{item:dotmigrationrelation} Dot-migration relations. If $\nu$ is a partition with at most $a + b$ parts, then \[
			\begin{gathered}
			\centering
			\labellist
			\pinlabel {\small$a+b$} at 123 146
			\pinlabel {\small$a$} at 43 10
			\pinlabel {\small$b$} at -7 150
			\pinlabel {$s_\nu$} at 100 90
			\endlabellist
			\includegraphics[width=.16\textwidth]{binding}
		\end{gathered}\:\: = \sum_{\lambda,\mu} c_{\lambda\mu}^\nu \:\: \begin{gathered}
			\centering
			\labellist
			\pinlabel {$s_\lambda$} at 45 75
			\pinlabel {$s_\mu$} at 25 120
			\endlabellist
			\includegraphics[width=.16\textwidth]{binding}
		\end{gathered}
		\]where the sum is over partitions $\lambda,\mu$ where $\lambda$ has at most $a$ parts and $\mu$ has at most $b$ parts. The numbers $c_{\lambda\mu}^\nu$ are the Littlewood--Richardson coefficients. 
		\item\label{item:foamForkRelation} The following relation holds: \vspace{5pt} \[
			\begin{gathered}
				\centering
				\labellist
				\pinlabel {\small$a + x + y + z$} at 15 90
				\pinlabel {\small$a$} at 121 90
				\pinlabel {\small$z$} at 91 74
				\pinlabel {\small$y$} at 70 73
				\pinlabel {\small$x$} at 49 74
				\pinlabel {\small$b + x + y + z$} at 97 -5
				\pinlabel {\small$b$} at 5 -4
				\endlabellist
				\includegraphics[width=.33\textwidth]{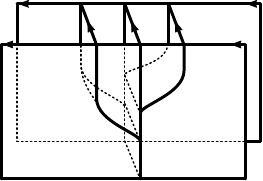}
			\end{gathered} \quad = \quad \begin{gathered}
				\includegraphics[width=.33\textwidth]{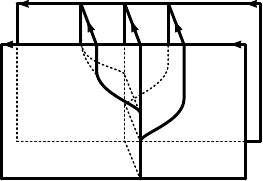}	
			\end{gathered}
		\]\vspace{0pt}
		\item\label{item:twoRungRelation} The following relation holds: \vspace{5pt}\[
			\begin{gathered}
				\centering
				\labellist
				\pinlabel {\small$b + x + y$} at 106 -5
				\pinlabel {\small$b$} at 5 -4
				\pinlabel {\small$y$} at 91 73
				\pinlabel {\small$x$} at 49 74
				\pinlabel {\small$a$} at 121 90
				\pinlabel {\small$a+x+y$} at 25 90
				\endlabellist
				\includegraphics[width=.33\textwidth]{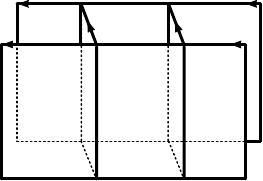}
			\end{gathered} \:\: = \hspace{-6pt}\sum_{\lambda \in P(x,y)}\hspace{-5pt} (-1)^{|\hat{\lambda}|} \:\:\:\: \begin{gathered}
				\centering
				\labellist
				\pinlabel {$s_\lambda$} at 52 30
				\pinlabel {$s_{\hat{\lambda}}$} at 75 53
				\endlabellist
				\includegraphics[width=.33\textwidth]{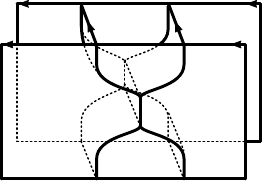}
			\end{gathered}
		\]\vspace{0pt}
		\item\label{item:oneRungRelation} If $\lambda \in P(x,y)$ and $\mu \in P(y,x)$, then \vspace{5pt} \[
			(-1)^{|\mu|}\:\:\begin{gathered}
				\centering
				\labellist
				\pinlabel {\small$b + x + y$} at 106 -5
				\pinlabel {\small$b$} at 5 -4
				\pinlabel {\small$x + y$} at 75 74
				\pinlabel {\small$x$} at 35 44
				\pinlabel {\small$y$} at 90 44
				\pinlabel {$s_\lambda$} at 52 52
				\pinlabel {$s_{\mu}$} at 75 33
				\pinlabel {\small$a$} at 121 90
				\pinlabel {\small$a+x+y$} at 25 90
				\endlabellist
			 	\includegraphics[width=.33\textwidth]{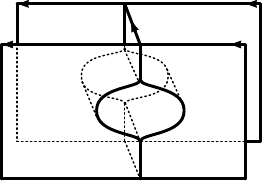}
			 \end{gathered}\:\: = \delta_{\mu,\hat{\lambda}} \:\: \begin{gathered}
				\includegraphics[width=.33\textwidth]{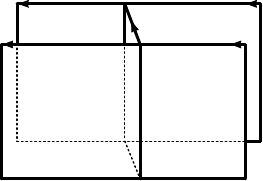}
			\end{gathered}
		\]where $\delta_{\mu,\hat{\lambda}}$ is $1$ if $\mu = \hat{\lambda}$ and is $0$ otherwise.
	\end{enumerate}
\end{prop}

\subsection{Rickard complexes}\label{subsec:RickardComplexes}

Let $a,b,k$ be integers satisfying $0 \leq a,b \leq N$ and $\max(a + b - N,0) \leq k \leq \min(a,b)$, and consider the web $W_k \coloneq W_k(a,b)$ and the foam $w_k\colon W_k \to W_{k+1}$ given in Figure~\ref{fig:standardweb}. Note that $\partial W_k(a,b)$ is independent of $k$. We denote this common web-boundary $\beta(a,b)$. 

\begin{figure}[!ht]
	\centering
	$W_k\coloneq \qquad\quad\:\:\:\:\begin{gathered}
		\labellist
		\pinlabel {\small$a$} at -3 35.5
		\pinlabel {\small$b$} at 25 36
		\pinlabel {\small$k$} at 25 19
		\pinlabel {\small${b - k}$} at 11 33
		\pinlabel {\small${a + b - k}$} at -12.5 19
		\pinlabel {\small${a - k}$} at 11 16.5
		\pinlabel {\small$a$} at 25 2.5
		\pinlabel {\small$b$} at -3 3
		\endlabellist
		\includegraphics[width=.07\textwidth]{standardWeb}
	\end{gathered}$
	\hspace{40pt}
	$w_k\coloneq \quad \begin{gathered}
		\labellist
		\pinlabel {\small$b$} at 13 79
		\pinlabel {\small$a$} at 120 78
		\pinlabel {\small${k+1}$} at 66 79
		\pinlabel {\small$1$} at 66 43
		\pinlabel {\small$a$} at 6 6
		\pinlabel {\small$b$} at 113 7
		\pinlabel {\small${a + b - k}$} at 66 7
		\endlabellist
		\includegraphics[width=.38\textwidth]{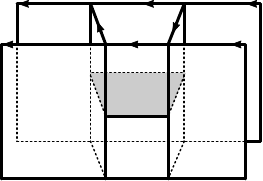}
	\end{gathered}$
	\captionsetup{width=.8\linewidth}
	\caption{Here $0 \leq a,b \leq N$ and $\max(a + b - N,0) \leq k \leq \min(a,b)$. On the left is the web $W_k\coloneq W_k(a,b)$ and on the right is the foam $w_k$ from $qW_k$ to $W_{k+1}$. The web $\ol{W_k}$ is at the bottom of the picture while $W_{k+1}$ is at the top. The remaining labels and orientations of the facets of $w_k$ are uniquely determined by the given data.}
	\vspace{-10pt}
	\label{fig:standardweb}
\end{figure}

\begin{df}\label{df:rickardComplex}
	The \textit{Rickard complex} associated to a positive crossing is the chain complex \[
		\left\llbracket \quad\begin{gathered}
			\vspace{-4pt}
			\labellist
			\pinlabel {\small$b$} at -4 4
			\pinlabel {\small$a$} at 39 3
			\endlabellist
			\includegraphics[width=.05\textwidth]{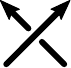}
		\end{gathered}\quad\right\rrbracket \coloneq \left(
		\begin{tikzcd}
			\cdots \ar[r] & h^kq^{-k} W_k \ar[r,"w_k"] & h^{k+1}q^{-(k+1)}W_{k+1} \ar[r] & \cdots
		\end{tikzcd}\right)
	\]where $W_k$ and $w_k$ are given in Figure~\ref{fig:standardweb}. This is a complex in $\sr C(\beta(a,b))$ supported in homological degrees $k$ where $\max(a + b - N,0) \leq k \leq \min(a,b)$. The Rickard complex associated to a negative crossing is the complex \[
		\left\llbracket \quad \begin{gathered}
			\vspace{-4pt}
			\labellist
			\pinlabel {\small$b$} at -4 4
			\pinlabel {\small$a$} at 39 3
			\endlabellist
			\includegraphics[width=.05\textwidth]{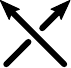}
		\end{gathered} \quad\right\rrbracket\coloneq \left( \begin{tikzcd}
			\cdots \ar[r] & h^{-(k+1)}q^{k+1} W_{k+1} \ar[r,"\ol{w}_k"] & h^{-k}q^k W_k \ar[r] & \cdots
		\end{tikzcd}\right)
	\]where $\ol{w}_k$ is the foam from $W_{k+1}$ to $W_k$ obtained by reflecting $w_k$ across a horizontal plane. This complex is supported in homological degrees $-k$ where $\max(a + b - N,0) \leq k \leq \min(a,b)$.
	We note that our grading and sign conventions match those of Wu \cite{MR3234803} but disagree with those of \cite{MR3877770}. 
\end{df}

Let $D$ be an oriented tangle diagram in a disc where strands are labeled by integers $0 \leq a \leq N$. The boundary of $D$ is naturally a web-boundary. The standard cube of resolutions construction using Rickard complexes to resolve the crossings results in a chain complex $\llbracket D\rrbracket$ in $\sr C(\partial D)$. The construction of $\llbracket D\rrbracket$ requires an arbitrary ordering of the crossings but its chain isomorphism type is independent of the ordering. If $D$ is a diagram in a disc that contains a mix of both crossings and trivalent vertices, then the same procedure may be applied to produce a complex $\llbracket D \rrbracket$ in $\sr C(\partial D)$. 

We collect a number of results concerning the behavior of $\llbracket D \rrbracket$ under local moves. They were proved by Wu \cite{MR3234803,MR2885476} in the setting of matrix factorizations over $\Q$ and were shown to hold nonequivariantly for foams over $\Z$ in \cite{MR3545951}. See \cite[Theorem 3.5]{MR3877770} for the equivariant version for foams over $\Z$. A proof of these relations in a somewhat different context is given in detail in \cite{https://doi.org/10.48550/arxiv.2111.13195}. 

\begin{thm}[{\cite[Theorem 3.5]{MR3877770}}]\label{thm:RmovesandForkMoves}
	\leavevmode
	\begin{enumerate}
		\item Reidemeister moves. Up to homotopy equivalence, $\llbracket - \rrbracket$ is invariant under the Reidemeister II and III moves. For the Reidemeister I move, there are homotopy equivalences \[
		h^a q^{-a(N-a+1)}\left\llbracket \:\:\:\begin{gathered}
			\vspace{-4pt}
			\centering
			\labellist
			\pinlabel {\small$a$} at -5 28
			\endlabellist
			\includegraphics[width=.07\textwidth]{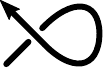}
		\end{gathered}\: \right\rrbracket \:\:\simeq\:\: \left\llbracket \:\:\:\begin{gathered}
			\vspace{-4pt}
			\centering
			\labellist
			\pinlabel {\small$a$} at -5 28
			\endlabellist
			\includegraphics[width=.022\textwidth]{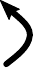}
		\end{gathered}\: \right\rrbracket \:\:\simeq\:\: h^{-a}q^{a(N - a + 1)} \left\llbracket \:\:\:\begin{gathered}
			\vspace{-4pt}
			\centering
			\labellist
			\pinlabel {\small$a$} at -5 28
			\endlabellist
			\includegraphics[width=.07\textwidth]{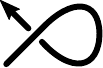}
		\end{gathered}\: \right\rrbracket\vspace{3pt}
	\]
		\item Fork sliding. There is a homotopy equivalence \[
		\left\llbracket\:\:\:\: \begin{gathered}
			\vspace{-3pt}
			\centering
			\labellist
			\pinlabel {\small$a$} at 1 63
			\pinlabel {\small$b$} at 44 65
			\pinlabel {\small$c$} at -5 34
			\pinlabel {\small${a+b}$} at 4 6
			\endlabellist
			\includegraphics[width=.07\textwidth]{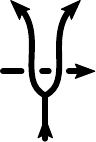}
		\end{gathered}\: \right\rrbracket \:\:\simeq\:\: \left\llbracket\:\:\:\: \begin{gathered}
			\vspace{-3pt}
			\centering
			\labellist
			\pinlabel {\small$a$} at 1 63
			\pinlabel {\small$b$} at 44 65
			\pinlabel {\small$c$} at -5 34
			\pinlabel {\small${a+b}$} at 4 6
			\endlabellist
			\includegraphics[width=.07\textwidth]{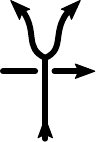}
		\end{gathered}\: \right\rrbracket
	\]There are also equivalences for the other variations of a vertex sliding over or under a strand. 
		\item Fork twisting. There are homotopy equivalences \[
		q^{ab} \left\llbracket \quad\:\begin{gathered}
			\vspace{-3pt}
			\centering
			\labellist
			\pinlabel {\small$a$} at -4 63
			\pinlabel {\small$b$} at 39 65
			\pinlabel {\small${a+b}$} at -2 6
			\endlabellist
			\includegraphics[width=.05\textwidth]{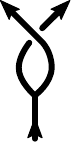}
		\end{gathered}\:\:\:\: \right\rrbracket \:\:\simeq\:\: \left\llbracket \quad\:\begin{gathered}
			\vspace{-3pt}
			\centering
			\labellist
			\pinlabel {\small$a$} at -4 63
			\pinlabel {\small$b$} at 39 65
			\pinlabel {\small${a+b}$} at -2 6
			\endlabellist
			\includegraphics[width=.05\textwidth]{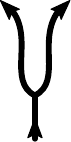}
		\end{gathered}\:\:\:\: \right\rrbracket \:\:\simeq\:\: q^{-ab}\left\llbracket \quad\:\begin{gathered}
			\vspace{-3pt}
			\centering
			\labellist
			\pinlabel {\small$a$} at -4 63
			\pinlabel {\small$b$} at 39 65
			\pinlabel {\small${a+b}$} at -2 6
			\endlabellist
			\includegraphics[width=.05\textwidth]{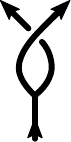}
		\end{gathered}\:\:\:\: \right\rrbracket
	\]There are also equivalences when the orientations of all the edges in these webs are reversed. 
	\end{enumerate}
\end{thm}

\begin{rem}
	Another common normalization of the gradings in the Rickard complex makes $\llbracket -\rrbracket$ invariant under Reidemeister I moves at the cost of less natural grading shifts for the fork sliding and fork twisting equivalences. See for example \cite{MR3877770}. 
\end{rem}

\begin{df}
	Let $W$ be an $\sl(N)$ web with an edge $e$ labeled $a$. If $p$ is a symmetric polynomial in $a$ variables, then the \textit{dot map} on $e$ associated to $p$, which is represented as \[
		\begin{gathered}
			\centering
			\labellist
			\pinlabel {\large$\bullet$} at 4 13
			\pinlabel {\small$p$} at 14 13
			\endlabellist
			\includegraphics[width=.011\textwidth]{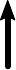}
		\end{gathered}
	\]is defined to be the identity foam of $W$ where the facet corresponding to $e$ is decorated by $p$. This is a map of degree $\deg(p)$ where the variables of $p$ are of degree $2$. A foam with a facet decorated by symmetric polynomial is a linear combination of dotted foams as explained in Remark~\ref{rem:decorations}. 
\end{df}

\begin{rem}\label{rem:centralDotMaps}
	Suppose $V$ and $W$ are webs with the same boundary, and $F$ is a foam from $V$ to $W$. Fix a point $q \in \partial V = \partial W$, and let $e$ and $f$ be the edges of $V$ and $W$, respectively, that are incident to $q$. Then the map given by $F$ from $V$ to $W$ intertwines the dot maps associated to $e$ and $f$. In particular, the dot maps associated to $e$ are central in the algebra of endomorphisms of $V$. 

	Let $D$ be an oriented tangle diagram in a disc with labeled strands. Given a basepoint $r$ on an arc labeled $a$ of $D$ away from crossings, any symmetric polynomial $p$ in $a$ variables defines a chain map $q^{\deg(p)}\llbracket D \rrbracket \to \llbracket D \rrbracket$ which is given on each complete resolution by the dot map on the edge containing $r$. A chain map of this sort is also called a dot map. 
\end{rem}

The following result is sometimes called \textit{dot sliding}, see \cite[Section 5.1]{MR3590355}.

\begin{prop}\label{prop:dotSliding}
	Let $p$ and $q$ be symmetric polynomials in $a$ and $b$ variables, respectively. Then the following dot maps are chain homotopic: \[
		\left\llbracket \:\:\:\begin{gathered}
			\vspace{-4pt}
			\labellist
			\pinlabel {\small$q$} at 23 30
			\pinlabel {\large$\bullet$} at 27 24
			\pinlabel {\small$b$} at -2 4
			\pinlabel {\small$a$} at 37 3
			\endlabellist
			\includegraphics[width=.08\textwidth]{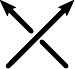}
		\end{gathered}\:\:\:\right\rrbracket \:\:\sim\:\: \left\llbracket \:\:\:\begin{gathered}
			\vspace{-4pt}
			\labellist
			\pinlabel {\small$q$} at 5 12
			\pinlabel {\large$\bullet$} at 9 6
			\pinlabel {\small$b$} at -2 4
			\pinlabel {\small$a$} at 37 3
			\endlabellist
			\includegraphics[width=.08\textwidth]{bigPosCrossing}
		\end{gathered}\:\:\:\right\rrbracket \hspace{30pt} \left\llbracket \:\:\:\begin{gathered}
			\vspace{-4pt}
			\labellist
			\pinlabel {\small$p$} at 14 30
			\pinlabel {\large$\bullet$} at 8.5 24
			\pinlabel {\small$b$} at -2 4
			\pinlabel {\small$a$} at 37 3
			\endlabellist
			\includegraphics[width=.08\textwidth]{bigPosCrossing}
		\end{gathered}\:\:\:\right\rrbracket \:\:\sim\:\: \left\llbracket \:\:\:\begin{gathered}
			\vspace{-4pt}
			\labellist
			\pinlabel {\small$p$} at 32 12
			\pinlabel {\large$\bullet$} at 26.5 6
			\pinlabel {\small$b$} at -2 4
			\pinlabel {\small$a$} at 37 3
			\endlabellist
			\includegraphics[width=.08\textwidth]{bigPosCrossing}
		\end{gathered}\:\:\:\right\rrbracket
	\]The analogous dot maps for the negative crossing are also chain homotopic. 
\end{prop}
\begin{proof}
	This is the equivariant version of \cite[Proposition 5.7]{MR3590355}, which is proved using the foam categories of Queffelec and Rose \cite{MR3545951}. These foam categories are defined by a series of local relations \cite[Equations (3.8)--(3.20)]{MR3545951}. These defining foam relations are imposed so that, among other things, the relations of the extended graphical calculus for categorified quantum $\sl(2)$ \cite{MR2963085}, suitably interpreted as relations among foams, are valid. The key relation in the extended graphical calculus needed for dot sliding is the ``square flop'' relation \cite[Lemma 4.6.4]{MR2963085}. By \cite[Proposition 2.20]{MR3877770}, the equivariant versions of Queffelec--Rose's defining relations are valid, so the equivariant version of the square flop relation is also valid. Rose--Wedrich's proof of dot sliding from the square flop relation \cite[Proof of Proposition 5.7]{MR3590355} finishes the argument. 
\end{proof}

\subsection{Different versions of the link invariant}\label{subsec:differentVersionsOfLinkInvt}

Let $D$ be a diagram of an oriented link $L$ with labeled components. Recall that $D$ induces a framing of $L$ called the \textit{blackboard framing}. If $L$ is a knot, then the blackboard framing is the Seifert framing plus the writhe of the diagram. By Theorem~\ref{thm:RmovesandForkMoves}, the chain homotopy type of $\llbracket D\rrbracket$ is an invariant of $L$ viewed as an oriented framed link with labeled components. To obtain a homological invariant, we pass $\llbracket D \rrbracket$ from the additive category $\sr C(\emp)$ to an abelian category and then take homology.

\begin{df}\label{df:stateSpace}
	The \textit{state space} $\sr F(W)$ of a closed $\sl(N)$ web $W$ is defined to be the graded finitely-generated free $\Sym(N)$-module $\Hom^*(\emp,W)$. This assignment extends to a $\Sym(N)$-linear additive functor $\sr F$ defined on $\sr C(\emp)$ sending $q^iW$ to $q^i\sr F(W)$. If $G$ is an $\sl(N)$ foam from $W_0$ to $W_1$ of degree $k$, viewed as a map from $q^{i+k}W_0$ to $q^{i}W_1$ in $\sr C(\emp)$, then $\sr F(G)$ is the grading-preserving map \[
		q^{i+k}\,\sr F(W_0) \to q^{i}\,\sr F(W_1)
	\]given by vertical composition $F \in \Hom(\emp,W_0) \mapsto F \cup_{W_0} G \in \Hom(\emp,W_1)$. 
\end{df}

The isomorphism type of $\sr F(W)$ is straightforward to determine using MOY calculus by Theorem~\ref{thm:RWcategorificationOfMOYCalculus}, but the maps induced by foams are less straightforward to determine explicitly. 
Since $\sr F$ is an additive functor, if $A$ is a chain complex in $\sr C(\emp)$ then $\sr F(A)$ is naturally a chain complex of graded $\Sym(N)$-modules. 

\begin{df}
	Let $D$ be a diagram of an oriented link $L$ with labeled components. The \textit{equivariant colored $\sl(N)$ complex} of $D$ is defined to be \[
		\KRC_{\U(N)}(D) \coloneq \sr F(\llbracket D \rrbracket).
	\]The chain homotopy type of $\KRC_{\U(N)}(D)$ is an invariant of the oriented framed labeled link $L$. 
	The \textit{equivariant colored $\sl(N)$ homology} of the oriented framed labeled link $L$, denoted $\KR_{\U(N)}(L)$, is the homology of $\KRC_{\U(N)}(D)$. 
\end{df}

\begin{rem}
	Different authors use different normalizations to obtain a link invariant that is independent of the framing. See \cite[Definition 3.3]{MR3877770} and \cite[Definition 12.16]{MR3234803} for two different conventions. Since these normalizations are local, they also shift gradings based on pairwise linking numbers of different components of the link. We work with the framed link invariant for simplicity. Furthermore, we adopt the convention that if $L$ is an oriented labeled link without an explicit framing, then we give each component its Seifert framing to define the colored $\sl(N)$ homology of $L$. The computations in Tables~\ref{table:trefoil24} to \ref{table:trefoil36} are with respect to the Seifert framing of the trefoil. 
\end{rem}

\begin{example}
	Let $U^a$ denote the unknot labeled $a$. The equivariant colored $\sl(N)$ homology of $U^a$ with the Seifert framing is \[
		\KR_{\U(N)}(U^a) \cong q^{-a(N-a)} H^*_{\U(N)}(\G(a,N))
	\]supported in homological grading zero. 
\end{example}

The nonequivariant version of colored $\sl(N)$ homology is defined using the evaluation \[
	\langle F \rangle_\Z\coloneq \langle F \rangle \,\big|_{X_1 = \cdots = X_N = 0} \in \Z
\]obtained from the Robert--Wagner evaluation by setting all of the variables equal to zero. For webs $V_0,V_1$ with the same boundary, we set \[
	\Hom^*_\Z(V_0,V_1) \coloneq \left(\bigoplus_G \Z \cdot G \right)\bigg/\!\sim
\]where $\sum_i a_i G_i \sim 0$ if $\sum_i a_i \: \langle G_i \cup_{\ol{V}_0 \cup V_1} H\rangle_\Z = 0$ for every foam $H$ from $\emp$ to $V_0 \cup \ol{V}_1$. We then let $\sr C_\Z(\beta)$ be the category having the same objects of $\sr C(\beta)$, but its morphism spaces are built from $\Hom^*_\Z(V_0,V_1)$ instead of $\Hom^*(V_0,V_1)$. In particular, morphisms consist only of $\Z$-linear combinations of foams rather than $\Sym(N)$-linear combinations. Note that there is an additive functor $\sr C(\beta) \to \sr C_\Z(\beta)$ arising from the natural maps $\Hom^*(V_0,V_1) \to \Hom^*_\Z(V_0,V_1)$. If $A$ is a chain complex in $\sr C(\beta)$, we let $A_\Z$ denote the resulting complex in $\sr C_\Z(\beta)$. Finally, we define the state space functor $\sr F_\Z$ taking $q^iW \in \sr C_\Z(\emp)$ to $q^i\Hom^*_\Z(\emp,W)$, viewed as a graded abelian group. It follows from the proof of \cite[Theorem 3.30]{MR4164001} that $\sr F_\Z(W)$ and $\sr F(W) \otimes_{\Sym(N)} \Z$ are naturally isomorphic functors. In particular, the Poincar\'e polynomial of $\sr F_\Z(W)$ is precisely the MOY polynomial of $W$.

\begin{df}
	Let $D$ be a diagram of an oriented labeled link $L$. The \textit{colored $\sl(N)$ complex} of $D$ is defined to be \[
		\KRC_N(D) \coloneq \sr F_\Z(\llbracket D \rrbracket_\Z) \cong \KRC_{\U(N)}(D) \otimes_{\Sym(N)} \Z
	\]The chain homotopy type of $\KRC_N(D)$ is an invariant of the oriented framed labeled link $L$. 
	The \textit{colored $\sl(N)$ homology} of $L$, denoted $\KR_N(L)$, is the homology of $\KRC_N(D)$. 
\end{df}

\begin{example}
	The colored $\sl(N)$ link homology of the unknot $U^a$ labeled $a$ with the Seifert framing is \[
		\KR_N(U^a) = q^{-a(N-a)} H^*(\G(a,N))
	\]supported in homological grading zero. 
\end{example}

To describe the module structures on colored $\sl(N)$ homology, we first point out a tautological module structure arising from the bending trick described in Remark~\ref{rem:bendingTrick}. If $V$ is a web in a disc, then \[
	\Hom^*_\Z(V,V) = \Hom^{* - s}_\Z(\emp,\ol{V} \cup V) = q^{s}\sr F_\Z(\ol{V} \cup V)
\]where $s$ is a grading shift depending only on $\partial V$ given explicitly in Remark~\ref{rem:bendingTrick}. We may view this identification as providing an action of the state space $\sr F_\Z(\ol{V}\cup V)$ on $V$. In particular, if $U^a$ is an unknot labeled $a$, then there is an action of $\KR_N(U^a) = \sr F_\Z(U^a)$ on the web consisting solely of a single strand labeled $a$. This is just another viewpoint on the dot maps discussed in section~\ref{subsec:RickardComplexes}, with an added description of relations that the dot maps satisfy. By horizontal composition, $\KR_N(U^a)$ acts on any web with a distinguished strand labeled $a$. Finally, if $D$ is an oriented link diagram with labeled components together with a basepoint on an arc labeled $a$, then there is an induced action of $\KR_N(U^a)$ on $\KRC_N(D)$ through chain maps. Since the action is through chain maps, there is an induced action of $\KR_N(U^a)$ on $\KR_N(L)$. This module structure only depends on the component of $L$ containing the basepoint, which follows from the argument in \cite[Section 3]{MR2034399}. 

\begin{df}
	Let $D$ be a diagram of an oriented labeled link $L$, and let $p \in D$ be a basepoint on an arc labeled $a$ away from crossings. The \textit{reduced colored $\sl(N)$ complex} of $D$ with respect to $p$ is defined to be \[
		\ol{\KRC}_N(D,p) \coloneq q^{-a(N-a)}([\G(a,N)]\cdot\KRC_N(D))
	\]where $[\G(a,N)] \in H^{2a(N-a)}(\G(a,N))$ is the fundamental class of $\G(a,N)$, and it acts on $\KRC_N(D)$ by the module structure induced by the basepoint $p$. The chain homotopy type of $\ol{\KRC}_N(D,p)$ is an invariant of the oriented framed labeled link $L$ together with its basepoint $p$. The \textit{reduced colored $\sl(N)$ homology} of $L$ with respect to $p$, denoted $\ol{\KR}_N(L,p)$ is the homology of $\ol{\KRC}_N(D,p)$. 
\end{df}
\begin{example}
	The reduced colored $\sl(N)$ homology of the unknot $U^a$ labeled $a$ with the Seifert framing is \[
		\ol{\KR}_N(U^a,p) = \Z
	\]supported in bigrading $(0,0)$. 
\end{example}

\section{Colored \texorpdfstring{$\sl(N)$}{sl(N)} complexes of the Hopf link and the trefoil}\label{sec:complexesOfHopfAndTref}

In section~\ref{subsec:homologicalPerturbationLemma}, we prove a version of the homological perturbation lemma. In section~\ref{subsec:HopfLinkComplex}, we apply the lemma in a simple way to show that the complex associated to the Hopf link is homotopy equivalent to a complex with no differential. We then turn our attention to simplifying the complex associated to the trefoil, where the bulk of our work is to adapt \cite[Theorem 3.24]{https://doi.org/10.48550/arxiv.2107.08117} to our setting. We follow their proof closely, but because they work with chain complexes with rational coefficients, we frequently are required to use different arguments in order to work with integer coefficients. They work with (nonequivariant) singular Soergel bimodules in the context of HOMFLYPT homology instead of $\sl(N)$ foams, but translating between these two languages is fairly standard \cite[Appendix A]{https://doi.org/10.48550/arxiv.2107.08117}. In section~\ref{subsec:shiftedRickardComplexes}, we adapt \cite[Proposition 2.31]{https://doi.org/10.48550/arxiv.2107.08117} concerning shifted Rickard complexes, which are homotopy equivalent to complexes associated to a crossing with a ``rung''. In section~\ref{subsec:fullTwist}, we finish proving our adaptation of \cite[Theorem 3.24]{https://doi.org/10.48550/arxiv.2107.08117}. Finally, in section~\ref{subsec:trefoilComplex}, we apply the homological perturbation lemma to complete our simplification of the complex associated to the trefoil.

\subsection{A homological perturbation lemma}\label{subsec:homologicalPerturbationLemma}

We review strong deformation retracts and prove a version of the homological perturbation lemma. Our chain complexes always lie in $\sr C(\beta)$ for a web-boundary $\beta$, though the results in this section are valid for bounded complexes in an additive category. 

\begin{df}\label{df:strongDeformationRetract}
	A \textit{strong deformation retract} of a chain complex $A$ onto a chain complex $\ol{A}$ consists of chain maps $\pi\colon A \to \ol{A}$ and $\iota\colon \ol{A} \to A$ and a homotopy $h\colon A \to A$ for which \[
		\pi\circ \iota = \Id_{\ol{A}} \qquad \iota\circ \pi - \Id_A = d\circ h + h\circ d
	\]and $h\circ \iota = 0$, $\pi\circ h = 0$, and $h^2 = 0$. These last three identities are called the \textit{side conditions}. If $\pi,\iota,h$ form a strong deformation retract that does not necessarily satisfy the side conditions, then $h$ can be replaced by $h'' = h'dh'$ where $h' = (\Id - \iota\pi)h(\Id - \iota\pi)$ so that $\pi,\iota,h''$ is a strong deformation retract satisfying the three side conditions \cite[Section 2]{MR893160}. 
\end{df}

\begin{df}\label{df:splitsOverPoset}
	Let $P$ be a finite poset and let $A$ be a chain complex. A \textit{splitting of $A$ over $P$} is a direct sum decomposition \[
		A = \bigoplus_{p \in P} A_p \qquad d = \sum_{p,q\in P} d_{q,p}
	\]for which the component of the differential $d_{q,p}\colon A_p \to A_q$ is zero unless $p \leq q$. 
\end{df}

\begin{examples}
	A chain complex that splits over the two element poset $\{0,1\}$ is essentially just a mapping cone. Let $A$ be the mapping cone of a chain map $f\colon B \to C$, and set $A_0 = B$ and $A_1 = hC$ with $d_{0,0} = d_B$, $d_{1,1} = -d_C$, and $d_{1,0} = f$, depicted as \[
		\begin{tikzcd}
			A_0 \ar[loop left,"d_{0,0}"] \ar[r,"d_{1,0}"] & A_1 \ar[loop right,"d_{1,1}"]
		\end{tikzcd}
	\]

	Let $P = \{(x,y) \in \Z^2 \:|\: 0 \leq y \leq x \leq 2\}$ with standard poset structure that $(x,y) \leq (x',y')$ if and only if $x \leq x'$ and $y \leq y'$. A complex $A$ that splits over $P$ may be depicted as \[
		\begin{tikzcd}[row sep=large]
			& & A_{22}\\
			& A_{11} \ar[ru] \ar[r] & A_{21} \ar[u]\\
			A_{00} \ar[rruu,bend left] \ar[r] \ar[ru] \ar[rru] \ar[rr,bend right] & A_{10} \ar[r] \ar[u] \ar[ru] \ar[ruu] & A_{20} \ar[u] \ar[uu,bend right]
		\end{tikzcd}
	\]with the understanding that there is an arrow from each term to itself. Note that there are no arrows between $A_{11}$ and $A_{20}$ since $(1,1)$ and $(2,0)$ are incomparable in $P$. 
\end{examples}

We provide a homological perturbation lemma for complexes that split over a finite poset. See for example \cite[Section 2.7]{https://doi.org/10.48550/arxiv.2205.12798} for the same result with characteristic two coefficients in the case that $P$ is the hypercube $\{0,1\}^k$. 

\begin{lem}[(Homological Perturbation Lemma)]\label{lem:homologicalPerturbationLemma}
	Let $A$ be a complex with a splitting over a finite poset $P$. Suppose that for each $p \in P$, there is a strong deformation retract \[
		\pi_p\colon A_p \to \ol{A}_p \qquad \iota_p\colon \ol{A}_p \to A_p \qquad h_p\colon A_p \to A_p
	\]of $A_p = (A_p,d_{p,p})$ onto a complex $\ol{A}_p = (\ol{A}_p,\ol{d}_{p})$. Then there is a strong deformation retract of $A$ onto the complex $\ol{A} = \bigoplus_{p} \ol{A}_p$ with differential $\ol{d} = \sum_{p\leq q} \ol{d}_{q,p}$ where $\ol{d}_{p,p} \coloneq \ol{d}_p$ and for $p < q$, \[
		\ol{d}_{q,p} \coloneq \sum_{p = p_1 <\cdots < p_k = q} \pi_{p_k} \circ d_{p_k,p_{k-1}}\circ h_{p_{k-1}} \circ \cdots \circ d_{p_3,p_2} \circ h_{p_2} \circ d_{p_2,p_1} \circ \iota_{p_1}.
	\]The strong deformation retract $\pi\colon A \to \ol{A}$, $\iota\colon \ol{A} \to A$, $h\colon A \to A$ is given in components by $\pi_{p,p} = \pi_p$, $\iota_{p,p} = \iota_p$, $h_{p,p} = h_p$ and for $p < q$, \begin{align*}
		\pi_{q,p} &\coloneq \sum_{p = p_1 <\cdots < p_k = q} \pi_{p_k} \circ d_{p_k,p_{k-1}}\circ h_{p_{k-1}} \circ \cdots \circ d_{p_3,p_2} \circ h_{p_2} \circ d_{p_2,p_1} \circ h_{p_1}\\
		\iota_{q,p} &\coloneq \sum_{p = p_1 <\cdots < p_k = q} h_{p_k} \circ d_{p_k,p_{k-1}}\circ h_{p_{k-1}} \circ \cdots \circ d_{p_3,p_2} \circ h_{p_2} \circ d_{p_2,p_1} \circ \iota_{p_1}\\
		h_{q,p} &\coloneq \sum_{p = p_1 <\cdots < p_k = q} h_{p_k} \circ d_{p_k,p_{k-1}}\circ h_{p_{k-1}} \circ \cdots \circ d_{p_3,p_2} \circ h_{p_2} \circ d_{p_2,p_1} \circ h_{p_1}. 
	\end{align*}
\end{lem}
\begin{proof}
	We first verify the lemma directly for the poset $P = \{0,1\}$. We then show that the general case follows from this case by induction. Let $(A,d)$ be a complex that splits over $P = \{0,1\}$. In particular, we have \[
		\begin{tikzcd}
			A_0 \ar[loop left,"d_{0,0}"] \ar[r,"d_{1,0}"] & A_1 \ar[loop right,"d_{1,1}"]
		\end{tikzcd}
	\]where $d_{0,0}d_{0,0} = 0$, $d_{1,1}d_{1,1} = 0$, and $d_{1,0}d_{0,0} + d_{1,1}d_{1,0} = 0$. Then $(\ol{A},\ol{d})$ is \[
		\begin{tikzcd}[column sep=large]
			\ol{A}_0 \ar[loop left,"\ol{d}_0"] \ar[r,"\pi_1 d_{1,0} \iota_0"] & \ol{A}_1 \ar[loop right, "\ol{d}_1"]
		\end{tikzcd}
	\]which is easily verified to be a chain complex. The maps $\pi\colon A \to\ol{A}$, $\iota\colon \ol{A} \to A$, and $h\colon A \to A$ are given by \[
		\pi = \begin{tikzcd}[row sep=large,column sep=large]
			A_0 \ar[d,"\pi_0"] \ar[rd,"\pi_1d_{1,0}h_0"] & A_1 \ar[d,"\pi_1"]\\
			\ol{A}_0 & \ol{A}_1
		\end{tikzcd} \qquad \iota = \begin{tikzcd}[row sep=large,column sep=large]
			A_0 & A_1 \\
			\ol{A}_0 \ar[u,"\iota_0"] \ar[ru,"h_1d_{1,0}\iota_0"] & \ol{A}_1 \ar[u,"\iota_1"]
		\end{tikzcd} \qquad h = \begin{tikzcd}[column sep=large]
			A_0 \ar[loop left,"h_0"] \ar[r,"h_1d_{1,0}h_0"] & A_1 \ar[loop right,"h_1"]
		\end{tikzcd}
	\]It is straightforward to check that $\pi$ and $\iota$ are chain maps and that $\iota\pi - \Id = dh + hd$. In the verification of the identity $\pi \iota = \Id$, the component of $\pi\iota$ from $\ol{A}_0$ to $\ol{A}_1$ is $\pi_1h_1d_{1,0}\iota_0 + \pi_1d_{1,0}h_0\iota_0$ where both terms are zero by the side conditions $\pi_1h_1 = 0$ and $h_0\iota_0 = 0$. The side conditions $\pi h = 0$, $h \iota = 0$, and $h^2 = 0$ are also all straightforward. 

	We now prove the result by induction on the size of the poset $P$. If $P$ has one element, the claim is tautological. For the inductive step, let $p \in P$ be a minimal element and let $Q = P\setminus p$. Then $A_Q \coloneq \bigoplus_{q \in Q} A_q$ is a subcomplex of $A$ that splits over the poset $Q$. By induction, the homological perturbation lemma applies to $A_Q$, giving a strong deformation retract onto $\smash{(\ol{A}_Q,\ol{d}_Q)}$. Now observe that $A$ splits over the poset $\{0,1\}$ with $A_0 = A_p$ and $A_1 = A_Q$. The case of the lemma for complexes that split over $\{0,1\}$ gives a strong deformation retract from $A$ to $\smash{\ol{A}_p \oplus \ol{A}_Q}$. It therefore suffices to check that $\smash{\ol{A}_p \oplus \ol{A}_Q}$ is indeed the complex $\smash{\ol{A}}$ described in the lemma, and that the strong deformation retract is given by the maps described in the lemma. This is straightforward; for example, we compute the induced differential $\smash{\ol{d}_{q,p}\colon \ol{A}_p \to \ol{A}_q}$ for some $q \in Q$ satisfying $p < q$. Let $d_{Q,p}\colon A_p \to A_Q$ denote $\sum_{p < q'} d_{q',p}$, and let $\pi_{q,Q}\colon A_Q \to \smash{\ol{A}_q}$ denote $\pi_{q} + \sum_{q_1 < q}\pi_{q,q_1}$. Then $\smash{\ol{d}_{q,p}}$ is \begin{align*}
		\pi_{q,Q}\circ d_{Q,p}\circ \iota_p &= \pi_{q,Q}\circ \left(d_{q,p} + \sum_{p < q_1 < q} d_{q_1,p}\right)\circ \iota_p\\
		&= \pi_{q} \circ d_{q,p}\circ \iota_p + \sum_{p < q_1 < q} \pi_{q,q_1}\circ d_{q_1,p}\circ \iota_p\\
		&= \pi_{q} \circ d_{q,p}\circ \iota_p + \sum_{p < q_1 < q} \left(\sum_{q_1 < q_2 < \cdots < q_k = q} \pi_{q_k}\circ d_{q_k,q_{k-1}}\circ h_{q_{k-1}}\circ\cdots\circ d_{q_2,q_1}\circ h_{q_1} \right)\circ d_{q_1,p}\circ\iota_p\\
		&= \sum_{p < q_1 < \cdots < q_k = q} \pi_{q_k}\circ d_{q_k,q_{k-1}} \circ h_{q_{k-1}}\circ\cdots\circ h_{q_1} \circ d_{q_1,p} \circ\iota_p
	\end{align*}as claimed. The other computations are similar. 
\end{proof}

\subsection{The Hopf link complex}\label{subsec:HopfLinkComplex}

\begin{lem}\label{lem:flatteningTwistClosureWk}
	If $\max(a + b - N,0) \leq k \leq \min(a,b)$, then there is a chain homotopy equivalence 
	\[
		\left\llbracket \:\:\begin{gathered}
			\vspace{-3pt}
			\centering
			\labellist
			\pinlabel {\small$a$} at 29 83
			\pinlabel {\small$b$} at 74 84
			\pinlabel {\small${b-k}$} at 51 84
			\pinlabel {\small$k$} at 75 55
			\pinlabel {\small${a-k}$} at 51 54
			\pinlabel {\small$a$} at 71 26
			\pinlabel {\small$b$} at 31 27
			\endlabellist
			\includegraphics[width=.17\textwidth]{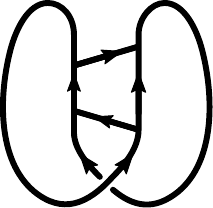}
		\end{gathered}\:\: \right\rrbracket \:\:\simeq\:\: h^kq^{ab - k(N+1)} \left\llbracket \qquad\quad\:\:\begin{gathered}
			\vspace{-3pt}
			\centering			
			\labellist
			\pinlabel {\small${b-k}$} at 30 80
			\pinlabel {\small$k$} at 59 44
			\pinlabel {\small${a+b-k}$} at -22 46
			\pinlabel {\small$a$} at 33 55
			\pinlabel {\small${a - k}$} at 35 18
			\pinlabel {\small$b$} at 70 5
			\endlabellist
			\includegraphics[width=.13\textwidth]{Theta}
		\end{gathered}\:\: \right\rrbracket
	\]
\end{lem}
\begin{proof}
	We claim that there are homotopy equivalences \[
		\left\llbracket \:\:\begin{gathered}
			\vspace{-3pt}
			\centering
			\includegraphics[width=.17\textwidth]{twistClosureWk}
		\end{gathered}\:\: \right\rrbracket \:\simeq\: \left\llbracket \begin{gathered}
			\vspace{-3pt}
			\centering
			\includegraphics[width=.17\textwidth]{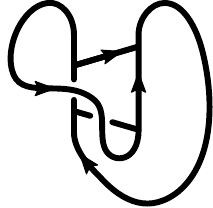}
		\end{gathered}\:\: \right\rrbracket \:\simeq\: h^{k-a}q^{(a-k)(N-a+1)} \left\llbracket \begin{gathered}
			\vspace{-3pt}
			\centering
			\includegraphics[width=.17\textwidth]{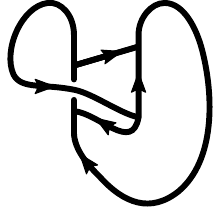}
		\end{gathered}\:\: \right\rrbracket
	\]The first equivalence is fork-sliding (Theorem~\ref{thm:RmovesandForkMoves}), while the second arises
	from the following variation of a fork-twist \begin{align*}
		\left\llbracket \quad\: \begin{gathered}
			\centering
			\labellist
			\pinlabel {\small$k$} at 12 65
			\pinlabel {\small${a-k}$} at -3 34
			\pinlabel {\small$a$} at 0 5
			\endlabellist
			\includegraphics[width=.07\textwidth]{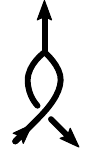}
		\end{gathered}\: \right\rrbracket &\:\:\simeq\:\: h^{k-a}q^{(a-k)(N-a+k+1)} \left\llbracket \:\begin{gathered}
			\centering
			\includegraphics[width=.07\textwidth]{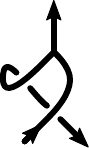}
		\end{gathered}\: \right\rrbracket\\
		&\:\:\simeq\:\: h^{k-a}q^{(a-k)(N-a+k+1)} \left\llbracket \:\begin{gathered}
			\centering
			\includegraphics[width=.07\textwidth]{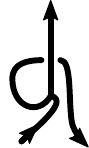}
		\end{gathered}\: \right\rrbracket \:\:\simeq\:\: h^{k-a}q^{(a-k)(N-a+1)} \hspace{8pt}\left\llbracket \:\begin{gathered}
			\centering
			\includegraphics[width=.07\textwidth]{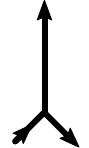}
		\end{gathered}\: \right\rrbracket 
	\end{align*}where the first equivalence is a Reidemeister I move, the second is a fork-slide, and the third is an ordinary fork-twist. Similarly, \[
		h^{k-a}q^{(a-k)(N-a+1)} \left\llbracket \begin{gathered}
			\vspace{-3pt}
			\centering
			\includegraphics[width=.17\textwidth]{firstForkTwist}
		\end{gathered}\:\: \right\rrbracket \:\:\simeq\:\: h^{k}q^{ab - k(N+1)} \left\llbracket \begin{gathered}
			\vspace{-3pt}
			\centering
			\includegraphics[width=.17\textwidth]{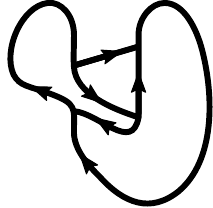}
		\end{gathered}\:\: \right\rrbracket
	\]which concludes the proof. 
\end{proof}

\begin{thm}\label{thm:slncomplexofHopfLink}
	If $0 \leq a,b \leq N$, then there is a homotopy equivalence \[
		\left\llbracket \HopfLink\right\rrbracket \simeq \bigoplus_{k = \max(a+b-N,0)}^{\min(a,b)} h^{2k} q^{ab - kN} \left\llbracket \Thetak \right\rrbracket
	\]where the differential of the complex on the right-hand side is zero. 
\end{thm}
\begin{proof}
	Let $C$ be the complex associated to the given diagram of the Hopf link, and note that $C$ splits over the poset $P = \{\: k \:|\: \max(a + b - N,0) \leq k \leq \min(a,b) \:\}$ in the following way. We let $(C_k,(-1)^k d_{k,k})$ be the complex \[
		h^kq^{-k} \left\llbracket \:\:\begin{gathered}
			\vspace{-3pt}
			\centering
			\labellist
			\pinlabel {\small$a$} at 29 83
			\pinlabel {\small$b$} at 74 84
			\pinlabel {\small$k$} at 75 55
			\endlabellist
			\includegraphics[width=.17\textwidth]{twistClosureWk}
		\end{gathered}\:\: \right\rrbracket 
	\]and we let $d_{k+1,k}\colon C_k \to C_{k+1}$ be the map induced by the foam $w_k$ given in Figure~\ref{fig:standardweb} used in the definition of the Rickard complex. 
	By Lemma~\ref{lem:flatteningTwistClosureWk}, there is a homotopy equivalence between $C_k$ and \[
		\ol{C}_k \coloneq h^{2k} q^{ab - k(N+1)} \left\llbracket \Thetak \right\rrbracket
	\]which is supported in a single homological grading. Since there are no nontrivial homotopies on $\ol{C}_k$, any homotopy equivalence between $C_k$ and $\ol{C}_k$ can be upgraded to a strong deformation retract $C_k \to \ol{C}_k$ satisfying the side conditions given in Definition~\ref{df:strongDeformationRetract}. By the homological perturbation lemma (Lemma~\ref{lem:homologicalPerturbationLemma}), we know that $C$ is homotopy equivalent to \[
		\ol{C} \coloneq \bigoplus_{k = \max(a + b - N,0)}^{\min(a,b)} \ol{C}_k
	\]equipped with some differential. Since $\ol{C}$ is supported only in even homological degrees, its differential must be zero which completes the proof. 
\end{proof}

\subsection{Shifted Rickard complexes}\label{subsec:shiftedRickardComplexes}

In this section, we adapt \cite[Proposition 2.31]{https://doi.org/10.48550/arxiv.2107.08117} to our setting. Our result is directly analogous, but we give a different proof in order to work with integer coefficients. 
We assume that $0 \leq b\leq a \leq N$. This assumption simplifies matters only because the webs in our Rickard complexes are biased towards the left (the edge with the highest label is on the left). There is an analogous complex defined with webs biased towards the right that turns out to be equivalent to the Rickard complex, see for example \cite[Corollary 12.17]{MR3234803} or \cite[Remark 2.24]{https://doi.org/10.48550/arxiv.2107.08117}. 

\begin{prop}\label{prop:shiftedRickardComplex}
	There is a homotopy equivalence \[
		\left\llbracket \qquad\quad\:\:\:\:\begin{gathered}
		  	\vspace{-3pt}
		  	\centering
		  	\labellist
			\pinlabel {\small${a+b-k}$} at -21 65
			\pinlabel {\small$k$} at 43 65
			\pinlabel {\small${a-k}$} at 19 57
			\pinlabel {\small$a$} at -3 2
			\pinlabel {\small$b$} at 41 3
			\endlabellist
		  	\includegraphics[width=.065\textwidth]{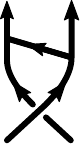}
		\end{gathered}\:\:\:\: \right\rrbracket \:\simeq\: q^{b(a-k)}\left(\begin{tikzcd}
		 	\cdots \ar[r] & h^mq^{-m(a - k+1)} V_m \ar[r,"v_m"] & h^{m+1}q^{-(m+1)(a-k+1)} V_{m+1} \ar[r] & \cdots
		 \end{tikzcd} \right)
	\]The complex on the right is called the \emph{shifted Rickard complex} and is supported in homological degrees $m$ where $\max(a + b-N,0) \leq m \leq k$. The webs $V_m$ and foams $v_m$ are given in Figure~\ref{fig:shiftedStandardWeb}.
\end{prop}

\begin{figure}[!ht]
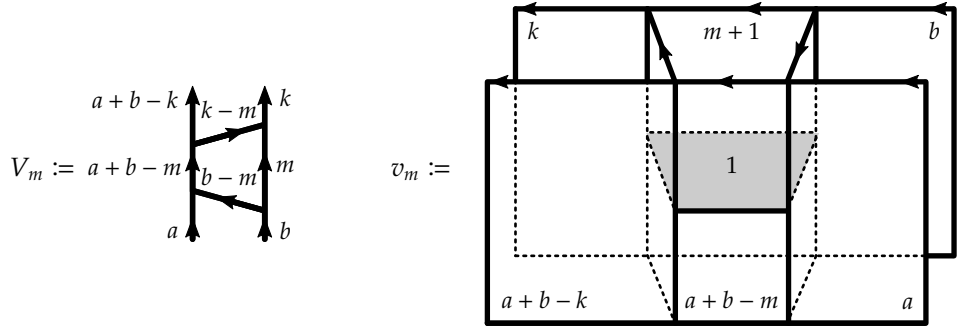

	\centering
	$V_m\coloneq \qquad\quad\:\:\:\:\begin{gathered}
		\labellist
		\pinlabel {\small${a+b-k}$} at -12.5 35.5
		\pinlabel {\small$k$} at 25 36
		\pinlabel {\small$m$} at 25 19
		\pinlabel {\small${k - m}$} at 11 33
		\pinlabel {\small${a + b - m}$} at -12.5 19
		\pinlabel {\small${b - m}$} at 11 16.5
		\pinlabel {\small$b$} at 25 3
		\pinlabel {\small$a$} at -3 2.5
		\endlabellist
		\includegraphics[width=.07\textwidth]{standardWeb}
	\end{gathered}$
	\hspace{40pt}
	$v_m\coloneq \quad\begin{gathered}
		\labellist
		\pinlabel {\small$k$} at 13 79
		\pinlabel {\small$b$} at 120 79
		\pinlabel {\small${m+1}$} at 66 79
		\pinlabel {\small$1$} at 66 43
		\pinlabel {\small${a+b-k}$} at 16 7
		\pinlabel {\small$a$} at 113 6
		\pinlabel {\small${a + b - m}$} at 66 7
		\endlabellist
		\includegraphics[width=.38\textwidth]{differential}
	\end{gathered}$
	\captionsetup{width=.8\linewidth}
	\caption{The labels satisfy $\max(a + b - N,0) \leq m \leq k \leq b \leq a \leq N$. On the left is the web $V_m = V_{m}(a,b,k)$ and on the right is the foam $v_m$ from $q^{a-k+1}V_m$ to $V_{m+1}$. The web $\ol{V_m}$ is at the bottom of the picture while $V_{m+1}$ is at the top. 
	}
	\vspace{-5pt}
	\label{fig:shiftedStandardWeb}
\end{figure}

We prove Proposition~\ref{prop:shiftedRickardComplex} using the following two lemmas. Recall that the (unshifted) Rickard complex is built from the webs $W_k$ and foams $w_k$ in Figure~\ref{fig:standardweb}. 

\begin{lem}\label{lem:homSpaceCalculationShiftedRickard}
	Let $k$ and $k + u$ satisfy $\max(a + b- N,0) \leq k\leq k + u \leq b$. Then \[
		\rank_\Z \Hom(q^jW_k,W_{k+u}) = \rank_\Z \Hom(q^jW_{k+u},W_k) = \begin{cases}
			0 & j < u^2 \\
			1 & j = u^2
		\end{cases}
	\]Now suppose $m$ and $m + u$ satisfy $\max(a + b - N,0) \leq m \leq m+u \leq k$. Then \[
		\rank_\Z \Hom(q^jV_m,V_{m+u}) = \rank_\Z \Hom(q^jV_{m+u},V_m) = \begin{cases}
			0 & j < u(a - k + u)\\
			1 & j = u(a - k + u)
		\end{cases}
	\]
\end{lem}
\begin{proof}
	Note that \[
		\Hom(q^jW_k,W_{k+u}) = \Hom^j(W_k,W_{k+u}) = \Hom^{j - a(N-a) - b(N-b)}(\emp,\ol{W_k} \cup W_{k+u})
	\]where the second equality is from the bending trick discussed in Remark~\ref{rem:bendingTrick}. By Theorem~\ref{thm:RWcategorificationOfMOYCalculus}, the rank of this free abelian group is determined by the MOY polynomial of $\ol{W_k} \cup W_{k+u}$. The first equality in the first statement of the lemma follows from the fact that \[
		\langle \ol{W_k} \cup W_{k+u}\rangle = \langle \ol{\ol{W_k} \cup W_{k+u}} \rangle = \langle \ol{W_{k+u}} \cup W_k\rangle. 
	\]The second equality follows from the fact that $\langle \ol{W}_k \cup W_{k+u}\rangle$ is monic of degree $a(N-a) + b(N-b) - u^2$ by direct computation using the MOY calculus relations in Figure~\ref{fig:MOYcalculus}. 

	Similarly for the second statement in the lemma, it suffices to show that $\langle \ol{V_m} \cup V_{m+u} \rangle = \langle \ol{V_{m+u}} \cup V_m\rangle$ is monic of degree \[
		\frac{a(N-a) + b(N-b) + k(N-k) + (a + b - k)(N - a - b + k)}{2} - u(a - k + u)
	\]which again is verified by direct computation using the MOY calculus. 
\end{proof}

\begin{lem}\label{lem:shiftedDifferentialPrimitive}
	$\Hom(qW_k,W_{k+1}) \cong \Z$ is generated by $w_k$. Similarly, $\Hom(q^{a - k + 1}V_m,V_{m+1}) \cong \Z$ is generated by $v_m$. 
\end{lem}
\begin{proof}
	We claim that the first statement implies the second.  
	Let $G$ be a generator of $\Hom(q^{a-k+1}V_m,V_{m+1})$ and write $v_m = nG$ for some $n \in \Z$. Then the foam \[
		\begin{gathered}
			\centering
			\labellist
			\pinlabel {\small$a$} at 13 79
			\pinlabel {\small$b$} at 120 79
			\pinlabel {\small$a-k$} at 30 55
			\pinlabel {$s_\lambda$} at 33 32
			\pinlabel {\small$k-m$} at 55 32
			\pinlabel {\small${m+1}$} at 80 79
			\pinlabel {\small$1$} at 80 43
			\pinlabel {\small$b$} at 5 7
			\pinlabel {\small$a$} at 113 6
			\pinlabel {\small${a + b - m}$} at 80 7
			\endlabellist
			\includegraphics[width=.33\textwidth]{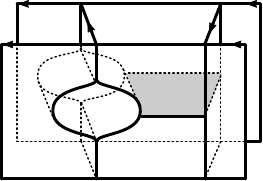}
		\end{gathered}
	\]for any partition $\lambda \in P(a - k,N-a+k)$ is also divisible by $n$. This is a foam from $W_m(b,a)$ to $W_{m+1}(b,a)$ of some degree depending on $\lambda$. If we set $\lambda = \mathrm{box}(a-k,N-a+k)$, then by foam relations~\ref{item:foamForkRelation} and \ref{item:oneRungRelation} of Proposition~\ref{prop:foamRelations}, the foam is of degree $1$ and agrees with $w_m$. If $w_m$ is a generator of $\Hom(qW_m(b,a),W_{m+1}(b,a))$, then $n = \pm1$ so $v_m$ is also a generator. 

	There are a number of ways to see that $w_k$ is a generator of $\Hom(qW_k,W_{k+1})$. We give a proof analogous to an argument we will use in the proof of Proposition~\ref{prop:shiftedRickardComplex}. Observe that \[
		C\coloneq \left\llbracket\:\:\:\:\begin{gathered}
			\vspace{-3pt}
			\labellist
			\pinlabel {\small$a-b$} at 65 8
			\pinlabel {\small$a$} at -5 29
			\pinlabel {\small$b$} at -4 4
			\endlabellist
			\includegraphics[width=.12\textwidth]{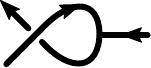}
		\end{gathered}\:\:\right\rrbracket \:\:\simeq\:\: h^bq^{ab - b(N+1)} \left\llbracket\:\:\:\:\begin{gathered}
			\vspace{-3pt}
			\labellist
			\pinlabel {\small$a-b$} at 65 8
			\pinlabel {\small$a$} at -5 29
			\pinlabel {\small$b$} at -4 4
			\endlabellist
			\includegraphics[width=.12\textwidth]{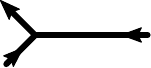}
		\end{gathered}\:\:\right\rrbracket
	\]by a variation of the fork twist as in the proof of Lemma~\ref{lem:flatteningTwistClosureWk}. We let $Y$ denote the web on the right. For $\max(a+b-N,0) \leq k \leq b$, there is an isomorphism between the summand $C_k$ in homological degree $k$ of the complex $C$ with a direct sum of $q$-shifts of $Y$ \[
		C_k = h^kq^{-k} \:\:\:\:\begin{gathered}
			\labellist
			\pinlabel {\small$a$} at -3 50
			\pinlabel {\small$b-k$} at 22 49
			\pinlabel {\small$b$} at 50 42
			\pinlabel {\small$k$} at 40 27
			\pinlabel {\small$a-b$} at 70 20
			\pinlabel {\small$a$} at 50 14
			\pinlabel {\small$a-k$} at 22 6
			\pinlabel {\small$b$} at -2 4
			\endlabellist
			\includegraphics[width=.15\textwidth]{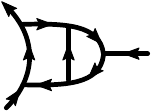}
		\end{gathered}\quad \:\:\cong\:\: h^kq^{-k}\: \qbinom{N-a+k}{k}\qbinom{N-a}{b-k} \quad\:\:\begin{gathered}
			\vspace{-3pt}
			\labellist
			\pinlabel {\small$a-b$} at 65 8
			\pinlabel {\small$a$} at -5 29
			\pinlabel {\small$b$} at -4 4
			\endlabellist
			\includegraphics[width=.12\textwidth]{sideTwistFlat}
		\end{gathered} \:\:\:\:\eqqcolon B_k
	\]The isomorphism follows by Robert and Wagner's categorification of MOY calculus (Theorem~\ref{thm:RWcategorificationOfMOYCalculus}). Note that the lowest power of $q$ that arises in $B_k$ is \[
		-b(N - a - b) + k(k - 2b - 1).
	\]This is a strictly decreasing function of $k$ in the range $\max(a + b - N,0) \leq k \leq b$. A calculation analogous to Lemma~\ref{lem:homSpaceCalculationShiftedRickard} gives \[
		\rank_\Z \Hom(q^jY,Y) = \begin{cases}
			0 & j < 0\\
			1 & j = 0.
		\end{cases}
	\]

	Now let $d_k\colon B_k \to B_{k+1}$ be the differential induced from the differential of $C$ by the isomorphisms $C_k \cong B_k$. By construction, the complex $B\coloneq \bigoplus B_k$ with differential $d_k$ is chain isomorphic to $C$. Because $C$, and therefore $B$, is homotopy equivalent to a complex supported in homological degree $b$, there are maps $h_k\colon B_{k} \to B_{k-1}$ for which $h_{k+1}d_k + d_{k-1}h_{k} = \Id_{B_k}$ when $\max(a + b - N,0) \leq k < b$. Consider the component of $h_k$ mapping out of \[
		q^{-b(N-a-b)+k(k-2b-1)} Y
	\]which is the lowest $q$-shift of $Y$ appearing as a direct summand of $B_k$. Since all $q$-shifts of $Y$ appearing in $B_{k-1}$ have a strictly larger power of $q$, it follows that this component of $h_k$ is zero. Thus, the restriction of the identity $h_{k+1}d_k + d_{k-1}h_k = \Id_{B_k}$ to this $q$-shift of $Y$ is $h_{k+1}d_k = \Id$. If $w_k = n G$ for some $n \in \Z$, then $d_k$ is also divisible by $n$. It then follows that the identity map on $Y$ is divisible by $n$. We now claim that $\Id_Y$ is a generator of $\Hom(Y,Y) \cong \Z$ from which it follows that $n = \pm1$. 

	Let $G$ be a generator of $\Hom(Y,Y) \cong \Z$ and write $\Id = mG$ for some $m \in \Z$. We know that $m\neq 0$ since $\Hom(Y,Y) \cong \Z$ implies that $Y$ is not isomorphic to zero. Express $G\circ G \in \Hom(Y,Y)$ as $lG$ for some $l\in \Z$. The identity $\Id\circ\Id = \Id \in \Hom(Y,Y)$ gives $m^2 l G = m G$ from which it follows that $m(ml - 1) = 0$. Thus $m = \pm 1$.
\end{proof}

\begin{proof}[Proof of Proposition~\ref{prop:shiftedRickardComplex}]
	Let $A$ denote the complex on the left-hand side in the statement of the proposition. The direct summand of $A$ in homological degree $l$ is \begin{align*}
		h^l q^{-l}\qquad\qquad\begin{gathered}
			\labellist
			\pinlabel {\small${a+b-k}$} at -12 53
			\pinlabel {\small$k$} at 24 53
			\pinlabel {\small$a$} at 24 35
			\pinlabel {\small$b$} at -3 35
			\pinlabel {\small$a-l$} at 11 33
			\pinlabel {\small${a+b-l}$} at -12 18
			\pinlabel {\small$l$} at 24 18
			\pinlabel {\small$b-l$} at 11 16
			\pinlabel {\small${a-k}$} at 11 49
			\pinlabel {\small$a$} at -3 2
			\pinlabel {\small$b$} at 24 3
			\endlabellist
			\includegraphics[width=.065\textwidth]{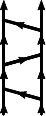}
		\end{gathered} \:\:\:\: &\cong \:\: h^l q^{-l}\bigoplus_{j=\max(a + b - N,0)}^{\min(k,l)} \qbinom{a - k - l + b}{a - k - l + j} \qquad\qquad\begin{gathered}
			\labellist
			\pinlabel {\small${a+b-k}$} at -12 53
			\pinlabel {\small$k$} at 24 53
			\pinlabel {\small$j$} at 24 35
			\pinlabel {\small$a+b-j$} at -12 35
			\pinlabel {\small$l - j$} at 11 33
			\pinlabel {\small${a+b-l}$} at -12 18
			\pinlabel {\small$l$} at 24 18
			\pinlabel {\small$b-l$} at 11 16
			\pinlabel {\small${k-j}$} at 11 49
			\pinlabel {\small$a$} at -3 2
			\pinlabel {\small$b$} at 24 3
			\endlabellist
			\includegraphics[width=.065\textwidth]{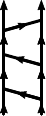}
		\end{gathered}\\
		&\cong \:\: h^l q^{-l}\bigoplus_{j=\max(a + b - N,0)}^{\min(k,l)} \qbinom{a - k - l + b}{a - k - l + j} \qbinom{b - j}{l - j}\:\: V_j
	\end{align*}for $\max(a + b - N,0) \leq l \leq b$, where the isomorphisms follow from Theorem~\ref{thm:RWcategorificationOfMOYCalculus}. Define a finite-length filtration $\cdots \subseteq\sr F^i(A) \subseteq \sr F^{i+1}(A) \subseteq \cdots$ by letting $\sr F^i(A)$ be the direct sum of the terms $q^nV_m$ for which $n + m(a - k + 1) \leq i$. By Lemma~\ref{lem:homSpaceCalculationShiftedRickard}, the differential of $A$ sends $\sr F^i(A)$ to itself. 

	Next, note that by Theorem~\ref{thm:RmovesandForkMoves}, we have homotopy equivalences \[
		A\coloneq\left\llbracket \qquad\quad\:\:\:\begin{gathered}
		  	\vspace{-3pt}
		  	\centering
		  	\labellist
			\pinlabel {\small${a+b-k}$} at -21 65
			\pinlabel {\small$k$} at 43 65
			\pinlabel {\small${a-k}$} at 19 57
			\pinlabel {\small$a$} at -3 2
			\pinlabel {\small$b$} at 41 3
			\endlabellist
		  	\includegraphics[width=.065\textwidth]{rungTwist}
		\end{gathered}\:\:\:\: \right\rrbracket \:\simeq\: \left\llbracket \:\: \begin{gathered}
			\vspace{-3pt}
		  	\centering
		  	\includegraphics[width=.065\textwidth]{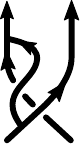}
		\end{gathered}\:\: \right\rrbracket \:\simeq\: q^{b(a-k)}\left\llbracket \:\: \begin{gathered}
			\vspace{-3pt}
		  	\centering
		  	\includegraphics[width=.065\textwidth]{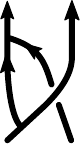}
		\end{gathered}\:\: \right\rrbracket \eqqcolon B
	\]where the first equivalence is a fork slide and the second is a fork twist. Just as we did for $A$, we express $B$ in terms of $q$-shifts of the webs $V_m$. The direct summand of $B$ in homological degree $l$ is \begin{align*}
		h^l q^{b(a-k)-l} \qquad\quad \begin{gathered}
			\labellist
			\pinlabel {\small$a+b-k$} at -19 65
			\pinlabel {\small$k$} at 78 65
			\pinlabel {\small$b$} at 20 55
			\pinlabel {\small$k-l$} at 50 55
			\pinlabel {\small$a-k$} at -11 32
			\pinlabel {\small$l$} at 77 32
			\pinlabel {\small$k$} at 20 8
			\pinlabel {\small$b-l$} at 50 8
			\pinlabel {\small$a$} at -3 2
			\pinlabel {\small$b$} at 77 3
			\endlabellist
			\includegraphics[width=.14\textwidth]{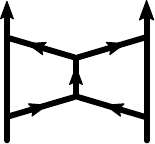}
		\end{gathered} \:\:\:\: &\cong \:\:h^l q^{b(a-k)-l} \bigoplus_{j=\max(a + b - l - N,0)}^{k-l} \qbinom{a -k + l}{l+j} \qquad\begin{gathered}
			\labellist
			\pinlabel {\small$a+b-k$} at -19 65
			\pinlabel {\small$k$} at 78 65
			\pinlabel {\small$k-l$} at 50 57
			\pinlabel {\small$l$} at 77 32
			\pinlabel {\small$j$} at 43 31
			\pinlabel {\small$b-l$} at 50 6
			\pinlabel {\small$a$} at -3 2
			\pinlabel {\small$b$} at 77 3
			\endlabellist
			\includegraphics[width=.14\textwidth]{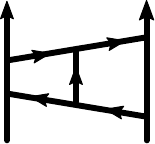}
		\end{gathered}\\
		&\cong \:\: h^l q^{b(a-k)-l} \bigoplus_{j=\max(a+b-N,l)}^{k} \qbinom{a - k + l}{j} \qbinom{j}{l} \:\: V_j
	\end{align*}for $\max(b + k-N,0) \leq l \leq k$, where the isomorphisms are again from Theorem~\ref{thm:RWcategorificationOfMOYCalculus}. We define a filtration $\cdots \subseteq \sr F^i(B) \subseteq \sr F^{i+1}(B) \subseteq \cdots$ in the same way as for $A$. We let $\sr F^i(B)$ be the direct sum of the terms $q^nV_m$ for which $n + m(a - k + 1) \leq i$, and again Lemma~\ref{lem:homSpaceCalculationShiftedRickard} implies that the differential of $B$ sends $\sr F^i(B)$ to itself. 

	By Lemma~\ref{lem:homSpaceCalculationShiftedRickard}, any homotopy equivalence between $A$ and $B$ is a filtered homotopy equivalence. In particular, the subcomplex $\sr F^i(A)$ is homotopy equivalent to $\sr F^i(B)$ for all $i$. We claim that for $i = b(a-k)$, we have $A = \sr F^i(A)$. This follows by direct computation from our explicit decomposition of $A$ into $q$-shifts of the webs $V_j$. It follows that $A$ is homotopy equivalent to the subcomplex $\sr F^{b(a - k)}(B)$. Again by direct computation using our explicit decomposition of $B$, we see that $\sr F^{b(a - k)}(B)$ is supported in homological gradings $m$ for which $\max(a + b - N,0) \leq m \leq k$ and its direct summand in degree $m$ is \[
		h^mq^{b(a - k) - m(a - k + 1)} V_m.
	\]Thus, $\sr F^{b(a - k)}(B)$ and the shifted Rickard complex are complexes having the same underlying objects but potentially different differentials. 

	By Lemma~\ref{lem:shiftedDifferentialPrimitive}, the differential in $\sr F^{b(a - k)}(B)$ from homological degree $m$ to degree $m+1$ is equal to an integral multiple $n_m v_m$ of $v_m$. It suffices to show that $n_m = \pm 1$ for each $m$. To do so, we apply the same argument used in the proof of Lemma~\ref{lem:shiftedDifferentialPrimitive}. First, observe that the complex \[
		\left\llbracket\:\:\:\quad\qquad\begin{gathered}
			\vspace{-3pt}
			\labellist
			\pinlabel {\small${a+b-k}$} at -21 65
			\pinlabel {\small$k$} at 43 65
			\pinlabel {\small${a-k}$} at 19 57
			\pinlabel {\small$b-k$} at 85 50
			\pinlabel {\small$a$} at -3 2
			\pinlabel {\small$b$} at 57 9
			\endlabellist
			\includegraphics[width=.17\textwidth]{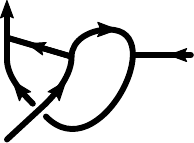}
		\end{gathered}\:\:\right\rrbracket
	\]is homotopy equivalent to a complex supported in homological degree $k$. This follows from a sequence of moves analogous to those used in the proof of Lemma~\ref{lem:shiftedDifferentialPrimitive}. Next, we horizontally compose the same web with $\sr F^{b(a-k)}(B)$ so that the resulting complex in homological degree $m$ is \[
		h^mq^{-m(a-k+1)} \qquad\quad\:\:\begin{gathered}
			\labellist
			\pinlabel {\small$a+b-k$} at -18 50
			\pinlabel {\small$k-m$} at 22 49
			\pinlabel {\small$k$} at 50 42
			\pinlabel {\small$m$} at 41 27
			\pinlabel {\small$a+b-m$} at -9 27
			\pinlabel {\small$b-k$} at 70 35
			\pinlabel {\small$b$} at 50 14
			\pinlabel {\small$b-m$} at 22 6
			\pinlabel {\small$a$} at -2 4
			\endlabellist
			\includegraphics[width=.15\textwidth]{sideTwistWk}
		\end{gathered}\quad \:\:\cong\:\: h^mq^{-m(a-k+1)}\: \qbinom{N-b+m}{m}\qbinom{N-a-b+k}{k-m} \quad\:\:\begin{gathered}
			\vspace{-3pt}
			\labellist
			\pinlabel {\small$b-k$} at 70 25
			\pinlabel {\small$a+b-k$} at 15 39
			\pinlabel {\small$b$} at -4 4
			\endlabellist
			\includegraphics[width=.12\textwidth]{sideTwistFlat}
		\end{gathered} 
	\]The lowest power of $q$ arising in this direct sum is \[
		-k(N - a - b) + m(m-2a - 1)
	\]which is strictly decreasing as a function of $m$ in the range $\max(a + b - N,0) \leq m \leq k$. We finish the proof exactly as in Lemma~\ref{lem:shiftedDifferentialPrimitive}. 
\end{proof}

\subsection{The full twist on two strands}\label{subsec:fullTwist}

We now use Proposition~\ref{prop:shiftedRickardComplex} to explicitly describe a chain complex $S$ that is homotopy equivalent to the complex \[
	\left\llbracket \:\:\:\begin{gathered}
		\vspace{-3pt}
		\labellist
		\pinlabel {\small$a$} at -4 60
		\pinlabel {\small$b$} at 38 61
		\pinlabel {\small$a$} at -3 3
		\pinlabel {\small$b$} at 37 4
		\endlabellist
		\includegraphics[width=.06\textwidth]{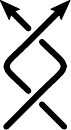}
	\end{gathered} \:\:\:\right\rrbracket
\]associated to a full twist on two strands where we continue to assume $0 \leq b \leq a \leq N$. This ``simplified'' complex $S$ is directly analogous to the complex in \cite[Theorem 3.24]{https://doi.org/10.48550/arxiv.2107.08117}. 

Because the construction is fairly long, we first describe the shape of $S$. We will define $S$ as a direct sum $\bigoplus_{k,l} S^k_l$ where $S^k_l$ is supported in homological degree $k + l$. The indices $l$ and $k$ satisfy $\max(a + b - N,0) \leq l \leq k \leq b$, and we imagine that these direct summands are arranged at lattice points $(k,l)$ in the plane. \[
	\begin{tikzcd}[row sep=1.5em, column sep=1.5em]
		& {} & & {} &\\
		{} \ar[r] & S^k_{l+1} \ar[rr,"s^k_{l+1}"] \ar[u] & & S^{k+1}_{l+1} \ar[r] \ar[u] & {}\\
		\\
		{} \ar[r] & S^k_{l} \ar[rr,"s^k_l"] \ar[uu,"d^k_l"] & & S^{k+1}_l \ar[uu,"d^{k+1}_l"] \ar[r] & {}\\
		& {} \ar[u] & & {} \ar[u] &
	\end{tikzcd}
\]We let $S^k = \bigoplus_l S^k_l$ be the direct sum of the terms in the $k$th column. There are ``vertical'' maps $d^k_l\colon S^k_l \to S^k_{l+1}$ that make $S^k$ into a chain complex, and there are ``horizontal'' maps $s^k_l\colon S^k_l \to S^{k+1}_l$ which together define chain maps $s^k \colon S^k \to S^{k+1}$. These chain maps satisfy $s^{k+1}\circ s^k = 0$ so all together, we have a double complex. The differential on $S$ is the total differential of this double complex, obtained by negating the vertical maps of the column $S^k$ for $k$ odd. 

We turn to the task of constructing the column complexes $S^k$ and proving Proposition~\ref{prop:twistWk}. 
Following \cite[Lemma 3.23]{https://doi.org/10.48550/arxiv.2107.08117}, we first describe a well-known bijection between the partitions in $P(r,s)$ for $r,s \ge 0$ and the set of square-free monomials of degree $s$ in $r+s$ variables $\zeta_1,\ldots,\zeta_{r + s}$. Note that partitions in $P(r,s)$ are in natural bijection with lattice paths from $(0,0)$ to $(s,r)$ moving only rightward and upward. Here is an illustration of the correspondence for $(r,s) = (4,5)$: \[
	(5,3,1) \quad\longleftrightarrow\quad \begin{gathered}
		\ydiagram{5,3,1,0}
	\end{gathered} \quad\longleftrightarrow \quad 
	\begin{gathered}
		\begin{tikzpicture}
		\draw[step=1.5em] (0,0) grid (5*1.5em,4*1.5em);
		\draw[line width=2pt] (0,0) -- (0,1.5em) -- (1.5em,1.5em) -- (1.5em,3em) -- (4.5em,3em) -- (4.5em,4.5em) -- (7.5em,4.5em) -- (7.5em,6em);
		\end{tikzpicture}
	\end{gathered}
\]Record the reverse of path going from $(s,r)$ to $(0,0)$ as a string of ``D''s (for \textit{down}) and ``L''s (for \textit{left}). The string for the given example is ``DLLDLLDLD''. Finally, associate to this string the monomial $\zeta_{j_1}\cdots \zeta_{j_s}$ where the numbers $1 \leq j_1 < \ldots < j_s \leq r+s$ are the locations of the ``L''s in the string. The monomial for this example is $\zeta_2\zeta_3\zeta_5\zeta_6\zeta_8$. Next, if $W$ is a web, then let $\zeta_j W$ denote the $q$-shift $q^{2j}W$. There is a natural identification \[
	q^{s(r+s+1)} \qbinom{r+s}{s}\: W = \bigoplus_{1 \leq j_1 < \cdots < j_s \leq r+s} \zeta_{j_1}\cdots \zeta_{j_s} W = \bigoplus_J \zeta_J W
\]where we have essentially just given labels to an explicit basis for the left-hand side. The second equality is simply the shorthand notation $\zeta_J\coloneq \zeta_{j_1}\cdots \zeta_{j_s}$ where $J = (j_1,\ldots,j_s)$.

\begin{df}\label{df:columnComplexSk}
	We define the column complex $S^k$ where $\max(a + b - N,0) \leq k \leq b$. If $\max(a + b - N,0) \leq l \leq k$, then set \[
		S^k_l \coloneq h^{k+l}q^{-l(a - k + 1) + b(a - k)-k} \qbinom{b-l}{b-k}\: U_l = h^{k+l} q^{-l(a - b + 1) + b(a - b - 1)} \bigoplus_J \zeta_J U_l
	\]where $U_l$ is given in Figure~\ref{fig:fullyShiftedRickardWebs} and where the direct sum is over $J = (j_1,\ldots,j_{b-k})$ for which $1 \leq j_1 < \cdots < j_{b-k} \leq b-l$. 
	The component of the differential $d^k_l\colon S^k_l \to S^k_{l+1}$ from $\zeta_J U_l$ to $\zeta_{J'} U_{l+1}$ is zero unless $j_i^{} - j_i' \in \{0,1\}$ for each $i = 1,\ldots,b-k$ in which case the map is $X^nu_l$ where $n = \sum_i j_i^{} - j_i'$, also given in Figure~\ref{fig:fullyShiftedRickardWebs}. 
\end{df}

\begin{figure}[!ht]
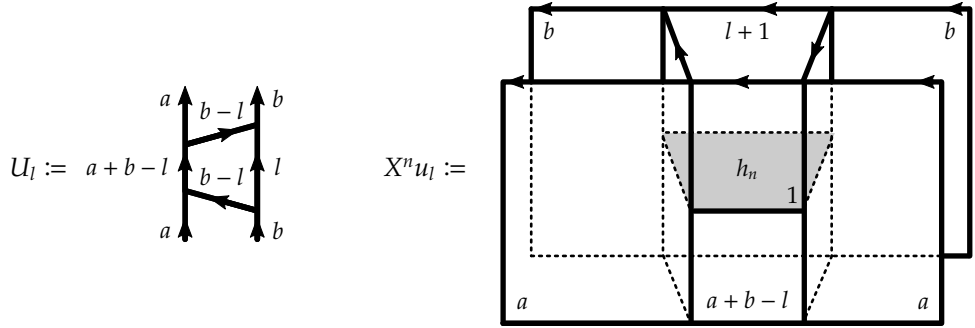

	\centering
	$U_l \coloneq \qquad\quad\:\:\:\:\begin{gathered}
		\labellist
		\pinlabel {\small$a$} at -3 35.5
		\pinlabel {\small$b$} at 25 36
		\pinlabel {\small$l$} at 25 19
		\pinlabel {\small${b - l}$} at 11 33
		\pinlabel {\small${a + b - l}$} at -12.5 19
		\pinlabel {\small${b - l}$} at 11 16.5
		\pinlabel {\small$b$} at 25 3
		\pinlabel {\small$a$} at -3 2.5
		\endlabellist
		\includegraphics[width=.07\textwidth]{standardWeb}
		\end{gathered}$
	\hspace{40pt}
	$X^nu_l\coloneq\quad\begin{gathered}
		\labellist
		\pinlabel {\small$b$} at 13 79
		\pinlabel {\small$b$} at 120 79
		\pinlabel {\small${l+1}$} at 66 79
		\pinlabel {\small$h_n$} at 66 42
		\pinlabel {\small$1$} at 78 35
		\pinlabel {\small$a$} at 6 6
		\pinlabel {\small$a$} at 113 6
		\pinlabel {\small${a + b - l}$} at 66 7
		\endlabellist
		\includegraphics[width=.38\textwidth]{differential}
	\end{gathered}$
	\captionsetup{width=.8\linewidth}
	\caption{Here $\max(a + b - N,0) \leq l \leq b \leq a \leq N$. On the left is the web $U_l = U_l(a,b)$ and on the right is the foam $X^nu_l$ from $q^{a-b+1+2n}U_l$ to $U_{l+1}$. The web $\ol{U_l}$ is at the bottom of the picture while $U_{l+1}$ is at the top. The shaded facet labeled $1$ is decorated by $h_n(X) = X^n$, which is equivalent to placing $n$ dots of weight $1$ on the facet.}
	\label{fig:fullyShiftedRickardWebs}
\end{figure}

\begin{prop}\label{prop:twistWk}
	There is a chain homotopy equivalence \[
		h^kq^{-k}\left\llbracket \:\:\:\quad\qquad\begin{gathered}
			\vspace{-3pt}
			\labellist
			\pinlabel {\small$a$} at -3 85
			\pinlabel {\small$b$} at 41 86
			\pinlabel {\small${b-k}$} at 19 79
			\pinlabel {\small${a+b-k}$} at -21 57
			\pinlabel {\small$k$} at 43 57
			\pinlabel {\small${a-k}$} at 19 57
			\pinlabel {\small$a$} at -3 2
			\pinlabel {\small$b$} at 41 3
			\endlabellist
			\includegraphics[width=.07\textwidth]{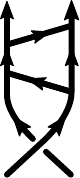}
		\end{gathered}\quad \right\rrbracket \:\:\simeq\:\: S^k
	\]
\end{prop}
\begin{proof}
	Let $A$ be the complex on the left-hand side in the statement of the proposition. By Proposition~\ref{prop:shiftedRickardComplex}, $A$ is homotopy equivalent to the complex $B = \bigoplus_l B_l$ with differential $d_l\colon B_l \to B_{l+1}$ given by \[
		B_l\coloneq h^{k+l}q^{-l(a-k+1)+b(a-k)-k}\qquad\quad\:\:\:\:\begin{gathered}
		 	\labellist
		 	\pinlabel {\small$a$} at -3 52
		 	\pinlabel {\small$b$} at 25 53
		 	\pinlabel {\small$b-k$} at 11 49
		 	\pinlabel {\small${a+b-k}$} at -12.5 35.5
			\pinlabel {\small$k$} at 25 36
			\pinlabel {\small$l$} at 25 19
			\pinlabel {\small${k - l}$} at 11 33
			\pinlabel {\small${a + b - l}$} at -12.5 19
			\pinlabel {\small${b - l}$} at 11 16.5
			\pinlabel {\small$b$} at 25 3
			\pinlabel {\small$a$} at -3 2.5
		 	\endlabellist
		 	\includegraphics[width=.07\textwidth]{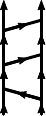}
		 \end{gathered} \hspace{37pt}d_l\coloneq \:\begin{gathered}
			\labellist
			\pinlabel {\small$b$} at 13 79
			\pinlabel {\small$b$} at 120 79
			\pinlabel {\small$k$} at 46 79
			\pinlabel {\small${l+1}$} at 79 79
			\pinlabel {\small$1$} at 79 42
			\pinlabel {\small$a$} at 6 6
			\pinlabel {\small$a$} at 113 6
			\pinlabel {\small${a + b - l}$} at 79 7
			\endlabellist
			\includegraphics[width=.38\textwidth]{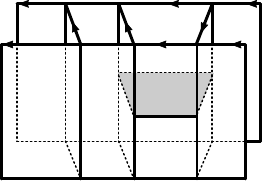}
		\end{gathered}
	\]We now define an explicit chain isomorphism between $B$ and $S^k$. We define $\phi_l\colon S^k_l \to B_l$ and $\psi_l\colon B_l \to S^k_l$ by specifying their component maps $\phi_l(J)\colon \zeta_JU_l \to B_l$ and $\psi_l(J)\colon B_l \to \zeta_JU_l$. Let $\mu$ be the partition in $P(k-l,b-k)$ corresponding to $\zeta_J = \zeta_{j_1}\cdots\zeta_{j_{b-k}}$ and let \[
		\phi_l(J)\coloneq\begin{gathered}
			\labellist
			\pinlabel {$s_\mu$} at 60 44
			\endlabellist
			\includegraphics[width=.38\textwidth]{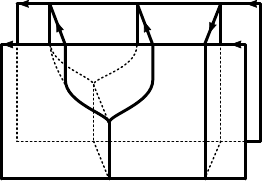}
		\end{gathered} \hspace{15pt} \psi_l(J)\coloneq (-1)^{|\mu|} \: \begin{gathered}
			\labellist
			\pinlabel {${s_{\hat{\mu}}}$} at 40 40
			\endlabellist
			\includegraphics[width=.38\textwidth]{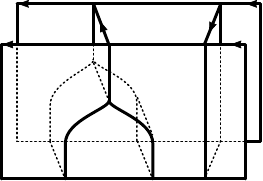}
		\end{gathered}
	\]The maps $\phi_l$ and $\psi_l$ are inverses by foam relations~\ref{item:twoRungRelation} and \ref{item:oneRungRelation} of \ref{prop:foamRelations}. To show that $\phi_l$ and $\psi_l$ define chain maps, we must verify that the map $\zeta_J U_l \to \zeta_{J'} U_{l+1}$ given by \[
		\psi_{l+1}(J')\circ d_l\circ \phi_l(J) = (-1)^{|\lambda|}\:\:\begin{gathered}
			\labellist
			\pinlabel {$s_{\hat{\lambda}}$} at 33 52
			\pinlabel {$s_\mu$} at 55 32
			\endlabellist
			\includegraphics[width=.38\textwidth]{primitive}
		\end{gathered}
	\]agrees with the component of the differential $d^k_l$ from $\zeta_J U_l$ to $\zeta_{J'}U_{l+1}$. Here $\lambda$ is the partition in $P(k-l-1,b-k)$ corresponding to $\zeta_{J'}$. By Pieri's formula and relations~\ref{item:dotmigrationrelation} and \ref{item:foamForkRelation} of Proposition~\ref{prop:foamRelations}, the map $\psi_{l+1}(J')\circ d_l \circ \phi_l(J)$ is equal to \[
		\sum_\nu (-1)^{|\lambda|}\:\:\begin{gathered}
			\labellist
			\pinlabel {$s_{\hat{\lambda}}$} at 35 65
			\pinlabel {$s_\nu$} at 60 45
			\pinlabel {\small$h_{|\mu| - |\nu|}$} at 75 25
			\endlabellist
			\includegraphics[width=.38\textwidth]{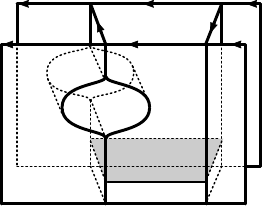}
		\end{gathered}
	\]where the sum is over partitions $\nu \in P(k-l-1,b-k)$ for which the Young diagram of $\nu$ fits inside of the Young diagram of $\mu$ and for which the skew Young diagram $\mu\setminus \nu$ has at most one box in each column. By foam relation~\ref{item:oneRungRelation}, the term in the sum corresponding to $\nu$ is equal to $X^{|\mu| - |\nu|}u_l$ if $\nu = \lambda$ and is zero otherwise. 

	For $\mu \in P(k-l,b-k)$ and $\nu \in P(k-l-1,b-k)$ with corresponding monomials $\zeta_{j_1}\cdots\zeta_{j_{b-k}}$ and $\zeta_{j_1'}\cdots\zeta_{j_{b-k}'}$, respectively, it suffices to show that $j_i^{} - j_i' \in \{0,1\}$ for $i = 1,\ldots,b-k$ if and only if $\nu$ fits inside $\mu$ and $\mu\setminus \nu$ has at most one box in each column. This latter condition is equivalent to the claim that the string of ``D''s and ``L''s corresponding to $\nu$ can be obtained from the string corresponding to $\nu$ by replacing a collection of disjoint substrings of the form ``DLL $\cdots$ L'' by the substrings ``LL $\cdots$ LD'' so that the last letter is ``D'', and then deleting this final ``D''. This condition is then equivalent to the claim that for each $i = 1,\ldots,b-k$, the $i$th ``L'' in the string of $\mu$ is in the same position as the $i$th ``L'' in the string of $\nu$, or is in one slot to the right, which proves the result. 
\end{proof}

We now define the ``horizontal'' maps $s^k_l\colon S^k_l \to S^{k+1}_{l}$ in the double complex $S$.

\begin{df}\label{df:rowMapsWk}
	Define the component of $s^k_l\colon S^k_l \to S^{k+1}_l$ from $\zeta_J U_l$ to $\zeta_{J'}U_l$ to be zero unless $J'$ is obtained from $J$ by deleting an element $j_i \in J = \{j_1 < \ldots < j_{b-k}\}$, in which case define the component map to be $(-1)^{i-1}E_{j_i}$ where $E_j \in \Hom(q^{2j}U_l,U_l)$ is the following difference of dot maps \[
		E_j\coloneq \:\:\begin{gathered}
			\labellist
			\pinlabel {\large$\bullet$} at 11 27
			\pinlabel {\small$e_j$} at 11 37
			\endlabellist
			\includegraphics[width=.035\textwidth]{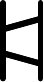}
		\end{gathered} \:\: - \:\: \begin{gathered}
			\labellist
			\pinlabel {\large$\bullet$} at 11 10.5
			\pinlabel {\small$e_j$} at 11 1
			\endlabellist
			\includegraphics[width=.035\textwidth]{smallLadder}
		\end{gathered}
	\]
\end{df}

It is straightforward to see that $s^{k+1}_l\circ s^k_l = 0$. The component map of $s^{k+1}_l\circ s^k_l$ from $\zeta_J U_l$ to $\zeta_{J''}U_l$ is clearly zero except when $\zeta_{J''}$ is obtained by deleting two elements $j_i,j_{i'}$ from $J$. In this case, the component map is simply $E_{j_i}E_{j_{i'}} - E_{j_{i'}}E_{j_i} = 0$ which is zero because the $E_j$ maps commute. 

To complete the construction of the double complex $S$, we turn to the task of proving that the horizontal and vertical maps commute, which is to say $s^k_{l+1}\circ d^k_l = d^{k+1}_l\circ s^k_l$. We view this as the assertion that $s^k\colon S^k \to S^{k+1}$ is a chain map between the column complexes. Strictly speaking, we mean that $s^k\colon S^k \to h^{-1}S^{k+1}$ is a chain map because chain maps are supposed to preserve homological degree, though we omit this extra notation when there is little possibility of confusion.  
Our proof is based on how the analogous result is proved in \cite[Proposition 3.10]{https://doi.org/10.48550/arxiv.2107.08117}. We summarize the argument. First, we define a chain complex $\hat{S}^k$ which has $S^k$ as a natural quotient complex. We then define a map $\hat{s}^k\colon\hat{S}^k \to \hat{S}^{k+1}$ that descends to $s^k\colon S^k \to S^{k+1}$, from which it suffices to prove that $\hat{s}^k$ is a chain map. Next, we construct another complex $\hat{R}^k$ together with a chain isomorphism $\Psi^k\colon \hat{R}^k \to \hat{S}^k$, which can be viewed as a change of basis. Finally, we verify that composite map $(\Psi^{k+1})^{-1}\circ \hat{s}^k\circ \Psi^k \colon \hat{R}^k \to \hat{R}^{k+1}$ is a chain map which completes the argument.

Recall that for $\max(a +b-N,0) \leq l \leq k \leq b$, we have \[
	S^k_l = h^{k+l}q^{\bullet} \bigoplus_{\substack{J \subseteq \{1,\ldots,b-l\}\\|J| = b-k}} \zeta_J U_l
\]where the precise $q$-shift $q^\bullet$ is given in Definition~\ref{df:columnComplexSk}. Let $m = b - \max(a + b - N,0)$ and set \[
	\hat{S}^k_l \coloneq h^{k+l}q^\bullet \bigoplus_{\substack{J \subseteq \{1,\ldots,m\}\\|J| = b-k}} \zeta_J U_l \hspace{30pt} \hat{R}^k_l \coloneq h^{k+l}q^\bullet \bigoplus_{\substack{I \subseteq \{1,\ldots,m\}\\|I| = b-k}} \xi_I U_l
\]where $\xi_i$ carries a $q$-shift of $q^{2i}$ just like $\zeta_i$. We define the differential $\hat{S}^k_l\to\hat{S}^k_{l+1}$ by the same formula for the differential $d^k_l\colon S^k_l \to S^k_{l+1}$. In particular, the component from $\zeta_JU_l$ to $\zeta_{J'}U_{l+1}$ is zero unless $j_i^{} - j_i' \in \{0,1\}$ for each $i = 1,\ldots,b-k$ in which case the map is $X^nu_l$ where $n = \sum_i j_i^{} - j_i'$. The obvious projection map $\hat{S}^k \to S^k$ is a chain map so $S^k$ is a quotient of $\hat{S}^k$. The differential $\hat{R}^k_l \to \hat{R}^k_{l+1}$ is simpler. Its component from $\xi_IU_l$ to $\xi_{I'}U_{l+1}$ is zero unless $I = I'$ in which case the map is $u_l$. 

\begin{lem}\label{lem:PsikChainIsomorphism}
	There is a chain isomorphism $\Psi^k\colon \hat{R}^k \to \hat{S}^k$. 
\end{lem}

\begin{proof}
	Let $\mathbf{A}$ be a finite alphabet, and let $\Sym(\mathbf{A})\langle \zeta_1,\ldots,\zeta_{m}\rangle$ denote the free $\Sym(\mathbf{A})$-module with basis the formal symbols $\zeta_1,\ldots,\zeta_{m}$. Then let $\Lambda_\zeta(\mathbf{A})$ denote the exterior algebra of $\Sym(\mathbf{A})\langle \zeta_1,\ldots,\zeta_{m}\rangle$ over the ring $\Sym(\mathbf{A})$. A basis of $\Lambda_\zeta(\mathbf{A})$ as a free $\Sym(\mathbf{A})$-module is given by $\zeta_J = \zeta_{j_1}\cdots \zeta_{j_n}$ where $J = \{j_1 < \cdots < j_n\} \subseteq \{1,\ldots,m\}$. Define $\Lambda_\xi(\mathbf{A})$ in the same way except with $\zeta_1,\ldots,\zeta_{m}$ replaced by $\xi_1,\ldots,\xi_{m}$, respectively. Next, define an isomorphism $\psi_{\mathbf{A}}\colon \Sym(\mathbf{A})\langle \xi_1,\ldots,\xi_m\rangle \to \Sym(\mathbf{A})\langle \zeta_1,\ldots,\zeta_m\rangle$ by \[
		\psi_{\mathbf{A}}(\xi_j) = h_{j-1}(\mathbf{A})\zeta_1 - h_{j-2}(\mathbf{A})\zeta_2 + \cdots + (-1)^{j-1} \zeta_j
	\]whose inverse is given by \[
		\psi_{\mathbf{A}}^{-1}(\zeta_i) = e_{i-1}(\mathbf{A})\xi_1 - e_{i-2}(\mathbf{A})\xi_2 + \cdots + (-1)^{i-1} \xi_i. 
	\]The fact that these two maps are inverses follows from the identity \[
		e_n(\mathbf{A}) - e_{n-1}(\mathbf{A})h_1(\mathbf{A}) + \cdots + (-1)^{n-1} e_1(\mathbf{A}) h_{n-1}(\mathbf{A}) + (-1)^n h_n(\mathbf{A}) = 0
	\]for all $n \ge 1$. It follows that $\psi_{\mathbf{A}}$ induces a graded $\Sym(\A)$-algebra isomorphism $\Lambda_\xi(\mathbf{A}) \to \Lambda_\zeta(\mathbf{A})$ which we also denote by $\psi_{\mathbf{A}}$. 

	Now assume that $|\mathbf{A}| = b-l$. We use $\psi_{\mathbf{A}}$ to define the isomorphism $\Psi_l^k\colon \hat{R}^k_l \to \hat{S}^k_l$. Let the component of $\Psi_l^k$ from $\xi_I U_l$ to $\zeta_J U_l$ be the following map. Consider the coefficient $P = P(I,J) \in \Sym(\mathbf{A})$ of $\zeta_J$ in the expression of $\psi_{\mathbf{A}}(\xi_I) \in \Lambda_\zeta(\mathbf{A})$ in terms of the basis. Define the component map to be\[
		\begin{gathered}
			\labellist
			\pinlabel {\large$\bullet$} at 11 27
			\pinlabel {\small$P$} at 11 38
			\endlabellist
			\includegraphics[width=.035\textwidth]{smallLadder}	
		\end{gathered}
	\]where we note that the dot lies on an edge labeled $b - l = |\mathbf{A}|$. An explicit formula for $P$ is \[
		P = (-1)^{j_1 + \cdots + j_{b-k} - (b - k)} \sum_{\sigma \in \fk{S}_{b-k}} (-1)^{\sign(\sigma)} \prod_{p=1}^{b-k} h_{i_{\sigma(p)} - j_p}(\mathbf{A})
	\]where $h_n = 0$ if $n < 0$ by convention. The inverse of $\Psi^k_l$ is constructed from $\psi_{\mathbf{A}}^{-1}$. In particular, its component map from $\zeta_J U_l$ to $\xi_I U_l$ is given by a dot map on the same edge as above except now the polynomial in $\Sym(\mathbf{A})$ is the coefficient of $\xi_I$ in $\psi_{\mathbf{A}}^{-1}(\zeta_J)$. The polynomial is \[
		(-1)^{i_1+\cdots+i_{b-k}-(b-k)} \sum_{\sigma \in \fk{S}_{b-k}} (-1)^{\sign(\sigma)} \prod_{p=1}^{b-k} e_{j_{\sigma(p)}-i_p}(\mathbf{A}).
	\]These component maps could be used to define $\Psi^k_l$ directly. However, the purpose of introducing the purely algebraic isomorphism $\psi_{\A}\colon \Lambda_\xi(\A) \to \Lambda_\zeta(\A)$ is to avoid having to check that $\Psi^k_l$ is an isomorphism just from the component maps. We now employ a similar strategy to show that $\Psi^k\colon \hat{R}^k \to \hat{S}^k$ is a chain map. 

	Let $\mathbf{B} \subset \mathbf{A}$ be a subset of size $|\mathbf{B}| = |\mathbf{A}| - 1 = b - l - 1$ and let $\{X\} = \mathbf{A}\setminus\mathbf{B}$. Let $\Sym(\mathbf{B}|X)$ denote the polynomials in the alphabet $\mathbf{A}$ that are symmetric in $\mathbf{B}$. In particular, there is an inclusion $\Sym(\mathbf{A}) \subset \Sym(\mathbf{B}|X)$. Define maps \[
		d_\xi\colon \Lambda_\xi(\mathbf{A}) \to \Lambda_\xi(\mathbf{A}) \otimes_{\Sym(\mathbf{A})} \Sym(\mathbf{B}|X) \hspace{30pt} d_\zeta\colon \Lambda_\zeta(\mathbf{A}) \to \Lambda_\zeta(\mathbf{A}) \otimes_{\Sym(\mathbf{A})} \Sym(\mathbf{B}|X)
	\]by declaring that $d_\xi(\xi_i) = \xi_i$ and $d_\zeta(\zeta_i) = \zeta_i + X\zeta_{i-1}$ with the convention $\zeta_0 = 0$. Then extend both maps to the entirety of their domains multiplicatively and $\Sym(\mathbf{B}|X)$-linearly. Then let\[
		\psi_{\mathbf{B}|X}\colon \Lambda_\xi(\mathbf{B}) \otimes_{\Sym(\mathbf{B})} \Sym(\mathbf{B}|X) \to \Lambda_\zeta(\mathbf{B}) \otimes_{\Sym(\mathbf{B})} \Sym(\mathbf{B}|X)
	\]be the isomorphism $\psi_{\mathbf{B}} \otimes \Id$. We now claim that \[
		d_\zeta = \psi_{\mathbf{B}|X}\circ d_\xi \circ \psi_{\mathbf{A}}^{-1}.
	\]To check this identity, it suffices to verify it on $\zeta_1,\ldots,\zeta_m$ because both sides are $\Sym(\mathbf{B}|X)$-algebra maps. We compute that \begin{align*}
		(\psi_{\mathbf{B}|X}\circ d_{\xi} \circ \psi_{\mathbf{A}}^{-1})(\zeta_i) &= \psi_{\mathbf{B}|X}\left(\sum_{j=1}^i (-1)^{j-1} e_{i-j}(\mathbf{A}) \xi_j \right)\\
		&= \sum_{j=1}^i (-1)^{j-1} e_{i-j}(\mathbf{A}) \sum_{s = 1}^j (-1)^{s-1} h_{j-s}(\mathbf{B}) \zeta_s = \sum_{s=1}^i \sum_{j=s}^i (-1)^{j-s} e_{i-j}(\mathbf{A}) h_{j-s}(\mathbf{B}) \zeta_s
	\end{align*}where the last equality involves a simple reindexing of the double sum. Now note that \[
		\sum_{j=s}^i (-1)^{j-s} e_{i-j}(\mathbf{A}) h_{j-s}(\mathbf{B}) = \sum_{t = 0}^{i-s} (-1)^t e_{i-s-t}(\mathbf{A}) h_t(\mathbf{B}) = e_{i-s}(X) = \begin{cases}
			1 & i -s = 0\\
			X & i - s = 1\\
			0 & i - s > 1
		\end{cases}
	\]where we have used Lemma~\ref{lem:generatingFunctionIdentity}. Thus $(\psi_{\mathbf{B}|X}\circ d_\xi \circ \psi_{\mathbf{A}}^{-1})(\zeta_i) = d_\zeta(\zeta_i)$ as claimed.

	We now use the identity $d_\zeta = \psi_{\mathbf{B}|X} \circ d_\xi \circ \psi_A^{-1}$ to show that $\Psi^k$ is a chain map. Note that the component of the differential $\hat{R}^k_l \to \hat{R}^k_{l+1}$ from $\xi_I U_l \to \xi_{I'} U_{l+1}$ is precisely $Pu_l$ where $P$ is the coefficient of $\xi_{I'}$ in the expression of $d_\xi(\xi_I)$. This coefficient $P$ is just $1$ when $\{i_1,\ldots,i_{b-k}\} = \{i_1',\ldots,i_{b-k}'\}$ and $0$ otherwise. Similarly, the component of the differential $\hat{S}^k_l \to \hat{S}^k_{l+1}$ from $\zeta_J U_l \to \zeta_{J'}U_{l+1}$ is $Qu_{l}$ where $Q$ is the coefficient of $\zeta_{J'}$ in the expression of $d_\zeta(\zeta_J)$. Notice our abuse of notation: since $Q$ is either $0$ or a power of $X$, we may interpret $Qu_l$ as the dotted foam given in Figure~\ref{fig:fullyShiftedRickardWebs}. The fact that $\Psi^k$ is a chain map now follows from the dot-migration relation (foam relation~\ref{item:dotmigrationrelation}) applied to the following three facets with associated formal alphabets $\mathbf{A},\mathbf{B},\{X\}$ where $|\mathbf{A}| = b-l$ and $|\mathbf{B}| = b - l - 1$. \[
		\begin{gathered}
			\labellist
		\pinlabel {\small$b$} at 13 79
		\pinlabel {\small$b$} at 120 79
		\pinlabel {\small${l+1}$} at 66 79
		\pinlabel {\small$\mathbf{B}$} at 47 60
		\pinlabel {\small$\{X\}$} at 66 43
		\pinlabel {\small$\mathbf{A}$} at 47 25
		\pinlabel {\small$1$} at 78 35
		\pinlabel {\small$a$} at 6 6
		\pinlabel {\small$a$} at 113 6
		\pinlabel {\small${a + b - l}$} at 66 7
		\endlabellist
			\includegraphics[width=.38\textwidth]{differential}
		\end{gathered}
	\]
	The relations imply that a decoration $f \in \Sym(\mathbf{A})$ on the facet labeled $b-l$ can be migrated to the linear combination of decorations on the facets labeled $b - l - 1$ and $1$ corresponding to the same polynomial $f$ under the inclusion $\Sym(\mathbf{A}) \subset \Sym(\mathbf{B}|X) = \Sym(\mathbf{B}) \otimes \Sym(X)$ induced by any identification $\mathbf{A} = \mathbf{B} \sqcup \{X\}$. 
\end{proof}
	
\begin{prop}\label{prop:WisaChainMap}
	The horizontal maps commute with the vertical maps, which is to say that $s^k_{l+1}\circ d^k_l = d^{k+1}_l \circ s^k_l$.
\end{prop}
\begin{proof}
	First, define $\hat{s}^k\colon \hat{S}^k \to \hat{S}^{k+1}$ by the same formula used to define $s^k$. In particular, define the component from $\zeta_J U_l$ to $\zeta_{J'}U_l$ to be zero unless $J'$ is obtained from $J$ by deleting an element $j_i \in J = \{j_1 < \cdots < j_{b-k}\}$, in which case define the map to be $(-1)^{i-1}E_{j_i}$. It suffices to show that $\hat{s}^k\colon\hat{S}^k \to \hat{S}^{k+1}$ is a chain map. 

	Define a map $\hat{r}^k\colon \hat{R}^k \to \hat{R}^{k+1}$ by declaring that the component map from $\xi_I U_l$ to $\xi_{I'}U_l$ is zero unless $I'$ is obtained from $I$ by deleting an element $i_j \in I = \{i_1 < \cdots < i_{b-k}\}$, in which case define the component map to be $(-1)^{j-1}H_{i_j}$ where $H_i \in \Hom(q^{2i}U_l,U_l)$ is the following map \[
		H_i \coloneq \sum_{s + t = i}(-1)^t\:\:\begin{gathered}
			\labellist
			\pinlabel {\large$\bullet$} at 20 36
			\pinlabel {\small$h_s$} at 29 36
			\pinlabel {\large$\bullet$} at 20 2
			\pinlabel {\small$e_t$} at 29 2
			\endlabellist
			\includegraphics[width=.035\textwidth]{smallLadder}
		\end{gathered}
	\]Because dot maps on edges incident to the boundary of the web are central by Remark~\ref{rem:centralDotMaps}, it follows that $\hat{r}^k$ is a chain map. By the dot-migration relation (foam relation~\ref{item:dotmigrationrelation}), we have \[
		H_i = \sum_{s + t = i_j}(-1)^t\:\:\begin{gathered}
			\labellist
			\pinlabel {\large$\bullet$} at 20 36
			\pinlabel {\small$h_s$} at 29 36
			\pinlabel {\large$\bullet$} at 20 2
			\pinlabel {\small$e_t$} at 29 2
			\endlabellist
			\includegraphics[width=.035\textwidth]{smallLadder}
		\end{gathered}\quad =\: \sum_{s + t=i_j}(-1)^t \:\:\begin{gathered}
			\labellist
			\pinlabel {\large$\bullet$} at 11 27
			\pinlabel {\small$h_s$} at 11 38
			\pinlabel {\large$\bullet$} at 11 10.5
			\pinlabel {\small$e_t$} at 11 2
			\endlabellist
			\includegraphics[width=.035\textwidth]{smallLadder}
		\end{gathered}
	\]We now show that the composite map $\Psi^{k+1} \circ r^k \circ (\Psi^k)^{-1}\colon \hat{S}^k \to \hat{S}^{k+1}$ is precisely $s^k$, where $\Psi^k$ is the isomorphism constructed in Lemma~\ref{lem:PsikChainIsomorphism}. 

	Continuing the notation from the proof of Lemma~\ref{lem:PsikChainIsomorphism}, let $\A$ be an alphabet of size $|\A| = b-l$ and recall that we have a $\Sym(\A)$-algebra isomorphism $\psi_{\A}\colon \Lambda_\xi(\A) \to \Lambda_\zeta(\A)$ given by \begin{align*}
		\psi_{\A}(\xi_j) &= h_{j-1}(\A)\zeta_1 - \cdots + (-1)^{j-1}\zeta_j\\
		\psi_{\A}^{-1}(\zeta_i) &= e_{i-1}(\A)\xi_1 -\cdots + (-1)^{i-1}\xi_i.
	\end{align*}Now let $\C$ be another alphabet of size $|\C| = b-l$ and let \[
		\psi_{\A|\C}\colon \Lambda_\xi(\A) \otimes_{\Sym(\A)} \Sym(\A|\C) \to \Lambda_\zeta(\A) \otimes_{\Sym(\A)} \Sym(\A|\C)
	\]be the isomorphism $\psi_{\A} \otimes \Id$. Here $\Sym(\A|\C) = \Sym(\A) \otimes_{\Z} \Sym(\C)$ denotes the ring of polynomials in the alphabet $\A \sqcup \C$ that are symmetric in $\A$ and also symmetric in $\C$. Define $\Sym(\A|\C)$-linear anti-derivations \[
		r\colon \Lambda_\xi\otimes_{\Sym(\A)}\Sym(\A|\C) \to \Lambda_\xi\otimes_{\Sym(\A)}\Sym(\A|\C), \qquad s\colon \Lambda_\zeta\otimes_{\Sym(\A)}\Sym(\A|\C) \to \Lambda_\zeta\otimes_{\Sym(\A)}\Sym(\A|\C)
	\]by the formulas \[
		r(\xi_j) = \sum_{s + t = j}(-1)^t h_s(\A)e_t(\C), \qquad s(\zeta_i) = e_i(\A) - e_i(\C)
	\]By definition of an anti-derivation, they satisfy the signed Leibniz rule $r(\xi_{I}\xi_{J}) = r(\xi_{I})\xi_{J} + (-1)^{|I|}\xi_{I}r(\xi_{J})$. Just as in the proof of Lemma~\ref{lem:PsikChainIsomorphism}, the coefficient of $\xi_{I'}$ in the expression of $r(\xi_I)$ in terms of the given basis is precisely the component map of $r^k\colon \hat{R}^k \to \hat{R}^{k+1}$ from $\xi_I U_l$ to $\xi_{I'}U_l$ where $\A$ and $\C$ are interpreted as the formal alphabets associated to the two ``rungs'' in $U_l$ labeled $b-l$. Similarly, the coefficients of $s(\zeta_J)$ give the component maps of $s^k\colon \hat{S}^k \to \hat{S}^{k+1}$. To show that $\Psi^k\circ r^k\circ (\Psi^k)^{-1} = s^k$, it suffices to verify that \[
		\psi_{\A|\C} \circ r \circ (\psi_{\A|\C})^{-1} = s.
	\]As both sides are anti-derivations, it suffices to check the identity on $\zeta_1,\ldots,\zeta_m$. We compute \begin{align*}
		(\psi_{\A|\C} \circ r \circ (\psi_{\A|\C})^{-1}) (\zeta_i) &= (\psi_{\A|\C}\circ r)\left(\sum_{j=1}^i (-1)^{j-1} e_{i-j}(\A) \xi_j \right)\\
		&= \sum_{j=1}^i (-1)^{j-1} e_{i-j}(\A) \sum_{t = 0}^j (-1)^t h_{j-t}(\A)e_t(\C)\\
		&= \left(\sum_{j=1}^i (-1)^{j-1}e_{i-j}(\A)h_j(\A) \right) - \sum_{t=1}^i e_t(\C) \left(\sum_{j=t}^i (-1)^{j + t}e_{i-j}(\A)h_{j-t}(\A)\right)
	\end{align*}where in the last line, the first sum corresponds to the terms where $t = 0$. By Lemma~\ref{lem:generatingFunctionIdentity}, the first sum is equal to $e_i(\A)$ while \[
		\sum_{j=t}^i(-1)^{j+t}e_{i-j}(\A)h_{j-t}(\A) = \begin{cases}
			1 & i = t\\
			0 & \text{else}.
		\end{cases}
	\]Thus $(\psi_{\A|\C}\circ r\circ(\psi_{\A|\C})^{-1})(\zeta_i) = e_i(\A) - e_i(\C)$ as required. 
\end{proof}

We have finished the construction of the double complex $S$, and we turn to the proof that it is homotopy equivalent to the complex associated to a full twist on two strands. We do so by horizontally composing $S$ with a negative twist below and proving that the resulting complex is homotopy equivalent to the Rickard complex. We begin with two lemmas.

\begin{lem}\label{lem:spaceOfClassesOfMapsSktoSkplus1}
	For $\max(a + b - N,0) \leq k < \min(a,b)$, the space of homotopy classes of chain maps from $S^k$ to $h^{-1}S^{k+1}$ is isomorphic to $\Z$. 
\end{lem}
\begin{proof}
	First, note that the space of homotopy classes of chain maps from $q\llbracket W_k \rrbracket$ to $\llbracket W_{k+1} \rrbracket$ is isomorphic to $\Z$ by Lemma~\ref{lem:homSpaceCalculationShiftedRickard}. Next, let \[
		\llbracket T^+(W_k) \rrbracket \coloneq \left\llbracket \:\:\:\quad\qquad\begin{gathered}
			\vspace{-3pt}
			\labellist
			\pinlabel {\small$a$} at -3 85
			\pinlabel {\small$b$} at 41 86
			\pinlabel {\small${b-k}$} at 19 79
			\pinlabel {\small${a+b-k}$} at -21 57
			\pinlabel {\small$k$} at 43 57
			\pinlabel {\small${a-k}$} at 19 57
			\pinlabel {\small$a$} at -3 2
			\pinlabel {\small$b$} at 41 3
			\endlabellist
			\includegraphics[width=.07\textwidth]{twistWk}
		\end{gathered}\quad \right\rrbracket \qquad \llbracket T^+(W_{k+1})\rrbracket \coloneq \left\llbracket \:\:\:\:\:\qquad\qquad\begin{gathered}
			\vspace{-3pt}
			\labellist
			\pinlabel {\small$a$} at -3 85
			\pinlabel {\small$b$} at 41 86
			\pinlabel {\small${a+b-k-1}$} at -29 57
			\pinlabel {\small$k+1$} at 51 57
			\pinlabel {\small$a$} at -3 2
			\pinlabel {\small$b$} at 41 3
			\endlabellist
			\includegraphics[width=.07\textwidth]{twistWk}
		\end{gathered}\qquad \right\rrbracket
	\]To prove the lemma, note that by Proposition~\ref{prop:twistWk}, it suffices to prove that the space of homotopy classes of chain maps $q\llbracket T^+(W_k) \rrbracket \to \llbracket T^+(W_{k+1}) \rrbracket$ is isomorphic to $\Z$. Let $\Hom(q\llbracket T^+(W_k) \rrbracket,\llbracket T^+(W_{k+1}) \rrbracket)$ denote this space of homotopy classes of chain maps. 

	Consider the following composite of maps given by horizontal composition \[
		\begin{tikzcd}
			\Hom(q\llbracket W_k\rrbracket,\llbracket W_{k+1}\rrbracket) \ar[r] & \Hom(q\llbracket T^+(W_k) \rrbracket,\llbracket T^+(W_{k+1}) \rrbracket) \ar[r,"\Psi"] & \Hom(q\llbracket T^-T^+(W_k) \rrbracket,\llbracket T^-T^+(W_{k+1}) \rrbracket)
		\end{tikzcd}
	\]where $T^-$ denotes horizontal composition with a negative crossing below, analogous to the notation $T^+$. Composing this map with the isomorphism \[
		\begin{tikzcd}
			\Hom(q\llbracket T^-T^+(W_k) \rrbracket,\llbracket T^-T^+(W_{k+1}) \rrbracket) \ar[r] & \Hom(q\llbracket W_k \rrbracket, \llbracket W_{k+1}\rrbracket)
		\end{tikzcd}
	\]given by Reidemeister II invariance (Theorem~\ref{thm:RmovesandForkMoves}) gives the identity map on $\Hom(q\llbracket W_k \rrbracket, \llbracket W_{k+1}\rrbracket)$. Thus $\Psi$ is surjective. Similarly, the composite map \[
		\begin{tikzcd}
			\Hom(q\llbracket T^+(W_k) \rrbracket,\llbracket T^+(W_{k+1}) \rrbracket) \ar[r,"\Psi"] & \Hom(q\llbracket T^-T^+(W_k) \rrbracket,\llbracket T^-T^+(W_{k+1}) \rrbracket) \ar[d]\\
			\Hom(q\llbracket T^+(W_k) \rrbracket,\llbracket T^+(W_{k+1}) \rrbracket) & \Hom(q\llbracket T^+T^-T^+(W_k) \rrbracket,\llbracket T^+T^-T^+(W_{k+1}) \rrbracket) \ar[l]
		\end{tikzcd}
	\]is also the identity so $\Psi$ is injective. Thus $\Psi$ is an isomorphism which proves the result.
\end{proof}

\begin{lem}\label{lem:Eiprimitive}
	Suppose $l = \max(a + b - N,0)$. Then the maps $E_i\in \Hom(q^{2i}U_l,U_l)$ for $i = 1,\ldots,b-l$ are primitive. 
\end{lem}
\begin{proof}
	First consider the case that $\max(a + b - N,0) = 0$ so that $l = 0$. Recall that there is a homogeneous identification between $\Hom^*(U_0,U_0)$ and $\Hom^*(\emp,\ol{U_0} \cup U_0)$ where \[
		\ol{U_0} \cup U_0 = \begin{gathered}
			\labellist
			\pinlabel {\small$b$} at 35 37
			\pinlabel {\small$b$} at 81 37
			\pinlabel {\small$a$} at 35 9
			\pinlabel {\small$a$} at 81 9
			\pinlabel {\small$a+b$} at 153 17
			\endlabellist
			\includegraphics[width=.25\textwidth]{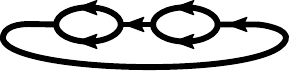}
		\end{gathered}
	\]where we have erased the edges labeled $0$. As a particular instance of Theorem~\ref{thm:RWcategorificationOfMOYCalculus}, a basis for $\Hom(\emp,\ol{U_0} \cup U_0)$ as a free $\Sym(N)$-module can be constructed using relation~\ref{item:thicknilHecke} of Proposition~\ref{prop:foamRelations}. In particular, a basis can be obtained so that under the identification $\Hom^*(U_0,U_0) = \Hom^*(\emp,\ol{U_0} \cup U_0)$, the map $E_i$ is sent to the difference of two distinct basis elements. It follows that $E_i$ is primitive. 

	The case that $\max(a + b - N,0) = a + b - N$ is handled similarly. The only difference is that in this case, the closed web $\ol{U_l} \cup \ol{U_l}$ has two edges labeled $N$. After applying the isomorphism of \cite[Claim 3.34]{MR4164001}, a similar argument shows that $E_i$ is primitive. 
\end{proof}

\begin{prop}\label{prop:fulltwist}
	There is a chain homotopy equivalence \[
		S \:\: \simeq\:\: \left\llbracket \:\:\:\begin{gathered}
		\vspace{-3pt}
		\labellist
		\pinlabel {\small$a$} at -4 60
		\pinlabel {\small$b$} at 38 61
		\pinlabel {\small$a$} at -3 3
		\pinlabel {\small$b$} at 37 4
		\endlabellist
		\includegraphics[width=.06\textwidth]{fullTwist}
	\end{gathered} \:\:\:\right\rrbracket
	\]
\end{prop}
\begin{proof}
	By Reidemeister II invariance (Theorem~\ref{thm:RmovesandForkMoves}), it suffices to show that the horizontal composition of $S$ with a negative twist below, which we denote $T^-(S)$, is homotopy equivalent to the Rickard complex associated to a positive crossing where the overstrand is labeled $b$ as in Definition~\ref{df:rickardComplex}. By Proposition~\ref{prop:twistWk}, the complex $T^-(S^k)$, the horizontal composition of the column complex $S^k$ with a negative twist below, is homotopy equivalent to the complex $h^kq^{-k}\llbracket W_k\rrbracket$. Since $h^kq^{-k}\llbracket W_k \rrbracket$ is supported in a single homological degree, there is a strong deformation retract \[
		\pi_k\colon T^-(S^k) \to h^kq^k\llbracket W_k\rrbracket, \qquad \iota_k\colon h^kq^k\llbracket W_k\rrbracket \to T^-(S^k), \qquad h_k\colon T^-(S^k) \to T^-(S^k)
	\]satisfying the side conditions of Definition~\ref{df:strongDeformationRetract}. By the homological perturbation lemma (Lemma~\ref{lem:homologicalPerturbationLemma}) for the poset $\{\:k\:|\: \max(a + b - N,0) \leq k \leq b\:\}$, there is a homotopy equivalence between $S$ and the complex \[
		\begin{tikzcd}[column sep=large]
			\cdots \ar[r] & h^kq^{-k}W_k \ar[rr,"\pi_{k+1}\circ T^-(s^k)\circ\iota_k"] & & h^{k+1}q^{-k-1}W_{k+1} \ar[r] & \cdots
		\end{tikzcd}
	\]The component of the differential from $h^kq^{-k}W_k$ to $h^{k+1}q^{-k-1}W_{k+1}$ is $\pi_{k+1}\circ T^-(s^k)\circ \iota_k$ by the formula given in the homological perturbation lemma, and all other components of the differential are zero by homological degree considerations. By Lemma~\ref{lem:shiftedDifferentialPrimitive}, it suffices to show that $\pi_{k+1}\circ T^-(s^k)\circ \iota_k$ is primitive for each $k$ satisfying $\max(a+b-N,0) \leq k \leq b - 1$. If $\pi_{k+1}\circ T^-(s^k)\circ\iota_k$ is not primitive, then the chain map $\iota_{k+1}\circ (\pi_{k+1}\circ T^-(s^k)\circ \iota_k)\circ\pi_k$, which is homotopic to $T^-(s^k)$, is also not primitive. It therefore suffices to show that $T^-(s^k)$ is a generator of the space of homotopy classes of chain maps from $T^-(S^k)$ to $h^{-1}T^-(S^{k+1})$. By the proof of Lemma~\ref{lem:spaceOfClassesOfMapsSktoSkplus1}, it suffices to show that $s^k$ is a generator of the space of homotopy classes of chain maps from $S^k$ to $h^{-1}S^{k+1}$. 

	We separate the proof into two cases. \begin{itemize}
		\item Case: $\max(a + b - N,0) \leq k \leq b - 2$. 

		Suppose $s^k$ is not a generator. Then there is a map $g\colon S^k \to h^{-1}S^{k+1}$ that decreases homological degree by one for which $s^k + d^{k+1}g + gd^k$ is not primitive. Recall that $S^k = \bigoplus_l S^k_l$ where $S^k_l$ is supported in homological degree $k + l$ and $\max(a+b-N,0) \leq l \leq k \leq b$. We let $g_l$ denote the component of $g$ from $S^k_l$ to $h^{-1}S^{k+1}_{l-1}$. 
		Let $n = \max(a+b-N,0)$ and focus on the component $s^k_n + d^{k+1}_n g_n + g_{n+1} d^k_n$ from $S^k_n$ to $h^{-1}S^{k+1}_n$. Note that $g_n = 0$ because $h^{-1}S^{k+1}_{n-1} = 0$ so the map is $s^k_n + g_{n+1}d^k_n$. \[
		 	\begin{tikzcd}
		 		S^k_{n+1} \ar[dr,"g_{n+1}"] &\\
		 		S^k_n \ar[r,"s^k_n"] \ar[u,"d_n^k"] & h^{-1}S^{k+1}_n
		 	\end{tikzcd}
		\]Now recall that $S^k_l = h^{k+1}q^\bullet \bigoplus_J \zeta_J U_l$ where $J$ ranges over subsets of $\{1,\ldots,b-l\}$ of size $b-k$. Fix a subset $K \subseteq \{1,\ldots,b-n\}$ of size $b - k$ for which $1 \in K$ and $b - n \in K$ and consider the direct summand $\zeta_K U_n$ of $S^k_n$. Note that $K$ exists because $b - k \ge 2$. The component of $s^k_n$ from $\zeta_K U_n$ to the direct summand $\zeta_{K\setminus 1} U_n$ of $h^{-1}S^{k+1}_n$ is $\pm E_1$ which is primitive by Lemma~\ref{lem:Eiprimitive}. We claim that the component of $g_{n+1}d^k_n$ from $\zeta_K U_n$ to $\zeta_{K\setminus1}U_n$ is zero. The nonzero component maps of $d^k_n$ out of $\zeta_KU_n$ are all of the form $X^\ast u_n$ where $\ast > 0$ because $b - n \in K$. By a computation exactly analogous to Lemma~\ref{lem:homSpaceCalculationShiftedRickard}, we have \[
			\rank_\Z \Hom(q^j U_n,U_{n+1}) = \rank_\Z \Hom(q^jU_{n+1},U_n) = \begin{cases}
				0 & j < a - b + 1\\
				1 & j = a - b + 1
			\end{cases}
		\]The $q$-shift of $\zeta_{K\setminus1}U_n$ is $2$ less than the $q$-shift of $\zeta_K U_n$. However, the highest value of $s$ for which there is a nonzero component map of $g_{n+1}d^k_n$ from $\zeta_KU_n$ to $q^sU_n$ is at least $(a - b + 3) + (a - b + 1)$ less than the $q$-shift of $\zeta_K U_n$. Thus, the component map of $s^k + d^{k+1}g + gd^k$ from $\zeta_K U_n$ to $\zeta_{K\setminus 1}U_n$ is $\pm E_1$, which is primitive. Thus $s^k$ is a generator. 
		\item Case: $k = b - 1$. 

		Our proof of this case is analogous to the argument in \cite[Proof of Theorem 3.24]{https://doi.org/10.48550/arxiv.2107.08117}. We must show that $s^{b-1}\colon S^{b-1} \to h^{-1}S^b$ is a generator of the space of homotopy classes of chain maps. Recall that \[
			S^{b-1}_l = h^{b-1+l}q^\bullet \bigoplus_{i=1}^{b-l} \zeta_i U_l \qquad\quad h^{-1}S^b_l = h^{b-1+l}q^\bullet U_l
		\]where $\bullet = -l(a - b + 1) + b(a - b - 1)$. Consider the map $\Theta\colon h^{-1}q^2S^b \to S^{b-1}$ whose component map from $U_l$ to $\zeta_iU_l$ is zero unless $i = 1$ in which case the component map is the identity map. It is easy to see that $\Theta$ is a chain map. The component of $s^{b-1}\circ \Theta$ from $h^{-1}q^2 S^b_l$ to $h^{-1}S^b_l$ is \[
			E_1 = H_1 = \:\:\begin{gathered}
			\labellist
			\pinlabel {\large$\bullet$} at 20 36
			\pinlabel {\small$e_1$} at 29 36
			\endlabellist
			\includegraphics[width=.035\textwidth]{smallLadder}
			\end{gathered}\:\: - \:\:\begin{gathered}
			\labellist
			\pinlabel {\large$\bullet$} at 20 2
			\pinlabel {\small$e_1$} at 29 2
			\endlabellist
			\includegraphics[width=.035\textwidth]{smallLadder}
			\end{gathered} \quad \colon h^{b-1+l}q^{\bullet+2} U_l \to h^{b-1+l}q^\bullet U_l
		\]Horizontally composing with a negative crossing below gives us chain maps \[
			\begin{tikzcd}
				h^{-1}q^2T^-(S^{b}) \ar[rr,"T^-(\Theta)"] & & T^-(S^{b-1}) \ar[rr,"T^-(s^{b-1})"] & & h^{-1} T^-(S^b).
			\end{tikzcd}
		\]Now consider the composite map \[
			\begin{tikzcd}[column sep=large]
				h^{b-1}q^{-b+2}W_b \ar[rr,"\pi_{b-1}\circ T^-(\Theta)\circ \iota_b"] & & h^{b-1}q^{-b+1}W_{b-1} \ar[rr,"\pi_b\circ T^-(s^{b-1})\circ \iota_{b-1}"] & & h^{b-1}q^{-b}W_b.
			\end{tikzcd}
		\]Note that this composite map is equal to \[
			\pi_b T^-(s^{b-1}) \iota_{b-1}  \pi_{b-1}  T^-(\Theta)  \iota_b = \pi_b T^-(s^{b-1})T^-(\Theta)\iota_b + \pi_b  (d h_{b-1} + h_{b-1}d)\iota_b = \pi_b T^-(s^{b-1}\Theta) \iota_b
		\]where the second equality is due to the fact that the complexes $h^{b-1}q^{-b+2}W_b$ and $h^{b-1}q^{-b}W_b$ have no differential. Furthermore, $\pi_bT^-(s^{b-1}\Theta)\iota_b$ is chain homotopic to \[
			\:\:\begin{gathered}
			\labellist
			\pinlabel {\large$\bullet$} at 20 36
			\pinlabel {\small$e_1$} at 29 36
			\endlabellist
			\includegraphics[width=.035\textwidth]{smallLadder}
			\end{gathered}\:\: - \:\:\begin{gathered}
			\labellist
			\pinlabel {\large$\bullet$} at 1.5 2
			\pinlabel {\small$e_1$} at -6 2
			\endlabellist
			\includegraphics[width=.035\textwidth]{smallLadder}
			\end{gathered} \quad\colon h^{b-1}q^{-b+2}W_b \to h^{b-1}q^{-b}W_b
		\]by dot sliding (Proposition~\ref{prop:dotSliding}), which means that $\pi_bT^-(s^{b-1}\Theta)\iota_b$ is equal to this difference of dot maps since there are no nontrivial homotopies. If $\pi_b T^-(s^{b-1})\iota_{b-1}$ is not primitive, then this difference of dot maps is also not primitive. However, an argument analogous to the proof of Lemma~\ref{lem:Eiprimitive} shows that this difference of dots map is indeed primitive. 
		\qedhere
	\end{itemize}
\end{proof}

\subsection{The trefoil complex}\label{subsec:trefoilComplex}

In this section, we describe a simple complex $\ol{T}$ that is homotopy equivalent to the complex associated to the trefoil labeled $a$. 
For each pair of integers $k,l$ satisfying $\max(2a-N,0) \leq l \leq k \leq a$, define \[
	\smash{\ol{T}}_k^{\,l}\coloneq h^{k + 2l} q^{-a-2l+a^2-lN} \bigoplus_{J} \zeta_J \Theta_l \qquad\text{where }\Theta_l\coloneq\qquad\:\: \begin{gathered}
			\vspace{-3pt}
			\centering			
			\labellist
			\pinlabel {\small${a-l}$} at 28 82
			\pinlabel {\small$l$} at 59 44
			\pinlabel {\small${2a-l}$} at -16 46
			\pinlabel {\small$a$} at 33 55
			\pinlabel {\small${a - l}$} at 35 17
			\pinlabel {\small$a$} at 70 5
			\endlabellist
			\includegraphics[width=.13\textwidth]{Theta}
		\end{gathered}\:
\]and where the direct sum is over subsets $J \subseteq \{1,\ldots,a-l\}$ of size $|J| = a-k$. Just as before, $\zeta_J = \zeta_{j_1}\cdots\zeta_{j_{a-k}}$ and $\zeta_j$ carries a $q$-shift of $q^{2j}$. Define a differential $\smash{\ol{T}}^{\,l}_k \to \smash{\ol{T}}^{\,l}_{k+1}$ by declaring that the component map from $\zeta_J\Theta_l$ to $\zeta_{J'}\Theta_l$ is zero unless $J'$ is obtained from $J$ by deleting an element $j_i \in J = \{j_1 < \cdots < j_{a-k}\}$, in which case define the component map to be $(-1)^{i-1}F_{j_i}$ where $F_j \in \Hom(q^{2j}\Theta_l,\Theta_l)$ is the following difference of dot maps \[
	F_j \coloneq \:\:\begin{gathered}
			\labellist
			\pinlabel {\large$\bullet$} at 27 72
			\pinlabel {\small$e_j$} at 32 60
			\endlabellist
			\includegraphics[width=.09\textwidth]{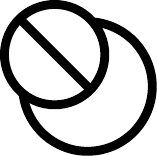}
		\end{gathered} \:\: - \:\: \begin{gathered}
			\labellist
			\pinlabel {\large$\bullet$} at 27 24
			\pinlabel {\small$e_j$} at 37 15
			\endlabellist
			\includegraphics[width=.09\textwidth]{mediumTheta}
		\end{gathered}
\]Let $\ol{T}\coloneq \bigoplus_{k,l} \smash{\ol{T}}^{\,l}_k$ equipped with this differential. 

\begin{thm}\label{thm:slncomplexTrefoil}
	If $0 \leq a \leq N$, then there is a homotopy equivalence \[
		\left\llbracket \:\:\begin{gathered}
			\vspace{-3pt}
			\centering
			\labellist
			\pinlabel {\small$a$} at 57 50
			\endlabellist
			\includegraphics[width=.13\textwidth]{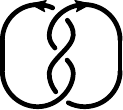}
		\end{gathered}\:\: \right\rrbracket \:\:\simeq\:\: \ol{T}
	\]
\end{thm}
\begin{proof}
	Consider the complex $S$ defined in section~\ref{subsec:fullTwist} in the special case that $b = a$. Just as the given diagram of the trefoil is obtained from the full twist on two strands by horizontally composing with a positive crossing below and closing off on both sides, let $T$ denote the complex obtained from $S$ by horizontally composing with a positive crossing below and closing off on both sides. By Proposition~\ref{prop:fulltwist}, the complex associated to the given diagram of the trefoil is homotopy equivalent to $T$. 
	We describe $T = \bigoplus_{k,l} T^l_k$ explicitly. We have \[
		T^l_k = h^{k+l}q^{-l-a} \bigoplus_J\zeta_J \left\llbracket \:\:\begin{gathered}
			\vspace{-3pt}
			\centering
			\labellist
			\pinlabel {\small$a$} at 29 83
			\pinlabel {\small$a$} at 74 84
			\pinlabel {\small${a-l}$} at 51 84
			\pinlabel {\small$l$} at 75 55
			\pinlabel {\small${a-l}$} at 51 54
			\pinlabel {\small$a$} at 71 26
			\pinlabel {\small$a$} at 31 27
			\endlabellist
			\includegraphics[width=.17\textwidth]{twistClosureWk}
		\end{gathered}\:\: \right\rrbracket 
	\]where we negate the differential of $T^l_k$ when $k + l$ is odd. The other nonzero components of the differential of $T$ are the maps $T^l_k \to T^l_{k+1}$ and $T^l_k \to T^{l+1}_k$ induced by horizontal composition from the differential on $S$. 

	By Lemma~\ref{lem:flatteningTwistClosureWk}, there is a strong deformation retract \[
		\pi^l_k \colon T^l_k \to \smash{\ol{T}}^{\,l}_k, \qquad \iota^l_k\colon \smash{\ol{T}}^{\,l}_k \to T^l_k, \qquad h^l_k\colon T^l_k \to T^l_k.
	\]We may assume that these maps preserve the internal direct sum splittings over subsets $J \subseteq \{1,\ldots,a-l\}$ of size $a-k$. We claim that the homological perturbation lemma (Lemma~\ref{lem:homologicalPerturbationLemma}) for the poset $\{\:(k,l) \:|\: \max(2a-N,0) \leq l \leq k \leq a\:\}$ gives a strong deformation retract from $T$ to $\ol{T}$. To see this, first note that the differential on $\smash{\bigoplus_{k,l}} \smash{\ol{T}}_k^{\,l}$ arising from the lemma respects the poset splitting. Since $\smash{\ol{T}}_k^{\,l}$ is supported in homological degree $k + 2l$, the only components of the differential that are possibly nontrivial are the maps $\smash{\ol{T}}_k^{\,l} \to \smash{\ol{T}}_{k+1}^{\,l}$. Other summands, such as $\smash{\ol{T}}_{k-1}^{\,l+1}$, may also be supported in the same homological degree as $\smash{\ol{T}}_{k+1}^{\,l}$, but the fact that the differential respects the poset structure implies that the component from $\smash{\ol{T}}_k^{\,l}$ to any such summand is zero. The component of the map $\smash{\ol{T}}_k^{\,l} \to \smash{\ol{T}}_{k+1}^{\,l}$ from $\zeta_J\Theta_l$ to $\zeta_{J'}\Theta_l$ is zero unless $J'$ is obtained from $J$ by deleting an element $j_i \in J = \{j_1 < \cdots < j_{a-k}\}$, in which case the component map is \[
		(-1)^{i-1} \pi^l_{k+1} \circ \left(\begin{gathered}
			\vspace{-3pt}
			\centering
			\labellist
			\pinlabel {\small$e_{j_i}$} at 53 84
			\pinlabel {\large$\bullet$} at 51 73
			\endlabellist
			\includegraphics[width=.15\textwidth]{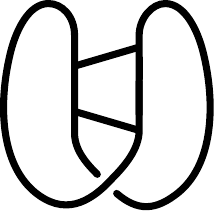}
		\end{gathered}\:\: - \:\:\begin{gathered}
			\vspace{-3pt}
			\centering
			\labellist
			\pinlabel {\small$e_{j_i}$} at 53 53
			\pinlabel {\large$\bullet$} at 51 42
			\endlabellist
			\includegraphics[width=.15\textwidth]{twistClosureWknoarrows}
		\end{gathered} \right) \circ \iota^l_k
	\]By following the sequence of the homotopies equivalences in Lemma~\ref{lem:flatteningTwistClosureWk} together with dot-sliding (Proposition~\ref{prop:dotSliding}), we find that this component map is chain homotopic to $(-1)^{i-1}F_{j_i}$ as maps from $\zeta_J\Theta_l$ to $\zeta_{J'}\Theta_l$. Since these complexes have no differential, chain homotopic maps are equal so the component map is precisely $(-1)^{i-1}F_{j_i}$ as claimed. 
\end{proof}

\section{\texorpdfstring{$\SU(N)$}{SU(N)} representation spaces}\label{sec:SUNRepSpaces}

In section~\ref{subsec:defsPrelimsRepSpaces}, we recall the notion of principal angles between a pair of vector subspaces. In section~\ref{subsec:repSpacesHopfTrefoil}, we explicitly describe the representation spaces of the Hopf link and the trefoil with arbitrary labels. The connected components of these spaces turn out to be homogeneous spaces. In section~\ref{subsec:cohomologyOfHomogeneousSpaces}, we review a theorem of Gugenheim and May \cite{MR0394720} concerning the cohomology of homogeneous spaces and prove a slight refinement.

\subsection{Principal angles}\label{subsec:defsPrelimsRepSpaces}

We recall the notion of principal angles between a pair of vector subspaces, introduced by Jordan \cite{MR1503705}. A single angle completely determines the relative position of a pair of lines passing through the origin. In the same way, the principal angles between a pair of vector subspaces are a sequence of angles that determine their relative position. 

\begin{df}
	Let $\Lambda_A$ and $\Lambda_B$ be complex vector subspaces of $\C^N$, of dimensions $a$ and $b$, respectively. Fix matrices $P_A$ and $P_B$ whose columns form orthonormal bases for $A$ and $B$, respectively. Let $c_1,\ldots,c_{\min(a,b)}$ be the singular values of the $a \x b$ matrix $(P_A)^*P_B$, where $(P_A)^*$ is the conjugate transpose of $P_A$. These singular values satisfy $1 \ge c_1 \ge \cdots \ge c_{\min(a,b)} \ge 0$. The \textit{principal angles} between $\Lambda_A$ and $\Lambda_B$ are the sequence of angles $0 \leq \theta_1 \leq \cdots \leq \theta_{\min(a,b)} \leq \pi/2$ for which $\cos(\theta_i) = c_i$. 

	Relative to fixed numbers $0 \leq a,b \leq N$, \textit{a sequence of principal angles} is a sequence $\theta_1,\ldots,\theta_{\min(a,b)}$ for which $0 \leq\theta_1 \leq \cdots \leq \theta_{\min(a,b)} \leq \pi/2$ and for which $\theta_1 = \cdots = \theta_{a + b - N} = 0$. The space of such sequences is naturally a simplex of dimension $N + \min(a,b) - a - b$. 
\end{df}

To elaborate on this definition, we recall that a singular value decomposition of $(P_A)^*P_B$ consists of unitary matrices $U \in \U(a)$ and $V \in \U(b)$ and a diagonal $a \x b$ matrix $\Sigma$ with real nonnegative entries for which $U\Sigma V = (P_A)^* P_B$. If the diagonal entries of $\Sigma$ are required to be in descending order, then $\Sigma$ is uniquely determined by $(P_A)^*P_B$. The diagonal entries of $\Sigma$ are the singular values of $(P_A)^*P_B$. The assumption that the columns of $P_A$ and $P_B$ are orthonormal implies that singular values are at most $1$. Furthermore, the number of principal angles between $\Lambda_A$ and $\Lambda_B$ that are equal to zero is precisely the dimension of $\Lambda_A \cap \Lambda_B$. Hence, if $a + b > N$, then it is always the case that $\theta_1 = \cdots = \theta_{a + b - N} = 0$, which explains the terminology for a sequence of principal angles. Any sequence of principal angles is realized by a pair of subspaces. 

It is clear from the definition that if $T \in \U(N)$, then the principal angles between $T\Lambda_A$ and $T\Lambda_B$ are the same as those between $\Lambda_A$ and $\Lambda_B$. If the principal angles between $\Lambda_A$ and $\Lambda_B$ agree with those between $\Lambda_A'$ and $\Lambda_B'$ where $\dim \Lambda_A' = \dim \Lambda_A$ and $\dim \Lambda_B' = \dim \Lambda_B$, then there is a unitary transformation $T \in \U(N)$ for which $\Lambda_A' = T\Lambda_A$ and $\Lambda_B' = T\Lambda_B$. This assertion can be verified directly. We record this observation in the following lemma. 

\begin{lem}\label{lem:unitaryOrbitsProductOfGrassmannians}
	Fix $0 \leq a,b \leq N$ and consider the diagonal action of $\U(N)$ on $\G(a,N) \x \G(b,N)$ given by \[
		T\cdot(\Lambda_A,\Lambda_B) = (T\Lambda_A,T\Lambda_B).
	\]The map sending $(\Lambda_A,\Lambda_B)$ to the sequence of principal angles between $\Lambda_A$ and $\Lambda_B$ is invariant under the diagonal action of $\U(N)$. Furthermore, it provides an identification of the orbit space of this $\U(N)$ action with the space of sequences of principal angles. 
\end{lem}

\subsection{Representation spaces of the Hopf link and the trefoil}\label{subsec:repSpacesHopfTrefoil}

Recall that for $a$ satisfying $0 \leq a \leq N$, we have a matrix \[
	\Phi_a \coloneq e^{\,a\pi i/N}\begin{pmatrix}
			-\Id_a & 0\\
			0 & \Id_{N-a}
		\end{pmatrix} \in \SU(N)
\]whose conjugacy class is denoted $C_a\subset \SU(N)$. By Definition~\ref{df:representationSpace}, if $L$ is an oriented labeled link, then the space $\sr R_N(L)$ is the set of homomorphisms $\rho\colon \pi_1(\R^3\setminus L) \to \SU(N)$ for which $\rho(\mu) \in C_{a(\mu)}$ for each meridian $\mu$. 

\begin{rem}\label{rem:identificationCkandGkN}
	There is an identification between $C_a$ and the complex Grassmannian $\G(a,N)$ given by \[
		M\in C_a \longleftrightarrow \Lambda \in \G(a,N)
	\]where $\Lambda$ is the eigenspace of $M$ associated to the eigenvalue $-e^{\,a\pi i/N}$. The eigenspace of $M$ with eigenvalue $e^{\,a\pi i/N}$ is the orthogonal complement of $\Lambda$ defined by the standard Hermitian inner product on $\C^N$, so the correspondence $M \mapsto\Lambda$ is indeed a bijection. 
\end{rem}

We first describe the representation spaces of the Hopf link labeled $0 \leq a,b \leq N$. The result follows quickly from simultaneous diagonalizability of commuting diagonalizable operators. If $a_1,\ldots,a_k$ is a sequence of nonnegative numbers, let $\F(a_1,\ldots,a_k;N)$ denote the partial flag manifold consisting of $k$-tuples $(\Lambda_1,\ldots,\Lambda_k)$ of pairwise orthogonal complex vector subspaces of $\C^N$ for which $\dim \Lambda_i = a_i$. In particular, the partial flag manifold $\F(a;N)$ is the complex Grassmannian $\G(a,N)$. Note that if $a_1 + \cdots + a_k > N$, then $\F(a_1,\ldots,a_k;N) = \emp$. 

\begin{prop}\label{prop:HopfLinkRepSpace}
	Let $L$ denote the positive Hopf link whose components are labeled by integers $0 \leq a,b \leq N$. Then \[
		\sr R_N(L) = \bigsqcup_{k=\max(a + b - N,\,0)}^{\min(a,\,b)} \F(k,a-k,b-k;N).
	\]
\end{prop}
\begin{proof}
	A presentation of the fundamental group $\pi_1$ of the complement of the Hopf link is \[
		\pi_1 = \langle\: A,B \:|\: AB = BA\: \rangle
	\]where $A$ and $B$ are meridians of components labeled $a$ and $b$, respectively. 

	Let $\rho$ be a representation in $\sr R_N(L)$, and let $M_A \coloneq \rho(A) \in C_a$ and $M_B \coloneq \rho(B) \in C_b$. Let $\Lambda_A$ and $\Lambda_B$ be the corresponding eigenspaces of $M_A$ and $M_B$, respectively, as considered in Remark~\ref{rem:identificationCkandGkN}. Consider the triple $(\Lambda_1,\Lambda_2,\Lambda_3)$ of vector subspaces of $\C^N$ where \begin{itemize}[noitemsep]
		\item $\Lambda_1$ is the intersection $\Lambda_A \cap \Lambda_B$,
		\item $\Lambda_2$ is the orthogonal complement of $\Lambda_A \cap \Lambda_B$ within $\Lambda_A$,
		\item $\Lambda_3$ is the orthogonal complement of $\Lambda_A \cap \Lambda_B$ within $\Lambda_B$. 
	\end{itemize}In particular, if $\dim\Lambda_1 = k$, then $\dim\Lambda_2 = a - k$ and $\dim \Lambda_3 = b - k$. Since $M_A$ and $M_B$ are diagonalizable and commute, they are simultaneously diagonalizable, which implies that $\Lambda_2$ and $\Lambda_3$ are orthogonal. Hence $(\Lambda_1,\Lambda_2,\Lambda_3)$ is an element of $\F(k,a-k,b-k;N)$. 

	Conversely, given any element $(\Lambda_1,\Lambda_2,\Lambda_3) \in \bigsqcup_k \F(k,a-k,b-k;N)$, let $\Lambda_A \coloneq \Lambda_1 \oplus \Lambda_2$ and $\Lambda_B \coloneq \Lambda_1 \oplus \Lambda_3$. By Remark~\ref{rem:identificationCkandGkN}, the subspaces $\Lambda_A$ and $\Lambda_B$ determine matrices $M_A \in C_a$ and $M_B \in C_b$, respectively, and orthogonality of $\Lambda_2$ and $\Lambda_3$ implies that the matrices commute. Hence $\rho(A) \coloneq M_A$ and $\rho(B) \coloneq M_B$ determines a representation $\rho \in \sr R_N(L)$. 
\end{proof}

\begin{rem}\label{rem:HopfPrincipalAngles}
	Let $(\Lambda_1,\Lambda_2,\Lambda_3) \in \F(k,a-k,b-k;N)$ and consider $\Lambda_A \coloneq \Lambda_1 \oplus \Lambda_2$ and $\Lambda_B \coloneq \Lambda_1 \oplus \Lambda_3$. The principal angles between $\Lambda_A$ and $\Lambda_B$ are $\theta_1 = \cdots = \theta_k = 0$ and $\theta_{k+1} = \cdots = \theta_{\min(a,b)} = \pi/2$. This correspondence gives an identification between $\F(k,a-k,b-k;N)$ and the space of pairs $(\Lambda_A,\Lambda_B) \in \G(a,N) \x \G(b,N)$ with the given sequence of principal angles. We also note that $\F(k,a-k,b-k;N)$ is the homogeneous space \[
		\F(k,a-k,b-k;N) = \frac{\U(N)}{\U(k) \x \U(a - k) \x \U(b-k) \x \U(N - a - b + k)}.
	\]
\end{rem}

We turn to the representation spaces of the trefoil labeled $0 \leq a \leq N$. The connected components are no longer partial flag manifolds, but they can be described conveniently using principal angles analogous to the description in Remark~\ref{rem:HopfPrincipalAngles}. 

\begin{prop}\label{prop:TrefoilRepSpace}
	Let $K$ denote the right-handed trefoil labeled by an integer $0 \leq a \leq N$. Then \[
		\sr R_N(K) = \bigsqcup_{l=\max(2a - N,\,0)}^a X_l
	\]where $X_l$ is the space of pairs $(\Lambda_A,\Lambda_B) \in \G(a,N) \x \G(a,N)$ with principal angles \[
		\theta_1 = \cdots = \theta_l = 0 \qquad \theta_{l+1} = \cdots = \theta_a = \pi/3.
	\]Each $X_l$ is an orbit of the conjugation action of $\U(N)$ on $\sr R_N(K)$ and can be explicitly described as the homogeneous space $X_l = \U(N)/K_l$ where $K_l \coloneq \U(l) \x \Delta\!\U(a - l) \x \U(N-2a + l)$ is the subgroup of block diagonal matrices \[
		\begin{pmatrix}
			U\\
			& V\\
			& & V\\
			& & & W
		\end{pmatrix} \in \U(N) \quad\qquad U \in \U(l),\quad V \in \U(a - l),\quad W \in \U(N - 2a + l).
	\]
\end{prop}
\begin{proof}
	A presentation of the fundamental group $\pi_1$ of the complement of the trefoil is \[
		\pi_1 = \langle \: A,B \:|\: ABA = BAB \:\rangle
	\]where $A$ and $B$ are meridians. 

	Let $\rho$ be a representation in $\sr R_N(K)$, and let $M_A\coloneq \rho(A) \in C_a$ and $M_B\coloneq \rho(B) \in C_b$. Let $\Lambda_A$ and $\Lambda_B$ be the associated eigenspaces of $M_A$ and $M_B$, respectively, as considered in Remark~\ref{rem:identificationCkandGkN}. Consider the inclusion \[
		\Psi\colon \sr R_N(K) \hookrightarrow \G(a,N) \x \G(a,N) \qquad \rho \mapsto (\Lambda_A,\Lambda_B).
	\]For $T \in \U(N)$, note that $\Psi(T\rho T^{-1}) = (T\Lambda_A,T\Lambda_B)$. Thus $\Psi$ realizes $\sr R_N(K)$ as a union of orbits of the diagonal action of $\U(N)$ on $\G(a,N) \x \G(a,N)$. By Lemma~\ref{lem:unitaryOrbitsProductOfGrassmannians}, the orbit space of $\G(a,N) \x \G(a,N)$ is precisely the space of sequences of principal angles. To prove the first part of the proposition, it suffices to show $\Psi(\sr R_N(K))$ is precisely the union of the orbits corresponding to the sequences of principal angles $\theta_1,\ldots,\theta_a$ for which $3\theta_i \in \Z\pi$. The inequalities $0 \leq \theta_i \leq \pi/2$ then imply that $\theta_i$ is either $0$ or $\pi/3$.

	Let $\theta_1,\ldots,\theta_a$ be a sequence of principal angles, and let $l$ be the number for which $\theta_1 = \cdots = \theta_l = 0$ while $\theta_{l+1} > 0$. Consider the $a$-dimensional subspaces of $\C^N$ given by \begin{align*}
		\Lambda_A &= \mathrm{span}(e_1,e_2,\ldots,e_l) \oplus \mathrm{span}(e_{l+1},e_{l+3},\ldots,e_{l + 2(a - l) - 1})\\
		\Lambda_B &= \mathrm{span}(e_1,e_2,\ldots,e_l)\oplus \mathrm{span}(c_{l+1} e_{l+1} + s_{l+1} e_{l+2}, c_{l+2}e_{l+3} + s_{l+2}e_{l+4},\ldots,c_a e_{2a - l - 1} + s_a e_{2a - l})
	\end{align*}where $e_1,\ldots,e_N$ is the standard basis of $\C^N$, and $c_i = \cos\theta_i$ and $s_i = \sin\theta_i$. By a computation using the angle-difference formulas for sine and cosine, the matrices $M_A,M_B \in C_a$ corresponding to $\Lambda_A,\Lambda_B$, respectively, are given by \begin{align*}
		e^{-a\pi i/N}M_A &= (-\Id_l) \oplus \bigoplus_{i=1}^{a - l} \begin{pmatrix}
			-1 & 0\\0 & 1
		\end{pmatrix} \oplus \Id_{N-2a+l}\\
		e^{-a\pi i/N}M_B &= (-\Id_l) \oplus \bigoplus_{i=1}^{a - l} \begin{pmatrix}
			-\cos(2\theta_{l+i}) & -\sin(2\theta_{l+i})\\
			-\sin(2\theta_{l+i}) & \cos(2\theta_{l+i})
		\end{pmatrix} \oplus \Id_{N-2a+l}
	\end{align*}where the direct sum notation denotes a block diagonal matrix. Because $M^{-1} = e^{-2a\pi i/N}M$ for every $M \in C_a$, the relation $M_AM_BM_A = M_BM_AM_B$ is equivalent to \[
		(e^{-2a\pi i/N}M_BM_A)^3 = \Id_N\!.
	\]By direct computation using the angle-sum formulas for sine and cosine, we find that \[
		(e^{-2a\pi i/N}M_BM_A)^3 = \Id_l \oplus \bigoplus_{i=1}^{a - l} \begin{pmatrix}
			\cos(6\theta_{l+i}) & -\sin(6\theta_{l+i})\\
			\sin(6\theta_{l+i}) & \cos(6\theta_{l+i})
		\end{pmatrix} \oplus \Id_{N - 2a + l}.
	\]It follows that $M_A$ and $M_B$ determine a representation $\rho$ if and only if $3\theta_i \in \Z\pi$ as claimed. 

	To explicitly describe the orbit $X_l$ corresponding to $\theta_1 = \cdots = \theta_l = 0$ and $\theta_{l+1} = \cdots = \theta_a = \pi/3$, we must determine the stabilizer of $(\Lambda_A,\Lambda_B) \in \G(a,N) \x \G(a,N)$ under the diagonal $\U(N)$ action, where $\Lambda_A$ and $\Lambda_B$ are any subspaces with the given principal angles. We now choose $\Lambda_A$ and $\Lambda_B$ to be the following column spaces, denoted by square brackets, \[
		\Lambda_A = \begin{bmatrix}
			\Id_a\\
			0
		\end{bmatrix} = \begin{bmatrix}
			\Id_l & 0\\
			0 & \Id_{a - l}\\
			0 & 0\\
			0 & 0
		\end{bmatrix} \qquad \Lambda_B = \begin{bmatrix}
			\Id_l & 0\\
			0 & \Id_{a - l}\\
			0 & \!\sqrt{3}\Id_{a - l}\\
			0 & 0
		\end{bmatrix}
	\]where $\!\sqrt{3}$ arises as $\tan(\pi/3)$. The stabilizer of $\Lambda_A$ is the subgroup of $\U(N)$ consisting of the block diagonal matrices $\U(a) \x \U(N-a)$. If $X \in \U(a)$ and $Y \in \U(N-a)$, then $X \oplus Y$ fixes $\Lambda_B$ if and only if \[
		Y \begin{pmatrix}
			0 & \!\sqrt{3} \Id_{a-l}\\
			0 & 0
		\end{pmatrix} = \begin{pmatrix}
			0 & \!\sqrt{3} \Id_{a-l}\\
			0 & 0
		\end{pmatrix}X.
	\]After writing $X$ and $Y$ as block matrices, a straightforward computation implies that the above condition holds if and only if \[
		X = \begin{pmatrix}
			U & 0\\
			0 & V
		\end{pmatrix} \qquad Y = \begin{pmatrix}
			V & 0\\
			0 & W
		\end{pmatrix}
	\]for $U \in \U(l)$, $V \in \U(a - l)$, and $W \in \U(N - 2a + l)$. 
\end{proof}

\begin{rem}
	We comment on the representation spaces associated to labeled $2$-bridge knots and links. For any positive number $m$, let $L$ denote the $(2,m)$ torus knot or link labeled by $0 \leq a,b \leq N$, with the understanding that $a = b$ if $m$ is odd. A similar calculation shows that $\sr R_N(L)$ can be identified with the union of orbits in $\G(a,N) \x \G(b,N)$ under the diagonal $\U(N)$ action corresponding to the sequences of principal angles $\theta_1,\ldots,\theta_{\min(a,b)}$ for which $m\theta_i \in \Z\pi$. 

	The fundamental group of the complement of any $2$-bridge knot or link $L$ has a presentation with just two generators, where both generators are meridians. It follows that $\sr R_N(L)$ embeds in $\G(a,N) \x \G(b,N)$ as a union of orbits. Using the fact that $M^{-1} = e^{-2a\pi i/N}M$ for every $M \in C_a \subset \SU(N)$, one can show from an explicit presentation that the corresponding $(2,m)$ torus knot or link $L'$ with the same determinant has the property that the images of $\sr R_N(L)$ and $\sr R_N(L')$ in $\G(a,N) \x \G(b,N)$ are the same. 
\end{rem}

\subsection{Cohomology of homogeneous spaces}\label{subsec:cohomologyOfHomogeneousSpaces}

Let $K$ be a closed Lie subgroup of a Lie group $G$, and assume both groups are compact and connected. Recall that $\B G$ is the quotient of a contractible space $\E G$ by a free action of $G$. The quotient of $\E G$ by the induced free action of $K$ yields $\B K$ with a natural map to $\B G$, giving a fiber bundle \[
	G/K \to \B K \to \B G
\]The Eilenberg--Moore spectral sequence associated to this bundle converges to $H^*(G/K)$ and has $E_2$-page \[
	\Tor_{H^*(\B G)}(\Z,H^*(\B K))
\]where $H^*(\B K)$ is a graded module over the graded commutative ring $H^*(\B G)$ by the map induced by the projection $\B K \to \B G$, and $\Z$ is a graded module over $H^*(\B G)$ by the map induced by inclusion of a point into $\B G$. 

The following theorem of Gugenheim and May \cite{MR0394720} gives general conditions on $G$ and $K$ for the collapse of the Eilenberg--Moore spectral sequence. We state their theorem with coefficients in $\Z$ for simplicity, though their result holds for any Noetherian coefficient ring with slightly different hypotheses on $G$ and $K$. The following hypotheses on $G$ and $K$ hold when they are isomorphic to products of unitary groups. 

\begin{thm}[{\cite[Theorem A]{MR0394720}}]\label{thm:gugenheimMay}
	Let $K$ be a closed connected Lie subgroup of a compact connected Lie group $G$. If $H^*(\B G)$ is a polynomial algebra and $H^*(K)$ has no torsion, then there is an isomorphism of graded abelian groups\[
		H^*(G/K) \cong \Tor_{H^*(\B G)}(\Z,H^*(\B K))
	\]where the bigrading on the right-hand side is collapsed to a single grading. 
\end{thm}

This result may be viewed as the assertion that the Eilenberg--Moore spectral sequence collapses at the $E_2$ page, with the additional claim that there are no additive extension issues between $H^*(G/K)$ and the $E_\infty$-page of the spectral sequence. 

We briefly review the basic properties of $\Tor$ for graded modules over a graded commutative ring. For a reference, see for example \cite[Sections 1-2]{MR219085}. The graded commutative rings that arise in our applications are supported in even degrees and are therefore genuinely commutative. If $S$ and $T$ are graded modules over a graded commutative ring $R$, then $\Tor_R(S,T)$ is the bigraded group defined in the following way. First, we choose a projective resolution \[
	\begin{tikzcd}
	 	\cdots \ar[r] & Q_{-1} \ar[r] & Q_0 \ar[r] & S
	 \end{tikzcd}
\]of $S$ by graded $R$-modules $Q_i$. Then, we consider the chain complex \[
	\begin{tikzcd}
	 	\cdots \ar[r] & Q_{-1} \otimes_R T \ar[r] & Q_0 \otimes_R T \ar[r] & 0
	 \end{tikzcd}
\]where $Q_i \otimes_R T$ is declared to be in homological degree $i \leq 0$. The homology of this complex is then defined to be $\Tor_R(S,T)$, and it is equipped with a homological grading $h$ and an internal grading $q$ arising from the grading on $Q_i \otimes_R T$. Tensoring a projective resolution of $T$ with $S$ yields a homotopy equivalent complex. $\Tor$ is functorial in its three parameters. In particular, if $T$ or $S$ is a commutative ring, then functoriality gives $\Tor_R(S,T)$ the structure of a $T$- or $S$-module, respectively. 

We will need a slightly stronger version of the isomorphism in Theorem~\ref{thm:gugenheimMay} in order to identify the module structure on colored $\sl(N)$ homology with the module structure on the cohomology of the space of $\SU(N)$ representations. Both sides of the isomorphism of Theorem~\ref{thm:gugenheimMay} are naturally modules over $H^*(\B K)$. The structure on the left-hand side is given by the fiber inclusion in the bundle $G/K \hookrightarrow \B K \to \B G$ while the structure on the right-hand side arises from functoriality of $\Tor$ mentioned above. We verify that the isomorphism of Theorem~\ref{thm:gugenheimMay} is an isomorphism of $H^*(\B K)$-modules. The argument amounts to following the proof of \cite[Theorem A]{MR0394720} and checking compatibility of the module structure along the way.  

\begin{prop}\label{prop:moduleStructGugenheimMay}
	The isomorphism of Theorem~\ref{thm:gugenheimMay} is an isomorphism of $H^*(\B K)$-modules. 
\end{prop}

\begin{proof}
	Eilenberg and Moore show that \[
		H^*(G/K) \cong \Tor_{C^*(\B G)}(\Z, C^*(\B K))
	\]where $C^*(\B K)$ denotes the space of cochains on $\B K$, viewed as a differential graded module over the differential graded algebra $C^*(\B G)$, and the $\Tor$ group here is the version of $\Tor$ for differential graded modules. See \cite[Chapter 7]{MR1793722}. This version of $\Tor$ is also defined as the homology of a chain complex which has $C^*(\B K)$ as a tensor factor. The chain complex is a module over $Z^*(\B K)$, the group of cocycles on $\B K$. After taking homology of the complex, the action of $Z^*(\B K)$ descends to an action of $H^*(\B K)$, thereby giving $\Tor_{C^*(\B G)}(\Z, C^*(\B K))$ the structure of a $H^*(\B K)$-module. It is clear from the definition of Eilenberg and Moore's isomorphism \cite[Theorem 7.14]{MR1793722} that it respects $H^*(\B K)$-module structures since it is induced by a chain map to $C^*(G/K)$ that is $Z^*(\B K)$-equivariant. 

	Our goal then is to show that \[
		\Tor_{C^*(\B G)}(\Z, C^*(\B K)) \cong \Tor_{H^*(\B G)}(\Z, H^*(\B K))
	\]as $H^*(\B K)$-modules. The chain complex whose homology is $\Tor_{C^*(\B G)}(\Z, C^*(\B K))$ is the total complex of a double complex and thereby carries a filtration. The Eilenberg--Moore spectral sequence is defined to be the spectral sequence associated to this filtered complex. In particular, the $E_2$-page of the spectral sequence is $\Tor_{H^*(\B G)}(\Z, H^*(\B K))$. The induced filtration on $\Tor_{C^*(\B G)}(\Z, C^*(\B K))$ is a filtration of $H^*(\B K)$-modules, and its associated graded object is isomorphic to the $E_\infty$-page of the spectral sequence as $H^*(\B K)$-modules. Gugenheim and May prove that $E_2 = E_\infty$ and that $\Tor_{C^*(\B G)}(\Z, C^*(\B K))$ is isomorphic to its associated graded object as abelian groups. Our goal is therefore equivalent to verifying that $\Tor_{C^*(\B G)}(\Z, C^*(\B K))$ is isomorphic to its associated graded object as $H^*(\B K)$-modules. 

	We proceed following Gugenheim and May. The first step, due to Baum \cite{MR219085}, is to pass to a maximal torus $T$ within $K$. See for example \cite[Proof of Theorem 4.3]{MR0394720}. Associated to the inclusion of $T$ into $G$ is $\Tor_{C^*(\B G)}(\Z, C^*(\B T))$, which we may view as a $H^*(\B K)$-module by restricting the $H^*(\B T)$-module structure using the map $H^*(\B K) \to H^*(\B T)$ induced by $\B T \to \B K$. The cochain level map $C^*(\B K) \to C^*(\B T)$ induces a filtered $H^*(\B K)$-module map \[
		\Tor_{C^*(\B G)}(\Z, C^*(\B K)) \to \Tor_{C^*(\B G)}(\Z, C^*(\B T))
	\]which is an inclusion onto a $H^*(\B K)$-module direct summand as explained in \cite[Proof of Theorem 4.3]{MR0394720}. We claim that if $\Tor_{C^*(\B G)}(\Z, C^*(\B T))$ is isomorphic to its associated graded object as $H^*(\B K)$-modules, then $\Tor_{C^*(\B G)}(\Z, C^*(\B K))$ is also isomorphic to its associated graded object as $H^*(\B K)$-modules. We prove the claim by induction on filtration level, where the inductive step is the following statement which is straightforward to verify. Suppose $0 \to A \to B \to C \to 0$ and $0 \to A' \to B' \to C' \to 0$ are short exact sequences of $H^*(\B K)$-modules, and suppose there is a map of short exact sequences \[
		\begin{tikzcd}
			0 \ar[r] & A' \ar[r] & B' \ar[r] & C' \ar[r] & 0\\
			0 \ar[r] & A \ar[r] \ar[u] & B \ar[r] \ar[u] & C \ar[r] \ar[u] & 0
		\end{tikzcd}
	\]where the maps $A \to A'$ and $C \to C'$ are inclusions onto $H^*(\B K)$-module direct summands. If the top sequence splits, then the bottom sequence splits as well and the map $B \to B'$ is an inclusion onto a $H^*(\B K)$-module direct summand. The base case of induction is handled by the fact that the filtrations are of finite length.

	Our aim now is to show that $\Tor_{C^*(\B G)}(\Z, C^*(\B T))$ is isomorphic to its associated graded object as $H^*(\B T)$-modules, from which it follows that they are isomorphic as $H^*(\B K)$-modules. This is equivalent to showing that \[
		\Tor_{C^*(\B G)}(\Z, C^*(\B T)) \cong \Tor_{H^*(\B G)}(\Z, H^*(\B T))
	\]as $H^*(\B T)$-modules. Gugenheim and May show that they are isomorphic as abelian groups. We verify that their isomorphism respects $H^*(\B T)$-module structures. 

	In \cite[Theorem 4.1]{MR0394720}, Gugenheim and May define a map \[
		g\colon C^*(\B T) \to H^*(\B T)
	\]of differential graded algebras that induces the identity on homology. The composite \[
		\begin{tikzcd}
			C^*(\B G) \ar[r] & C^*(\B T) \ar[r,"g"] & H^*(\B T)
		\end{tikzcd}
	\]gives $H^*(\B T)$ the structure of a differential graded module over $C^*(\B G)$. The map $g$ induces an isomorphism \[
		\Tor_{C^*(\B G)}(\Z, C^*(\B T)) \to \Tor_{C^*(\B G)}(\Z, H^*(\B T))
	\]which is clearly a $H^*(\B T)$-module map. Gugenheim and May show in \cite[Theorem 2.3]{MR0394720} that \[
		\Tor_{C^*(\B G)}(\Z, H^*(\B T)) \cong \Tor_{H^*(\B G)}(\Z, H^*(\B T))
	\]by explicitly constructing a chain complex of $H^*(\B T)$-modules that simultaneously computes both homology groups. Thus the groups are isomorphic as $H^*(\B T)$-modules which completes the proof. 
\end{proof}

\section{Proofs of the main results}\label{sec:proofsofMainResults}

\subsection{Dot maps and the cohomology of partial flag manifolds}

Fix integers satisfying $\max(a + b - N,0) \leq k \leq \min(a,b)$. In this section, we give an explicit isomorphism between the state space of \vspace{2pt} \[
	\Theta \coloneq \qquad\qquad\begin{gathered}
			\centering
			\labellist
			\pinlabel {\small${b-k}$} at 26 81
			\pinlabel {\small$k$} at 59 44
			\pinlabel {\small${a+b-k}$} at -22 46
			\pinlabel {\small$a$} at 33 55
			\pinlabel {\small${a - k}$} at 35 18
			\pinlabel {\small$b$} at 70 5
			\endlabellist
			\includegraphics[width=.13\textwidth]{Theta}
		\end{gathered}
\]and the $\U(N)$-equivariant cohomology of the partial flag manifold $\F(k,a-k,b-k;N)$. The web $\Theta$ appears in our simplified complexes for the Hopf link given in Theorem~\ref{thm:slncomplexofHopfLink}. In the special case that $a = b$, the web also appears in our simplified complex of the trefoil given in Theorem~\ref{thm:slncomplexTrefoil}. The key feature of our explicit isomorphism is that it intertwines certain natural endomorphisms defined on each side. The argument is an elaboration of \cite[Proposition 4.12]{MR4164001}. 

The natural endomorphisms on the state space of $\Theta$ are given by the dot maps \[
	\Xdoti \hspace{30pt} \Ydoti \hspace{30pt} \Zdoti
\]The relevance of these dot maps can be seen from the definition of the simplified complex of the trefoil given in section~\ref{subsec:trefoilComplex}. 
As for the natural endomorphisms on the $\U(N)$-equivariant cohomology of $\F(k,a-k,b-k;N)$, first note that this partial flag manifold is the homogeneous space \[
	\F(k,a-k,b-k;N) = \frac{\U(N)}{\U(k) \x \U(a -k) \x \U(b-k)\x \U(N-a-b+k)}
\]The following lemma is well-known and follows from the definitions.
\begin{lem}\label{lem:equivariantCohomologyHomogeneousSpace}
	Let $K$ be a closed connected subgroup of a Lie group $G$. The $G$-equivariant cohomology of the homogeneous space $G/K$ is isomorphic to $H^*(\B K)$ where the $H^*(\B G)$-module structure is given by the ring map $H^*(\B G) \to H^*(\B K)$ induced by $\B K \to \B G$. 
\end{lem}
\begin{proof}
	The $G$-equivariant cohomology of $G/K$ is defined to be the singular cohomology of $\E G \x_G G/K$. It is straightforward to see that there is an identification between $\E G \x_G G/K$ and the space of left cosets $\E G/K$ of the free action of $K$ on $\E G$. Hence $\E G \x_G G/K = \E G/K = \B K$. 
\end{proof}

The $\U(N)$-equivariant cohomology of $\F(k,a-k,b-k;N)$ is therefore \begin{align*}
	H^*_{\U(N)}(\F(k,a-k,b-k;N)) &\cong H^*(\BU(k)) \otimes H^*(\BU(a-k)) \otimes H^*(\BU(b-k)) \otimes H^*(\BU(N-a-b+k))\\
	&\cong \Sym(k) \otimes \Sym(a-k) \otimes \Sym(b-k) \otimes \Sym(N-a-b+k).
\end{align*}We view this ring as the ring of polynomials in $N$ variables that are symmetric in the first $k$ variables, the next $a-k$ variables, the next $b-k$ variables, and the last $N - a -b+k$ variables. Define \[
	x_i\coloneq e_i \otimes 1 \otimes 1 \otimes 1, \quad y_i \coloneq 1 \otimes e_i \otimes 1 \otimes 1, \quad z_i\coloneq 1\otimes 1\otimes e_i \otimes 1, \quad w_i\coloneq 1\otimes 1 \otimes 1 \otimes e_i
\]where $e_i$ denotes the elementary symmetric polynomial in the appropriate number of variables. The action of $H^*(\BU(N)) = \Sym(N)$ is induced by the natural inclusion of $\Sym(N)$ into this space of partially symmetric polynomials. Explicitly, the $m$th elementary symmetric polynomial in $N$ variables is identified with \[
	\sum_{i+j+k+l = m} x_i\cdot y_j\cdot z_k\cdot w_l \in H^*_{\U(N)}(\F(k,a-k,b-k;N)).
\]

\begin{prop}\label{prop:generalizedThetaPartialFlag}
	There is a homogeneous isomorphism of graded modules over $\Sym(N) = H^*(\BU(N))$\[
		H^*_{\U(N)}(\F(k,a-k,b-k;N)) \to \sr F(\Theta)
	\]of degree $-\frac12\dim \F(k,a-k,b-k;N)$ that intertwines the endomorphisms given by multiplication by $x_i,y_i,z_i$ with the dot map endomorphisms \begin{align*}
		\Xdoti, \Ydoti, \Zdoti
	\end{align*}respectively.
\end{prop}

\begin{proof}
	The result of applying the MOY calculus relation~(\ref{eq:forkRelation}) given in Figure~\ref{fig:MOYcalculus} in a neighborhood of the edge labeled $a$ in $\Theta$ is \[
		\Theta' \coloneq \qquad\qquad\begin{gathered}
			\centering
			\labellist
			\pinlabel {\small$a+b-k$} at -15 23
			\pinlabel {\small$k$} at 56 23
			\pinlabel {\small${b-k}$} at 85 23
			\pinlabel {\small${a-k}$} at 36 23
			\endlabellist
			\includegraphics[width=.2\textwidth]{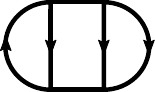}
		\end{gathered}
	\]The isomorphism $\sr F(\Theta') \to \sr F(\Theta)$ induced by this move (see Theorem~\ref{thm:RWcategorificationOfMOYCalculus} and relation~\ref{item:matveevPiergalini} of Proposition~\ref{prop:foamRelations}) intertwines \[
		\Xprimedoti, \Yprimedoti, \Zprimedoti \quad\leftrightarrow \quad \Xdoti, \Ydoti, \Zdoti
	\]respectively. Next, consider the web \[
		\Theta'' \coloneq \quad\begin{gathered}
			\centering
			\labellist
			\pinlabel {\small$N$} at 10 24
			\pinlabel {\small$m$} at 31 23
			\pinlabel {\small$a-k$} at 61 24
			\pinlabel {\small$k$} at 80 24
			\pinlabel {\small$b-k$} at 108 24
			\endlabellist
			\includegraphics[width=.25\textwidth]{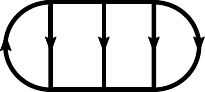}
		\end{gathered}
	\]where $m = N - a - b + k$. Let $G\colon \Theta'' \to \Theta'$ be the foam which is given by \[
		\begin{gathered}
			\centering
			\labellist
			\pinlabel {\small${a + b - k}$} at 14 40
			\pinlabel {\small$m$} at 20 16
			\pinlabel {\small$N$} at 20 5
			\endlabellist
			\includegraphics[width=.15\textwidth]{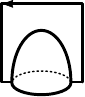}
		\end{gathered}
	\]in a neighborhood of the disc bounded by the edges labeled $N$ and $m$ and is the identity otherwise. By \cite[Equation 13 of Proposition 3.32]{MR4164001}, the foam $G$ induces an isomorphism on state spaces. It is clear that this isomorphism intertwines \[
		\Xprimeprimedoti, \Yprimeprimedoti, \Zprimeprimedoti \quad\leftrightarrow \quad \Xprimedoti, \Yprimedoti, \Zprimedoti
	\]respectively.

	We define an action of $H^*_{\U(N)}(\F(k,a-k,b-k;N))$ on $\sr F(\Theta'')$. Let $x_i,y_i,z_i,w_i$ act by the dot maps \[
		\Xprimeprimedoti, \Yprimeprimedoti, \Zprimeprimedoti, \Wprimeprimedoti
	\]respectively. The inclusion $\Sym(N)\subseteq H^*_{\U(N)}(\F(k,a-k,b-k;N))$ induces an action of $\Sym(N)$ on $\sr F(\Theta)$. By the dot-migration relation (foam relation~\ref{item:dotmigrationrelation}), the $i$th elementary symmetric polynomial in $N$ variables acts by the dot map \[
		\Nprimeprimedoti
	\]If $H$ is a closed foam and $H'$ is obtained by adding a dot of weight $1 \leq i \leq N$ to a facet of $H$ labeled $N$, then the Robert--Wagner evaluation satisfies \[
		\langle H' \rangle = e_i(X_1,\ldots,X_N)\langle H\rangle.
	\]Thus, the $\Sym(N)$-action on $\sr F(\Theta)$ induced by the inclusion $\Sym(N) \subseteq H^*_{\U(N)}(\F(k,a-k,b-k;N))$ agrees with the natural $\Sym(N)$-action on $\sr F(\Theta)$. 

	We now define a $H^*_{\U(N)}(\F(k,a-k,b-k;N))$-module isomorphism \[
		\Phi\colon H^*_{\U(N)}(\F(k,a-k,b-k;N)) \to \sr F(\Theta'')
	\]by an elaboration of \cite[Proposition 4.12]{MR4164001}. To specify such a map, it suffices to specify $\Phi(1)$. We let $\Phi(1)$ be the standard foam $F$ from $\emp$ to $\Theta''$ with no dots. This standard foam, which has seven facets, three seams, and no singular points, can be obtained by starting with a disc bounding the outer circle of $\Theta''$ and attaching three half discs in the simplest way. See \cite[Proposition 4.12]{MR4164001} for a figure. 

	We claim that $\Phi$ is surjective. Since $\Theta''$ can be reduced to a single circle using just the bigon relation of MOY calculus (relation~\ref{eq:bigonRelation}), foam relation~\ref{item:thicknilHecke} implies $\sr F(\Theta'')$ is spanned by the standard foam $F$ with varying collections of dots. The dot-migration then implies that $\Phi$ is surjective. By MOY calculus, $H^*_{\U(N)}(\F(k,a-k,b-k;N))$ and $\sr F(\Theta)$ are free $\Sym(N)$-modules of the same rank. Since finitely-generated modules over a Noetherian ring are Hopfian, it follows that $\Phi$ is an isomorphism. 
\end{proof}

\subsection{Koszul resolutions}

In this section, we prove the main results stated in the introduction. We first recall the basics of Koszul resolutions. 
Let $R$ be a commutative ring, and let $e_1,\ldots,e_n$ be a basis for $R^n$. Let $\Lambda^k R^n$ denote the $k$th exterior power of $R^n$, and recall that a basis for $\Lambda^k R^n$ as a free $R$-module consists of the elements \[
	e_I \coloneq e_{i_1} \wedge \cdots \wedge e_{i_k} \qquad 1 \leq i_1 < \cdots < i_k \leq n.
\]

\begin{df}
	Let $x_1,\ldots,x_n \in R$ be a sequence of elements. The \textit{Koszul complex} $K(x_1,\ldots,x_n)$ associated to the sequence is the chain complex of free $R$-modules \[
		\begin{tikzcd}
			0 \ar[r] & \Lambda^n R^n \ar[r,"d"] & \Lambda^{n-1} R^n \ar[r,"d"] & \cdots \ar[r,"d"] & \Lambda^1 R^n \ar[r,"d"] & \Lambda^0 R^n \ar[r] & 0
		\end{tikzcd}
	\]where $\Lambda^k R^n$ is declared to lie in homological grading $-k$. The differential $d$ on $\Lambda^1 R^n$ is given by \[
		d(e_i) = a_i
	\]and is extended to $\Lambda^k R^n$ by the Leibniz rule $d(\alpha\wedge\beta) = d\alpha \wedge \beta + (-1)^{\deg \alpha} \alpha \wedge d\beta$. In particular, the component of the differential from $R e_I$ to $R e_{I'}$ where $I'$ is obtained from $I$ by deleting an element $i_j \in I = \{i_1 < \cdots < i_k\}$ is given by multiplication by $(-1)^{j-1} a_{i_j}$. 

	If $R$ is graded and $x_i$ has degree $n_i$, then we define $e_i$ to also have degree $n_i$ so that the differential of the Koszul complex $K(x_1,\ldots,x_n)$ preserves grading. 
\end{df}

\begin{df}
	A sequence of elements $x_1,\ldots,x_n$ in a commutative ring $R$ is \textit{regular} if $R/(x_1,\ldots,x_n)\neq 0$ and if $x_i$ is not a zero divisor in $R/(x_1,\ldots,x_{i-1})$ for $i = 1,\ldots,n$.
\end{df}

The following proposition is well-known. See for example \cite[Theorem 4.6]{MR1878556}. 

\begin{prop}\label{prop:regularSeqKoszulResolution}
	If $x_1,\ldots,x_n$ is a regular sequence of elements in a commutative ring $R$, then the Koszul complex $K(x_1,\ldots,x_n)$ is a free resolution of the $R$-module $R/(x_1,\ldots,x_n)$. In other words, the complex \[
		\begin{tikzcd}
			0 \ar[r] & \Lambda^n R^n \ar[r] & \cdots \ar[r] & \Lambda^1 R^n \ar[r] & \Lambda^0 R^n \ar[r] & R/(x_1,\ldots,x_n) \ar[r] & 0
		\end{tikzcd}
	\]is exact. 
\end{prop}

We now prove Proposition~\ref{prop:equivariantCohomologyisomorphism}, which asserts that equivariant colored $\sl(N)$ homology is isomorphic to the $\U(N)$-equivariant cohomology of $\sr R_N$ for the trefoil and the Hopf link. Our proof gives the bigrading on $\KR_{\U(N)}$. 

\begin{proof}[Proof of Proposition~\ref{prop:equivariantCohomologyisomorphism}]
	Let $L$ be the positive Hopf link. Up to homotopy equivalence, its equivariant colored $\sl(N)$ complex $\KRC_{\U(N)}(L)$ is obtained from the complex given in Theorem~\ref{thm:slncomplexofHopfLink} by applying the state space functor $\sr F$ given in Definition~\ref{df:stateSpace}. This complex has no differential so \[
		\KR_{\U(N)}(L) \cong \KRC_{\U(N)}(L) \cong \bigoplus_{k=\max(a + b - N,0)}^{\min(a,b)} h^{2k}q^{ab-kN} \sr F\left( \Thetak \right).
	\]By Proposition~\ref{prop:HopfLinkRepSpace}, we have $\sr R_N(L) = \bigsqcup_{k = \max(a + b - N,0)}^{\min(a,b)} \F(k,a-k,b-k;N)$, and by Proposition~\ref{prop:generalizedThetaPartialFlag}, we have \[
		\KR_{\U(N)}(L) \cong \bigoplus_{k=\max(a + b - N,0)}^{\min(a,b)} h^{2k}q^{ab-kN-\dim \F(k,a-k,b-k;N)/2} H^*_{\U(N)}(\F(k,a-k,b-k;N))
	\]as $\Sym(N)$-modules. We note that $ab - kN - \dim \F(k,a-k,b-k;N)/2 = (a + b)(a + b - N - 2k) + 2k^2$.

	Now let $K$ be the right-handed trefoil labeled $a$. By Proposition~\ref{prop:TrefoilRepSpace}, we have \[
		\sr R_N(K) = \bigsqcup_{l = \max(2a-N,0)}^a \U(N)/K_l
	\]where $K_l = \U(l) \x \Delta\!\U(a-l) \x \U(N-2a+l)$. By Lemma~\ref{lem:equivariantCohomologyHomogeneousSpace}, we have \[
		H^*_{\U(N)}(\sr R_N(K)) \cong \bigoplus_{l=\max(2a-N,0)}^{a} H^*(\B K_l).
	\]Up to homotopy equivalence, the complex $\KRC_{\U(N)}(K)$ is obtained from the complex given in Theorem~\ref{thm:slncomplexTrefoil} by applying the state space functor $\sr F$. By Proposition~\ref{prop:generalizedThetaPartialFlag}, we have \[
		\KRC_{\U(N)}(K) \cong \bigoplus_{l=\max(2a-N,0)}^a h^{a + 2l}q^{-a-2l+a^2-lN - \dim \F(l,a-l,a-l;N)/2} K(y_1-z_1,\ldots,y_l-z_l)
	\]where $K(y_1-z_1,\ldots,y_l-z_l)$ denotes the Koszul complex associated to the sequence \[
		y_1-z_1,\ldots,y_l-z_l \in H^*_{\U(N)}(\F(l,a-l,a-l;N)) = \Sym(l) \otimes \Sym(a-l) \otimes \Sym(a-l) \otimes \Sym(N-2a+l)
	\]where we recall that $y_i = 1 \otimes e_i \otimes 1 \otimes 1$ and $z_i = 1 \otimes 1 \otimes e_i \otimes 1$. This sequence is clearly regular and \[
		\frac{H^*_{\U(N)}(\F(l,a-l,a-l;N))}{(y_1-z_1,\ldots,y_l-z_l)} \cong H^*(\B K_l)\cong H^*_{\U(N)}(\U(N)/K_l).
	\]By Proposition~\ref{prop:regularSeqKoszulResolution}, we have \[
		\KR_{\U(N)}(K) \cong \bigoplus_{l=\max(2a-N,0)}^a h^{a+2l}q^{-a-2l+a^2-lN - \dim \F(l,a-l,a-l;N)/2} H^*(\B K_l)
	\]where $K$ is equipped with the blackboard framing in its diagram in Theorem~\ref{thm:slncomplexTrefoil}. The invariant for $K$ with the zero framing is obtained by applying the grading shift $h^{-3a}q^{3a(N-a+1)}$. For the reader's convenience, we note that \[
		-a-2l+a^2-lN -\frac12\dim \F(l,a-l,a-l;N)+3a(N-a+1) = a(N-a)+2(a-l)(a-l+1).\qedhere
	\]
\end{proof}

Now we turn to nonequivariant colored $\sl(N)$ homology. Recall that the rank of a connected compact Lie group $G$ is the dimension of a maximal torus in $G$. If $G = \U(k_1) \x \cdots \x \U(k_n)$ then the rank of $G$ is $k_1 + \cdots + k_n$. The following result of Borel computes the cohomology of $G/H$ when $H$ has the same rank as $G$. See \cite[Theorem 8.3]{MR1793722}. 

\begin{thm}[{\cite[Proposition 30.2]{MR51508}}]\label{thm:borel}
	Let $H$ be a closed connected subgroup of a compact connected Lie group $G$. If $H$ and $G$ have the same rank, then $H^*(\B H)$ is a free module over $H^*(\B G)$. Furthermore \[
		H^*(G/H) \cong \Z \otimes_{H^*(\B G)} H^*(\B H)
	\]as $H^*(\B H)$-modules. In particular, the map $H^*(\B H) \to H^*(G/H)$ induced by $G/H \to \B H$ is surjective.
\end{thm}

\begin{proof}[Proof of Proposition~\ref{prop:bigradedHopfLink}]
	The representation space $\sr R_N(L)$ for the Hopf link $L$ is computed in Proposition~\ref{prop:HopfLinkRepSpace}. From the proof of Proposition~\ref{prop:equivariantCohomologyisomorphism}, we know that the equivariant colored $\sl(N)$ complex of $L$ is \[
		\KRC_{\U(N)}(L) \cong \bigoplus_{k=\max(a + b - N,0)}^{\min(a,b)} h^{2k}q^{(a + b)(a + b - N - 2k) + 2k^2} H^*_{\U(N)}(\F(k,a-k,b-k;N))
	\]The nonequivariant colored $\sl(N)$ complex of $L$ is obtained by tensoring with $\Z$ where the variables of $\Sym(N)$ are sent to zero. Thus \[
		\KR_N(L) \cong \KRC_N(L) \cong \bigoplus_{k=\max(a + b - N,0)}^{\min(a,b)} h^{2k}q^{(a + b)(a + b - N - 2k) + 2k^2} H^*(\F(k,a-k,b-k;N))
	\]by Borel's theorem (Theorem~\ref{thm:borel}). Here we are using Lemma~\ref{lem:equivariantCohomologyHomogeneousSpace} and the fact that \[
		\F(k,a-k,b-k;N) = \frac{\U(N)}{\U(k) \x \U(a-k) \x \U(b-k) \x \U(N-a-b+k)}
	\]where the rank of $\U(k) \x \U(a-k) \x \U(b-k) \x \U(N-a-b+k)$ is equal to $N$, the rank of $\U(N)$. 
\end{proof}
\begin{proof}[Proof of Proposition~\ref{prop:bigradedTrefoil}]
	The representation space $\sr R_N(K)$ for the trefoil $K$ is computed in Proposition~\ref{prop:TrefoilRepSpace}. For $\max(2a-N,0) \leq l \leq a$, let $K_l \subseteq H_l \subseteq \U(N)$ be the subgroups \begin{align*}
		K_l &= \U(l) \x \Delta\!\U(a-l) \x \U(N - 2a+l)\\
		H_l &= \U(l) \x \U(a-l) \x \U(a-l) \x \U(N-2a+l).
	\end{align*}
	From the proof of Proposition~\ref{prop:equivariantCohomologyisomorphism}, we know that the equivariant colored $\sl(N)$ complex of $K$ is \[
		\KRC_{\U(N)}(K) \cong \bigoplus_{l=\max(2a-N,0)}^a h^{a + 2l}q^{a(N-a) + 2(a-l)(a-l+1)} K(y_1-z_1,\ldots,y_l-z_l)
	\]where $K(y_1-z_1,\ldots,y_l-z_l)$ denotes the Koszul complex associated to the sequence \[
		y_1-z_1,\ldots,y_l-z_l \in H^*(\B H_l) = \Sym(l) \otimes \Sym(a-l) \otimes \Sym(a-l) \otimes \Sym(N-2a+l)
	\]where we recall that $y_i = 1 \otimes e_i \otimes 1 \otimes 1$ and $z_i = 1 \otimes 1 \otimes e_i \otimes 1$. By Proposition~\ref{prop:regularSeqKoszulResolution}, $K(y_1-z_1,\ldots,y_l-z_l)$ is a free resolution of $H^*(\B K_l)$ as a $H^*(\B H_l)$-module. By Borel's theorem (Theorem~\ref{thm:borel}), $H^*(\B H_l)$ is a free $H^*(\BU(N))$-module so $K(y_1-z_1,\ldots,y_l-z_l)$ is a free resolution of $H^*(\B K_l)$ as a $H^*(\BU(N))$-module. Thus, the homology of the tensor product \[
		\Z \otimes_{\BU(N)} K(y_1-z_1,\ldots,y_l-z_l)
	\]is by definition $\Tor_{H^*(\BU(N))}(\Z,H^*(\B K_l))$ so \[
		\KR_N(K) \cong \bigoplus_{l=\max(2a-N,0)}^a h^{a + 2l}q^{a(N-a) + 2(a-l)(a-l+1)} \Tor_{H^*(\BU(N))}(\Z,H^*(\B K_l)).\qedhere
	\]
\end{proof}
\begin{proof}[Proof of Theorem~\ref{thm:plainisomorphismforHopfLinkAndTrefoil}]
	The isomorphism for the Hopf link follows directly from Proposition~\ref{prop:bigradedHopfLink}. The isomorphism for the trefoil follows from Proposition~\ref{prop:bigradedTrefoil} and Gugenheim--May's theorem (Theorem~\ref{thm:gugenheimMay}).
\end{proof}

Lastly, we turn to module structures and reduced homology. We recall that the tautological bundle over $\G(a,N)$ is the rank $a$ bundle whose fiber over $\Lambda \in \G(a,N)$ is the vector space $\Lambda$. The Chern classes of this tautological bundle generate $H^*(\G(a,N))$ as a ring. The fundamental class of $\G(a,N)$ is the $(N-a)$th power of the $a$-th Chern class of this bundle. 

\begin{proof}[Proof of Proposition~\ref{prop:moduleandReducedIsomorphism} for the Hopf link]
	Consider the Hopf link $L$ together with the basepoint \[
		\:\begin{gathered}
		\vspace{-3pt}
		\centering
		\labellist
		\pinlabel {\small$a$} at 8 19
		\pinlabel {\large$\bullet$} at 7 6
		\pinlabel {\small$b$} at 52 20
		\endlabellist
		\includegraphics[width=.11\textwidth]{hopfLink}
	\end{gathered}\:
	\]Consider the dot map $X_i$ at the given basepoint associated to the elementary symmetric polynomial $e_i$. Let $C$ be the complex associated to the given diagram of the Hopf link, let $\ol{C}$ be the complex on the right-hand side of Theorem~\ref{thm:slncomplexofHopfLink}, and let $\pi\colon C \to \ol{C}$, $\iota\colon \ol{C} \to C$, $h\colon C \to C$ be the strong deformation retract given in the proof. Recall that $C$ and $\ol{C}$ split over the poset $P = \{\:k\:|\: \max(a + b - N,0) \leq k \leq \min(a,b)\:\}$ where \[
		\ol{C}_k = h^{2k} q^{ab - k(N+1)} \left\llbracket \Thetak\right\rrbracket
	\]Consider the map $\pi\circ X_i\circ \iota\colon \ol{C} \to \ol{C}$. By homological degree considerations, the component of the map from $\ol{C}_k \to \ol{C}_j$ is zero whenever $k \neq j$, so the only possibly nonzero components are $\ol{C}_k \to \ol{C}_k$ which are given by $\pi_k\circ X_i \circ \iota_k$. By Proposition~\ref{prop:dotSliding}, we know that $\pi_k\circ X_i\circ\iota_k$ is homotopic to, and therefore equal to, the dot map \[
		\:\begin{gathered}
		\labellist
		\pinlabel {\large$\bullet$} at 25 50
		\pinlabel {\small$e_i$} at 20 38
		\endlabellist
		\includegraphics[width=.08\textwidth]{smallerTheta}
	\end{gathered}\:
	\]
	The isomorphism of Proposition~\ref{prop:generalizedThetaPartialFlag} intertwines this dot map with the endomorphism of $H^*_{\U(N)}(\F(k,a-k,b-k;N))$ given by $\sum_{j+k=i} x_jy_k$. By Borel's theorem (Theorem~\ref{thm:borel}), we have \[
		H^*(\F(k,a-k,b-k;N)) \cong \Z \otimes_{H^*(\BU(N))} H^*_{\U(N)}(\F(k,a-k,b-k;N)). 
	\]Consider the vector bundle $\sr A \to \F(k,a-k,b-k;N)$ of rank $a$ whose fiber over the triple $(\Lambda_1,\Lambda_2,\Lambda_3)$ is the vector space $\Lambda_1 \oplus \Lambda_2 \subseteq \C^N$.
	Under the isomorphism, the element $1 \otimes \sum_{j+k=i} x_jy_k$ on the right-hand side is mapped to the $i$th Chern class of $\sr A$ on the left-hand side. 

	The restriction of the map $\sr R_N(L) \to \G(a,N)$ to $\F(k,a-k,b-k;N) \to \G(a,N)$ is just the map that sends $(\Lambda_1,\Lambda_2,\Lambda_3)$ to $\Lambda_1 \oplus\Lambda_2 \in \G(a,N)$. By naturality of Chern classes, the $i$th Chern class of the tautological plane bundle over $\G(a,N)$ is sent to the $i$th Chern class of $\sr A$ under the map $H^*(\G(a,N)) \to H^*(\F(k,a-k,b-k;N))$. Thus, the actions of the Chern classes of the tautological bundle over $\G(a,N)$ on $H^*(\sr R_N(L))$ match the actions on the dot maps associated to elementary symmetric polynomials on $\KR_N(L)$. The two actions of $H^*(\G(a,N))$ therefore agree.

	As for reduced homology, we first claim that there is a chain homotopy equivalence \[
		\ol{\KRC}_N \left( \:\:\begin{gathered}
			\vspace{-3pt}
			\centering
			\labellist
			\pinlabel {\small$a$} at 29 83
			\pinlabel {\small$b$} at 74 84
			\pinlabel {\small${b-k}$} at 51 84
			\pinlabel {\small$k$} at 75 55
			\pinlabel {\small${a-k}$} at 51 54
			\pinlabel {\small$a$} at 71 26
			\pinlabel {\small$b$} at 31 27
			\pinlabel {\large$\bullet$} at 12 12 
			\endlabellist
			\includegraphics[width=.17\textwidth]{twistClosureWk}
		\end{gathered}\:\: \right) \:\:\simeq\:\: h^kq^{ab - k(N+1)} \ol{\KRC}_N \left( \qquad\quad\:\:\begin{gathered}
			\vspace{-3pt}
			\centering			
			\labellist
			\pinlabel {\small${b-k}$} at 30 80
			\pinlabel {\small$k$} at 59 44
			\pinlabel {\small${a+b-k}$} at -22 46
			\pinlabel {\small$a$} at 41 45
			\pinlabel {\large$\bullet$} at 20 55
			\pinlabel {\small${a - k}$} at 35 18
			\pinlabel {\small$b$} at 70 5
			\endlabellist
			\includegraphics[width=.13\textwidth]{Theta}
		\end{gathered}\:\: \right)
	\]where these reduced complexes are defined with respect to the indicated basepoints. The equivalence is obtained by following the proof of Lemma~\ref{lem:flatteningTwistClosureWk}. At one point in the argument, the basepoint must be moved across a crossing. For this, we use the argument in \cite[Section 3]{MR2034399} where instead of moving the basepoint, we move the other strand the ``long way'' around. By the above analysis, the reduced complex on the right-hand side, which has no differential, is the image of the action of the fundamental class of $\G(a,N)$ on $H^*(\F(k,a-k,b-k;N))$, which we denote by $[\G(a,N)]\cdot H^*(\F(k,a-k,b-k;N))$. 
	The homological perturbation lemma (Lemma~\ref{lem:homologicalPerturbationLemma}) now implies that \[
		\ol{\KR}_N(L) \cong q^{-a(N-a)}\bigoplus_{k=\max(a+b-N,0)}^{\min(a,b)} h^{2k} q^{(a + b)(a + b - N - 2k) + 2k^2} [\G(a,N)] \cdot H^*(\F(k,a-k,b-k;N))
	\]For $\max(a + b -N,0) \leq k \leq \min(a,b)$, let $H_k \subseteq J \subseteq G = \U(N)$ be the subgroups \begin{align*}
		H_k = \U(k) \x \U(a -k) \x \U(b-k) \x \U(N-a-b+k) \qquad J = \U(a) \x \U(N-a).
	\end{align*}The fiber of the bundle $G/H_k = \F(k,a-k,b-k;N) \to \G(a,N) = G/J$ is $J/H_k$ so $\ol{\sr R}_N(L) = \bigsqcup_k J/H_k$. It follows from Borel's theorem (Theorem~\ref{thm:borel}) that $H^*(G/H_k) \to H^*(J/H_k)$ is surjective so by Leray--Hirsch, we have that $H^*(G/H_k) \cong H^*(J/H_K) \otimes H^*(G/J)$ as $H^*(G/J)$-modules. In particular, we have $[G/J]\cdot H^*(G/H_k) \cong H^*(J/H_k)$ so $\ol{\KR}_N(L) \cong H^*(\ol{\sr R}_N(L))$. 
\end{proof}

\begin{proof}[Proof of Proposition~\ref{prop:moduleandReducedIsomorphism} for the trefoil]
	Let $K$ be the trefoil with the basepoint \[
		\:\:\begin{gathered}
			\vspace{-3pt}
			\centering
			\labellist
			\pinlabel {\small$a$} at 57 50
			\pinlabel {\large$\bullet$} at 6 6
			\endlabellist
			\includegraphics[width=.13\textwidth]{trefoil}
		\end{gathered}\:\:
	\]By the proof of Proposition~\ref{prop:bigradedTrefoil}, we know that \[
		\KR_N(K) \cong \bigoplus_{l=\max(2a-N,0)}^a h^{a + 2l} q^{a(N-a)+2(a-l)(a-l+1)} \Tor_{H^*(\BU(N))}(\Z,H^*(\B K_l))
	\]where $K_l = \U(l) \x \Delta\!\U(a-l) \x \U(N-2a+l) \subseteq G = \U(N)$. 
	By an argument similar to the one for the Hopf link above, the dot map at the given basepoint associated to $e_i$ is given on the $l$th direct summand by the following endomorphism of $\Tor_{H^*(\B G)}(\Z,H^*(\B K_l))$. Set $H_l = \U(l) \x \U(a-l) \x \U(a-l) \x \U(N-2a+l)$ and $J = \U(a) \x \U(N-a)$. The natural action of $H^*(\B K_l)$ on $\Tor_{H^*(\B G)}(\Z,H^*(\B K_l))$ induces an action of $H^*(\B J)$ by the composite map $H^*(\B J) \to H^*(\B H_l) \to H^*(\B K_l)$. The endomorphism is the action of $e_i \otimes 1 \in H^*(\B J) = \Sym(a) \otimes \Sym(N-a)$. By Proposition~\ref{prop:moduleStructGugenheimMay}, this action agrees with the corresponding action of $H^*(\B J)$ on $H^*(G/K_l)$ arising from the composite map $G/K_l \to \B K_l \to \B H_l \to \B J$. Now consider the commutative diagrams \[
		\begin{tikzcd}
			G/K_l \ar[d] \ar[r] & \B K_l \ar[d]\\
			G/J \ar[r] & \B J
		\end{tikzcd} \qquad\qquad \begin{tikzcd}
			H^*(G/K_l) & H^*(\B K_l) \ar[l]\\
			H^*(G/J) \ar[u] & H^*(\B J) \ar[u] \ar[l]
		\end{tikzcd}
	\]It follows that the action of $e_i \otimes 1 \in H^*(\B J)$ on $H^*(G/K_l)$ agrees with the action of the $i$th Chern class of the tautological bundle of $\G(a,N) = G/J$. Thus, the actions of $H^*(\G(a,N))$ on $\KR_N(K)$ and $\sr R_N(K)$ agree.

	We now turn to reduced homology. Note that $\ol{\sr R}_N(K) = \bigsqcup_l J/K_l$ since the fiber of the bundle $G/K_l \to G/J$ is $J/K_l$. By the argument given for the reduced homology of the Hopf link, the reduced complex of the trefoil $K$ is \begin{align*}
		\ol{\KRC}_N(K) &\cong q^{-a(N-a)}\bigoplus_{l = \max(2a-N,0)}^a h^{a + 2l}q^{a(N-a) + 2(a-l)(a-l+1)} [G/J]\cdot (\Z \otimes_{H^*(\B G)} K(y_1 - z_1,\ldots,y_l-z_l))
	\end{align*}where $K(y_1-z_1,\ldots,y_l-z_l)$ is the Koszul complex associated to $y_1-z_1,\ldots,y_l-z_l \in H^*(\B H_l)$. By definition, $K(y_1-z_1,\ldots,y_l-z_l)$ is a complex of free $H^*(\B H_l)$-modules, so $\Z \otimes_{H^*(\B G)} K(y_1-z_1,\ldots,y_l-z_l)$ is a complex of free $H^*(G/H_l)$-modules by Borel's theorem. Since $H^*(G/H_l)$ is a free $H^*(G/J)$-module by Leray--Hirsch, we see that $\Z \otimes_{H^*(\B G)} K(y_1-z_1,\ldots,y_l-z_l)$ is a complex of free $H^*(G/J)$-modules where the differential commutes with the action of $H^*(G/J)$. It follows that there is an isomorphism of complexes\[
		q^{-a(N-a)}[G/J]\cdot (\Z \otimes_{H^*(\B G)} K(y_1-z_1,\ldots,y_l-z_l)) \cong q^{a(N-a)}\Z \otimes_{H^*(G/J)} (\Z \otimes_{H^*(\B G)} K(y_1-z_1,\ldots,y_l-z_l))
	\]where the latter complex is the quotient complex where all positive degree terms of $H^*(G/J)$ are set to be zero. Next, we observe that \begin{align*}
		\Z \otimes_{H^*(G/J)} (\Z \otimes_{H^*(\B G)} K(y_1 - z_1,\ldots,y_l-z_l)) 
		&\cong \Z \otimes_{H^*(\B J)} K(y_1-z_1,\ldots,y_l-z_l).
	\end{align*}Hence \[
		\ol{\KR}_N(K) \cong \bigoplus_{l = \max(2a-N,0)}^a h^{a + 2l}q^{2a(N-a)+2(a-l)(a-l+1)} \Tor_{H^*(\B J)}(\Z,H^*(\B K_l))
	\]so again by Gugenheim--May's theorem (Theorem~\ref{thm:gugenheimMay}) we have $\ol{\KR}_N(K) \cong H^*(\ol{\sr R}_N(K))$.
\end{proof}

\phantomsection
\addcontentsline{toc}{section}{\protect\numberline{}References}
\raggedright
\bibliography{slNsuN}
\bibliographystyle{alpha}

\end{document}